\documentclass{amsart}

\usepackage{amsmath,amssymb,amsfonts,amsthm}
\usepackage[bbgreekl]{mathbbol}
\usepackage{mathrsfs}
\usepackage{dsfont}
\usepackage{stmaryrd}
\usepackage{physics}
\usepackage{faktor}
\usepackage{aliascnt}

\numberwithin{equation}{section}

\theoremstyle{definition}

\usepackage{graphicx}
\usepackage{array,multirow,booktabs,makecell}
\usepackage{longtable}
\usepackage{blkarray,bigstrut}
\usepackage{float}
\usepackage{geometry}

\usepackage{pdflscape}
\usepackage{rotating}

\usepackage[usenames,dvipsnames,table]{xcolor}

\usepackage{tikz}
\usetikzlibrary{
    matrix,
    arrows,
    arrows.meta,
    cd,
    calc,
    positioning,
    automata,
    backgrounds,
    shapes.multipart,
    decorations.pathmorphing,
    decorations.pathreplacing,
    decorations.markings
}

\usepackage{pgf}
\usepackage{pgfplots}
\usepackage[all]{xy}

\usepackage{nicematrix}
\usepackage{ytableau}
\usepackage{genyoungtabtikz}

\usepackage[backref=page]{hyperref}

\hypersetup{
    colorlinks=false,
    citebordercolor=Cerulean,
    urlbordercolor=ForestGreen,
    linkbordercolor=Fuchsia
}

\usepackage[capitalize]{cleveref}

\Crefname{equation}{}{}

\usepackage{paralist}
\usepackage{dirtytalk}
\usepackage{url}

\renewcommand{\arraystretch}{1.2}

\usepackage{shuffle}
\usepackage{tcolorbox}
\usepackage[new]{old-arrows}

\usepackage{frcursive}

\newcommand{\rk}{\operatorname{rk}}

\usepackage{mynewcommands}

\begin{document}

\sloppy

\title{Enumerating cores of charged multipartitions}

\author{Thomas Gerber}
\address[T.G.]{Université Claude Bernard Lyon 1}
\email{gerber@math.univ-lyon1.fr}
\author{Emily Norton}
\address[E.N.]{University of Kent}
\email{e.norton@kent.ac.uk}

%%%%%%%%%%%%%%%%%%%%%%%%%%%%%%%%%%%%%%%%%%%%%%%%%%
%%%%%%%%%%%%%%%%%%%%%%%%%%%%%%%%%%%%%%%%%%%%%%%%%%
%%%%%%%%%%%%%%%%%%%%%%%%%%%%%%%%%%%%%%%%%%%%%%%%%%
%%%%%%%%%%%%%%%%%%%%%%%%%%%%%%%%%%%%%%%%%%%%%%%%%%
%%%%%%%%%%%%%%%%%%%%%%%%%%%%%%%%%%%%%%%%%%%%%%%%%%
%%%%%%%%%%%%%%%%%%%%%%%%%%%%%%%%%%%%%%%%%%%%%%%%%%
%%%%%%%%%%%%%%%%%%%%%%%%%%%%%%%%%%%%%%%%%%%%%%%%%%

\begin{abstract}
Granville and Ono proved that there is an $e$-core partition of $n$ for every $n\in\mathbb{N}$ if and only if $e\geq 4$, which translates to a statement about the existence of defect $0$ blocks for symmetric groups in positive characteristic and defect $0$ unipotent blocks of finite general linear groups in positive, non-defining characteristic \cite{Ono1994,GO1996}. Motivated by analogous applications to the block theory of finite classical groups, imprimitive spetses, and cyclotomic Hecke algebras, we prove similar positivity statements about different variants of $e$-cores for charged multipartitions.
\end{abstract}

\maketitle

\begingroup
  \hypersetup{hidelinks}
\setcounter{tocdepth}{2} 
  \tableofcontents
\endgroup
\section*{Introduction}

A $t$-core partition is a partition that does not contain a removable rim-hook of length $t$. The ``$t$-core conjecture'' states that there is a $t$-core partition of $n$ for every $n\in\N$ if and only if $t\geq 4$, and was proved by Granville and Ono \cite[Theorem 1]{GO1996}. This has a consequence for representations of symmetric (and alternating) groups in characteristic $p>0$: as two ordinary irreducible characters labeled by partitions of $n$ belong to the same $p$-block of $S_n$ if and only if the partitions have the same $p$-core \cite{Brauer1947,Robinson1947}, it follows that every symmetric group $S_n$, $n\in\N$, has a defect $0$ $p$-block if and only if $p\geq 5$ \cite[Theorem 2]{GO1996}. The goal of this paper is to state and prove meaningful generalizations of the $t$-core conjecture in higher level, that is, replacing partitions with appropriate sets of charged multipartitions suggested by representation theory. These generalizations of the $t$-core conjecture imply analogous consequences for the defect $0$ unipotent blocks of finite classical groups, the $\Phi$-Harish-Chandra series of imprimitive spetses that are singletons, and the defect $0$ blocks of cyclotomic Hecke algebras.

\medskip

\textbf{Enumeration and main results.} The gold standard method to deal with a statement like the $t$-core conjecture is to work with the generating function of the type of object being enumerated, a gadget that considers all $n$ at once. 
In order to study positivity questions around higher-level analogues of core partitions, our first step is thus to find the generating functions of the charged multipartitions playing the role of cores. These are: 
\begin{itemize}
\item cocores (level $2$, in the context of finite classical groups, introduced by Olsson \cite{Olsson1986}), 
\item spetsial $d$-cores (arbitrary level, in the context of imprimitive spetses, introduced by Malle under the name $(d,\zeta)$-cores \cite{Malle1995}),
\item $(e,\bs)$-cores (arbitrary level, in the context of cyclotomic Hecke algebras, introduced by Jacon and Lecouvey \cite{JaconLecouvey2020}).
\end{itemize}
We study their combinatorics in \Cref{sec_higher_level_cores}. In this context, it is really important to replace the size of a (charged) multipartition with its rank, defined by Malle \cite[§3.5]{Malle1995}. The rank depends on both the size of the multipartition and the charge. We study the rank statistic in \Cref{sec_rank,rk_cusp}. The formulas for the generating functions of higher-level cores hinge on Malle's rank-preserving analogue of the Littlewood decomposition for type $G(\ell,1,n)$-symbols \cite[Corollary 3.5]{Malle1995}, which we establish in \Cref{rk_vs_lit} without restrictions.
\medskip

It turns out that the generating functions of the cocores and spetsial $d$-cores factorize nicely (\Cref{thm_prod_cocores} and \Cref{gf_spetsialdcores}), allowing us to establish the analogue of the $t$-core conjecture in these settings by standing on the shoulders of Granville and Ono and dealing with the cases of $d=1,2,3$ by hand. This is \Cref{thm_positivity_cocores} and \Cref{thm_positivity_spetsialdcores}. On the other hand, pinning down the generating functions of the $(e,\bs)$-cores for an individual $\bs$ is much more difficult, and we only managed to find their generating functions for small $e$ and $\ell$. The trick here is to realize that when $e=d\ell$, the generating functions of spetsial $d$-cores are linear combinations of generating functions of $(e,\bs)$-cores multiplied by some powers of $q$ (\Cref{gf_cores_decomp}). This allows us to compute the $(e,\bs)$-core generating functions for small values of $e$ and $\ell$ and determine positivity of their coefficients in \Cref{thm_gf_escores}.

\medskip

We summarize the various enumeration results of this paper in \Cref{table_gf}.
The different generating functions are indeed always of the form
$$\sff(q) = \sum_{n\in\N} f(n)q^n = \sum_{\la\in\cS} q^{\Vert\la\Vert} $$
where the notations for the series $\sff$, the set $\cS$, the statistic $\Vert.\Vert$ and the coefficients $f$ are also detailed in the table. 
The first three rows of \Cref{table_gf} contain well-known formulas: the partition generating function is due to Euler, and the product formula for the generating function of $e$-core partitions to Garvan, Kim, and Stanton \cite{GKS1990} (who used Littlewood's decomposition \cite{Lit1951}) 
and Granville-Ono's positivity theorem \cite[Theorem 1]{GO1996}.
All further rows contain new results on the generating functions of higher-level cores and positivity of their coefficients.

\begin{table}[h!]
\centering
\renewcommand{\arraystretch}{1.8}
\begin{math}
\begin{array}{@{}l@{\hskip 10pt}l@{\hskip 10pt}l@{\hskip 10pt}l@{\hskip 10pt}l@{\hskip 10pt}l@{\hskip 10pt}l}
\\
\hline
\text{Objects}
&
\text{$\cS$}
&
\text{Statistic $\Vert.\Vert$}
&
\text{$\sff(q)$}
&
\text{$f(n)$}
& \text{Formula}
& \text{Positivity}
\\
\hline
\text{partitions $\la$} & \cP & |\la| & \sfp(q)  & p(n) & (q;q)_\infty^{-1}&\text{yes}
\\
\text{$\ell$-partitions $\bla$} & \cP^\ell & |\bla| & p(q)^\ell  & p(n)^\ell &(q;q)_\infty^{-\ell}&\text{yes}
\\
{\renewcommand{\arraystretch}{1}
\begin{array}[t]{@{}l}
\text{$e$-core} \\
\text{partitions $\la$}
\end{array}}
& \cC_e & |\la| & \sfc_e(q)  & c_e(n) &(q^e;q^e)_\infty^{e}(q;q)_\infty^{-1}
& \text{yes iff }e\geq 4
\\
{\renewcommand{\arraystretch}{1}
\begin{array}[t]{@{}l}
\text{$\ell$-charges $\bsig$}\\
\text{(cusp. symbols)}
\end{array}}
&\Z^\ell[\si] & \rk(\bsig) & \mathsf{r}^\ell_\si(q)  & r^\ell_\si(n) &
{\renewcommand{\arraystretch}{1}
\begin{array}[t]{@{}l}
\Theta_{A_{\ell-1}}(\frac{z}{2}) \text{ if } \si=0
\\
\text{(\Cref{rem_cusp_theta})}
\\
\text{see \Cref{charge_gf_si=1}}  \text{ if } \si=1
\end{array}}
&{\renewcommand{\arraystretch}{1}
\begin{array}[t]{@{}l}
\text{yes iff }\ell\geq 5 \text{ if }\si=0
\\
\text{yes iff }\ell\geq 4 \text{ if }\si=1
\\
\text{(\Cref{thm_positivity_cuspidals})}
\end{array}}
\\
{\renewcommand{\arraystretch}{1}
\begin{array}[t]{@{}l}
\text{charged}
\\
\text{$\ell$-partitions}
\\ 
|\bla,\bsig\rangle, \bsig\in\Z^\ell[\si] 
\end{array}}
& \cU_\si^\ell  & \rk(|\bla,\bsig\rangle) & \sfu^\ell_\si(q)  & u^\ell_\si(n) 
&{\renewcommand{\arraystretch}{1}
\begin{array}[t]{@{}l}
\mathsf{r}^\ell_\si(q)  \sfp(q)^\ell
\\
\text{(\Cref{symbolgf_dec})}
\end{array}}
&
\text{yes if $\si\equiv0,\pm1 \bmod \ell$}
\\
{\renewcommand{\arraystretch}{1}
\begin{array}[t]{@{}l}
\text{$d$-cocores} 
\\
|\bla,\bsig\rangle, \bsig\in\Z^\ell[\si] 
\end{array}}
& \cC_{[d,\si]}^2 &  \rk(|\bla,\bsig\rangle) & \sfz^2_{d,\si}(q) & z^2_{d,\si}(n) & 
{\renewcommand{\arraystretch}{1}
\begin{array}[t]{@{}l}
\phi(q)\sfc_d(q)^2\text{ if }\si=0
\\
2\psi(q)\sfc_d(q)^2\text{ if }\si=1
\\
\text{(\Cref{thm_prod_cocores})}
\end{array}}
& {\renewcommand{\arraystretch}{1}
\begin{array}[t]{@{}l}
\text{yes iff } d\geq 2,
\\ 
\text{(\Cref{thm_positivity_cocores})}
\end{array}}
\\
{\renewcommand{\arraystretch}{1}
\begin{array}[t]{@{}l}
\text{spetsial}
\\
\text{$d$-cores } 
\\
|\bla,\bsig\rangle, \bsig\in\Z^\ell[\si]
\end{array}}
& \cC^\ell_{[d,\si]} & \rk(|\bla,\bsig\rangle) & \sfz^\ell_{d,\si}(q)  & z^\ell_{d,\si}(n) & 
{\renewcommand{\arraystretch}{1}
\begin{array}[t]{@{}l}
 \mathsf{r}^\ell_\si(q)\sfc_d(q)^\ell
\\
\text{(\Cref{gf_spetsialdcores})}
\end{array}}
& 
{\renewcommand{\arraystretch}{1}
\begin{array}[t]{@{}l}
\text{if }\ell=2,\text{ see cocores}
\\
\text{if }\ell\geq 3:
\\
\si=0:\text{yes iff }\ell d\geq 5
\\
\si=1:\text{yes iff }\ell d\geq 4,
\\
\text{(\Cref{thm_positivity_spetsialdcores})}
\end{array}}
\\
{\renewcommand{\arraystretch}{1}
\begin{array}[t]{@{}l}
\text{$e$-core}
\\
\text{$\ell$-partitions} 
\\
\text{with charge $\bs$}
\end{array}}
& \cC_{e,\bs}^\ell & |\bla| & \sfc_{e,\bs}(q)  & c_{e,\bs}(n)
& \text{see \Cref{thm_gf_escores}}
&\text{see \Cref{thm_gf_escores}}
\\
\hline
\end{array}
\end{math}
\caption{Notation and results on the various generating functions studied}
\label{table_gf}
\end{table}

\medskip

\textbf{Applications in representation theory.} Our results have applications to modular representation theory of finite groups of Lie type, $\Phi$-Harish-Chandra theory of spetses, and block theory for cyclotomic Hecke algebras.

\medskip

Let $v$ be a power of an odd prime $p$, let $l$ be an odd prime different from $p$, and let $e$ be the order of $v$ mod $l$. First of all, positivity of the coefficients of the ``$\si=1$'' cocore generating function $\sfz^2_{d,1}({{q}})$ implies that $\mathrm{SO}_{2n+1}(v)$ and $\mathrm{Sp}_{2n}(v)$ both have a defect $0$ unipotent 
$l$-block if $e=2d$. This is a straightforward corollary for these finite classical groups of Dynkin types $B$ and $C$. On the other hand, the situation is more delicate in types $D$ and $^2D$, as we bundle the two types of orthogonal groups together in the generating functions we study and they cannot be easily pulled apart. Positivity of the coefficients of the ``$\si=0$'' cocore generating function $\sfz^2_{d,0}({{q}})$ implies that at least one of $\mathrm{O}_{2n}^+(v),\mathrm{O}_{2n}^-(v)$ has a defect $0$ unipotent 
$l$-block if $e=2d$. Accordingly, when $e$ is even (the unitary prime case), $\mathrm{SO}_{2n+1}(v),\;\mathrm{Sp}_{2n}(v),$ and at least one of $\mathrm{O}_{2n}^\pm(v)$ have a defect $0$ unipotent $l$-block if and only if $e\geq 4$. When $e$ is odd (the linear prime case), $\mathrm{SO}_{2n+1}(v)$, $\mathrm{Sp}_{2n}(v),$ and $\mathrm{O}_{2n}^\pm(v)$ all have a defect $0$ unipotent $l$-block if and only if $e\geq 3$. This is \Cref{cor_fglt}.

\medskip

Secondly, the spetsial $d$-cores are related, as the name suggests, to the theory of spetses.
This mysterious program initiated by Brou\'{e}, Malle, and Michel seeks to reverse-engineer a theory of exotic analogues of finite reductive groups where Weyl group data is replaced by complex reflection group data \cite{BMM1999,Malle1998}. 
While this paper is not directly about the spetses program, our results have consequences for imprimitive spetses. In our study of different generalizations of cores to higher level, we show that the combinatorial framework of $\ell$-symbols underpinning the ``unipotent characters" of imprimitive spetses (that is, for the complex reflection groups $G(\ell,m,n)$ where $m$ divides $\ell$) and their ``block theory" has a direct link to the combinatorics of higher-level Fock spaces: in \Cref{spetsial combinat} we establish a precise relationship with the level-rank duality, which is the combinatorial shadow of a duality functor for affine Lie algebras. Note that a connection between $\Phi$-Harish-Chandra theory and level-rank duality was recently found in type $A$ \cite{TrinhXue2023}. 

\medskip

In \Cref{sec_rank}, we recast the rank statistic of an $\ell$-symbol -- which, if it were actually labeling a representation of something, would tell the rank of the group whose representation it labels -- in terms of charged $\ell$-partitions, and this shift in perspective reveals that the generating functions of ``cuspidal symbols" are related to sums of squares and are often Theta series of lattices (see e.g. \Cref{rksumsquares,sigma=2 gf}). For instance, when $\ell=4$, the Theta series of the famous sphere-packing lattice in $3$-dimensional Euclidean space (the fcc lattice) coincides with the generating function of cuspidal $4$-symbols whose content is congruent to $0$ mod $4$ (after substituting $q^2$ for $q$), see \Cref{charge gf level 4}. Our result on ``$\si=1$'' spetsial $d$-cores implies that when $e$ is divisible by $\ell$, the $G(\ell,1,n)$-spets has a $\Phi_e$-block consisting of a single symbol for every $n\in\N$ if and only if $e\geq 4$.
This is \Cref{spetsial_results}.

\medskip

Thirdly, the $(e,\bs)$-cores control block theory for cyclotomic Hecke algebras, also known as Ariki-Koike algebras.
These deformations of the group algebra of $G(\ell,1,n)$ have attracted extensive attention in the literature, studied from representation-theoretic, combinatorial, geometric, and topological viewpoints.
Being symmetric algebras, there is a notion of defect that measures the complexity of the associated blocks, and we are interested again in the defect $0$ blocks.
More precisely, the data of $(e,\bs)\in\Z_{\geq 2}\times \Z^\ell$ give rise to a non-semisimple cyclotomic Hecke algebra, and it was proved in \cite[Corollary 4.4]{JaconLecouvey2021} that
two multipartitions of the same size are in the same block for this algebra if and only if 
they have the same $(e,\bs)$-core.
In particular, $\ell$-partitions that are their own $(e,\bs)$-core label the defect $0$ blocks.
Variations on higher-level cores in this context can be found in
\cite{Fayers2007, Fayers2019, JaconLecouvey2020, JaconLecouvey2021, ChlouverakiJacon2023, Lyle2024, CGJ2025}.
Therefore, our positivity results on the $(e,\bs)$-cores generating function
established in \Cref{escore_gf_sec} imply the existence of defect $0$ blocks
for the associated cyclotomic Hecke algebras, see \Cref{sec_def0_Hecke}.
Note that finding a general formula for the generating function  of the $(e,\bs)$-cores
appears to be a particularly challenging problem.

\section{Core partitions, generating functions, and positivity}
\label{sec_part}

\subsection{Partitions and multipartitions: notation and basic definitions}
We start by setting some conventions and notation and recalling some background.
Let us follow the French convention in denoting by $\N$ the set of natural numbers together with $0$:
\[
\N=\{0,1,2,3,4,\ldots\}.
\]

A {\em partition} $\lambda$ is a (possibly infinite) sequence $\lambda=(\lambda_1,\lambda_2,\la_3,\ldots )$, $\la_i\in\N$ for all $i\geq 1$, such that $\la_i>0$ for only finitely many $i$ and $\la_1\geq\la_2\geq\la_3\geq\ldots$ We further say {\em $\la$ is a partition of $n$} if $\sum_{i\geq 1}\la_i=n$, in which case we write $|\la|=n$ and call $n$ the \textit{size} of $\la$. Partitions are the same if their non-zero parts are the same. The unique partition of $0$ is the empty partition, denoted $\emptyset$. 
We call the $\la_i$'s the {\em parts} of $\la$. We may depict $\la$ using its {\em Young diagram}, 
an array of boxes with $\la_1$ boxes in its first row, $\la_2$ boxes in its second row, etc, as illustrated below (in English notation) for $\la=(4,4,2)$, a partition of $10$:
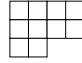
\[
\ds
\Yboxdim{7pt}
\young(<><><><>,<><><><>,<><>)
\]
We abbreviate repeated non-zero parts of $\la$ using exponential notation, e.g. $(5,5,5,2,2,2,2,1)=(5^3,2^4,1)$.

\medskip

For $\ell\in\Z_{\geq 1}$, an {\em $\ell$-partition} is an $\ell$-tuple $\bla=(\la^1,\la^2,\ldots,\la^\ell)$ where each $\la^j$ is a partition for $j=1,\ldots,\ell$. We refer to the $\la^j$'s as the {\em component partitions} (or just {\em components}) of $\bla$. A {\em multipartition} refers to an $\ell$-partition for some $\ell\geq 2$ (not specified). A $2$-partition is usually called a {\em bipartition.} The Young diagram of $\bla$ is the $\ell$-tuple of Young diagrams of its component partitions. For example, the Young diagram of the $3$-partition $((3,2,1),(4,1),(1^5))$ is 
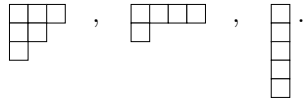
\[
\ds
\Yboxdim{7pt}
\young(<><><>,<><>,<>) \quad , \quad \young(<><><><>,<>)  \quad,\quad \young(<>,<>,<>,<>,<>)\;.
\]
The {\em size} of an $\ell$-partition is the sum of the sizes of its component partitions, i.e. $|\bla|=|\la^1|+\ldots+|\la^\ell|$, which is the number of boxes in the Young diagram of $\bla$. For instance, $((3,2,1),(4,1),(1^5))$ has size $16$. The unique $\ell$-partition of $0$ is the empty $\ell$-partition, which we denote $\bemp$. The integer $\ell$ will often be referred to as the {\em level}; this comes from the appearance of (charged) $\ell$-partitions in the domains of affine Lie algebras and mathematical physics. 

\medskip

For $n\in\N$, denote by $\cP(n)$ the set of partitions of $n$, and by $\cP=\bigcup_{n\in\N}  \cP(n)$ the set of all partitions.
We denote by $\cP^\ell=\{\bla=(\la^1,\la^2,\ldots,\la^\ell)\;\mid\; \la^j\in\cP\}$ the set of all $\ell$-partitions.
Consider the formal power series
$$\sfp(q)  = \sum_{\la\in\cP} q^{|\la|} =\sum_{n\geq 0} p(n)q^n,$$
where $p(n)=|\cP(n)|$. That is, the coefficient of $q^n$ in $\sfp(q)$ is the number of partitions of $n$. The formal power series $\sfp(q)$ is the {\em generating function of partitions}. Considering all partitions of all possible sizes at once and enumerating them using a generating function is a simple but powerful concept, largely thanks to Euler's product formula for $\sfp(q)$:
\begin{equation}
\label{gf_part}
\sfp(q)= \prod\limits_{n=1}^\infty \frac{1}{1-q^n}.
\end{equation}
Using the Pochhammer symbol notation $(a;q)_\infty:=\prod_{n\geq 1}(1-aq^{n-1})$, Euler's formula for the generating function of partitions may be written as $\sfp(q)=(q;q)_\infty^{-1}$. 
For $\ell\geq 1$, the generating function of $\ell$-partitions (with respect to their size) is then
\[\sfp(q)^\ell=\sum_{\bla\in\cP^\ell}q^{|\bla|}=\prod\limits_{n=1}^\infty \frac{1}{(1-q^n)^\ell}=
(q;q)_\infty^{-\ell}.\]

\begin{Not}
We will use sans serif typeface for our generating functions, e.g. $\sff(q)$. For $n\in\N$, we will write the coefficient of $q^n$ in $\sff(q)$ as $f(n)$, using ordinary typeface. Thus, $\sff(q)=\sum\limits_{n\in\N}f(n)q^n$.

We will be working with multiple parameters which index our generating functions. We will always put the level $\ell$ in superscript and the parameters $e$ or $d$ (which are the forbidden hooklengths defining the cores) in subscript. These decorations precede the variable $q$. So, for instance, the notation $\sff^\ell(q)$ will indicate the generating function of some level $\ell$ thing, rather than the $\ell$-th power of a generating function $\sff(q)$\footnote{However, in the example of the generating function of $\ell$-partitions with respect to size, the two amount to the same thing!}. For the $\ell$'th power of $\sff(q)$ we will write $\sff(q)^\ell$. We hope this will not lead to confusion.

\end{Not}

\subsection{The $e$-core and $e$-quotient of a partition and the Littlewood decomposition}
\label{sec_Lit}

\newcommand\gry{\Yfillcolour{black!40}}
\newcommand\white{\Yfillcolour{white}}

Fix $e\in\Z_{\geq 1}$ and let $\la$ be a partition. An  \textit{$e$-rim-hook} is a sequence of $e$ adjacent\footnote{by which we mean, they must share an edge} boxes on the border of the Young diagram of $\la$ such that removing them still yields the Young diagram of a partition. We will often refer to an $e$-rim-hook simply as an {\em $e$-hook}. 
A partition $\la$ is called an \textit{$e$-core} if its Young diagram has no $e$-hook. 
Let $\cC_e(n)$ be the set of $e$-core partitions of $n$, and set $c_e(n)=|\cC_e(n)|$.

\medskip

Recursively removing all $e$-hooks from a partition $\la$ yields the \textit{$e$-core of $\la$}, denoted $\la_{(e)}$.
\begin{Exa}\label{exa_core_1}
Let $\la=(10,6,3^2)\in\cP(22)$.
The Young diagram of $\la$ has a $7$-rim-hook, for instance the one represented in gray below, so it is not a $7$-core.
$$\ds
\Yboxdim{7pt}
\young(<><><><><><><><><><>,<><>!\gry<><><><>,!\white<><>!\gry<>,!\white<>!\gry<><>)
$$
We have to remove two successive $7$-hooks to reach the $7$-core of $\la$, which is $\la_{(7)} = (3,2^2,1).$
On the other hand, if we take $e=5$ then it is easy to see that $\la$ is a $5$-core, so $\la_{(5)} = \la$.
\end{Exa}
There might be several ways to remove $e$-hooks, but it turns out that $\la_{(e)}$ does not depend on the choice of
the $e$-hooks to be removed. 
This is best seen using the abacus of the partition, as we explain next.

\medskip

A level $1$ \textit{symbol} is a subset $\bS$ of $\Z$ satisfying the condition that there exists $M\in\N$ such that
$\be\in \bS$ for all $\be\in\Z_{\leq -M}$ and $\be\notin \bS$ for all $\be \in\Z_{\geq M}$.
Because of these conditions, we can write $\bS=\{ \be_1,\be_2,\be_3,\ldots \}$ with $\be_k>\be_{k+1}$ for all $k\geq 1$ and $\be_{k+1}=\be_k-1$ for all $k$ sufficiently large. 
We picture a symbol $\bS$ by a one-row {\em abacus} $\bA$ by placing beads on the integer number line at positions $\be_1,\be_2,\be_3,\ldots$.  A one-row abacus is sometimes also called a {\em Maya diagram}. 
\begin{Exa}\label{ex_abacus} If $\bS=\{ 10,5,1,0,-4,-5,-6,\ldots \}$, we picture its abacus $\bA$ as
 
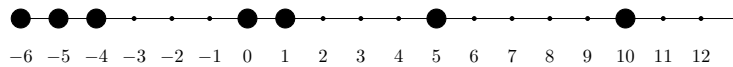
\begin{center}
\begin{tikzpicture}[scale=0.5, bb/.style={draw,circle,fill,minimum size=2.5mm,inner sep=0pt,outer sep=0pt}, wb/.style={draw,circle,fill,minimum size=0.5mm,inner sep=0pt,outer sep=0pt}]
%%%%%%%%%%%%%%%%%%%%%
%%%  On dessine les lignes de l'abaque et tous les emplacements
\draw [line width=0.1mm] (-6,1) -- (13,1);
    \foreach \k in {-6,...,12}
    {
        \node [wb] at (\k,1) {};
    }
%%%%%%%%%%%%%%%%%%%%%
%%%  On place les billes
\foreach \k/\j in { -6/1,-5/1,-4/1, 0/1,1/1, 5/1, 10/1}
{
    \node [bb] at (\k,\j) {};
}
%%%%%%%%%%%%%%%%%%%%%
%%%  On fait appraître la graduation
\foreach \k in {-6,...,12}
{
    \node [scale = 0.7] at (\k,0) {$\k$};
}
\end{tikzpicture}
\end{center}
 Note that there are infinitely many beads to the left and infinitely many spaces to the right.
 \end{Exa}

\medskip

Given a partition $\la= (\la_1,\la_2,\ldots )$,  we view it as having infinitely many parts of size $0$ and set $\be_k=\la_k-k+1$ for each $k\geq 1$. The $\be_k$'s are often called the {\em $\be$-numbers of $\la$}. Note that the $\be$-numbers of $\la$ satisfy the conditions  $\be_k>\be_{k+1}$ for all $k\geq 1$ and $\be_{k+1}=\be_k-1$ for all $k$ sufficiently large. The set of $\be$-numbers of $\la$ is therefore a symbol which we denote by $\bS(\la)$ and call the \textit{symbol of} $\la$:
\[ 
\bS(\la) = \{ \la_k-k+1 \,\mid \, k\in\Z_{\geq 1}  \}.
\]
We let $\bA(\la)$ be the corresponding one-row abacus and refer to it as the {\em abacus of $\la$}.

\medskip

Now, removing an $e$-hook from $\la$ amounts to moving a bead $e$ positions to the left to an empty space in $\bA(\la)$. That is, if $\be\in \bS(\la)$ and $\be-e\notin\bS(\la)$, then we replace $\be$ with $\be-e$ in $\bS(\la)$. 
Doing this recursively yields the $e$-core $\la_{(e)}$ when no more $e$-hooks can be removed, and shows that $\la_{(e)}$ is well-defined.
\begin{Exa} Let $\la=(10,6,3^2)$.
The symbol of $\la$ is as in \Cref{exa_core_1}, and its abacus $\bA(\la)$ is depicted in \Cref{ex_abacus}. Take $e=7$ again and let us compute $\la_{(7)}$ from $\bA(\la)$. In total, one can move two beads $7$ positions to the left to obtain the following abacus
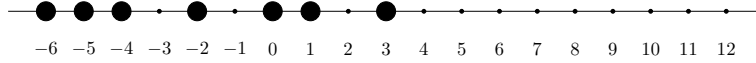
\begin{center}
\begin{tikzpicture}[scale=0.5, bb/.style={draw,circle,fill,minimum size=2.5mm,inner sep=0pt,outer sep=0pt}, wb/.style={draw,circle,fill,minimum size=0.5mm,inner sep=0pt,outer sep=0pt}]
%%%%%%%%%%%%%%%%%%%%%
%%%  On dessine les lignes de l'abaque et tous les emplacements
\draw [line width=0.1mm] (-7,1) -- (13,1);
    \foreach \k in {-6,...,12}
    {
        \node [wb] at (\k,1) {};
    }
%%%%%%%%%%%%%%%%%%%%%
%%%  On place les billes
\foreach \k/\j in { -6/1,-5/1,-4/1,-2/1, 0/1,1/1, 3/1}
{
    \node [bb] at (\k,\j) {};
}
%%%%%%%%%%%%%%%%%%%%%
%%%  On fait appraître la graduation
\foreach \k in {-6,...,12}
{
    \node [scale = 0.7] at (\k,0) {$\k$};
}
\end{tikzpicture}
\end{center}
which is $\bA(\la_{(7)})$. We read the $7$-core $\la_{(7)}=(3,2^2,1)$ off its abacus by counting the number of spaces to the left of each bead, starting from the rightmost bead.
\end{Exa}

From the abacus description, it is clear that a partition $\la$ is a $2$-core if and only if the beads and spaces in its abacus alternate in a pattern like

\begin{center}
\begin{tikzpicture}[scale=0.5, bb/.style={draw,circle,fill,minimum size=2.5mm,inner sep=0pt,outer sep=0pt}, wb/.style={draw,circle,fill,minimum size=0.5mm,inner sep=0pt,outer sep=0pt}]
%%%%%%%%%%%%%%%%%%%%%
%%%  On dessine les lignes de l'abaque et tous les emplacements
\draw [line width=0.1mm] (-7,1) -- (13,1);
    \foreach \k in {-6,...,12}
    {
        \node [wb] at (\k,1) {};
    }
%%%%%%%%%%%%%%%%%%%%%
%%%  On place les billes
\foreach \k/\j in { -6/1,-5/1,-4/1, -3/1, -2/1, 0/1, 2/1, 4/1, 6/1}
{
    \node [bb] at (\k,\j) {};
}
%%%%%%%%%%%%%%%%%%%%%
%%%  On fait appraître la graduation
%\foreach \k in {-6,...,12}
%{
%    \node [scale = 0.7] at (\k,0) {$\k$};
%}
\end{tikzpicture}
\end{center}

which is the case if and only if $\la=\emptyset$ or $\la=(k,k-1,\ldots,3,2,1)$ for some $k\geq 1$. It follows that there is a $2$-core partition of $n$ if and only if $n=\frac{k(k+1)}{2}$ for some $k\in\N$, that is, $n$ is a triangular number. In this case, the Young diagram of $\la=\la_{(2)}$ has a triangular, or staircase, shape. The first few $2$-cores are depicted below:
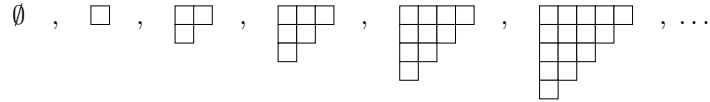
\[
\ds
\Yboxdim{7pt}
\emptyset\quad, \quad \young(<>)\quad,\quad  \young(<><>,<>)\quad, \quad \young(<><><>,<><>,<>)\quad,\quad \young(<><><><>,<><><>,<><>,<>) \quad , \quad \young(<><><><><>,<><><><>,<><><>,<><>,<>) \quad ,  \;\ldots
\]

\medskip

The number $c_3(n)$ of $3$-core partitions of $n$ was given in \cite{GO1996} by the formula $c_3(n)=\sum_{d\mid 3n+1}\left( \frac{d}{3}\right)$, where $\left( \frac{-}{3}\right)$ is the Legendre symbol.  An explicit description of the $3$-core partitions themselves was given in \cite{Robbins}. Among the numbers $n$ that admit a $3$-core partition, we have found the following observation useful. First, we recall that an integer of the form $k(k+1)$, $k\in\N$, is called a \textit{pronic number}. 
\begin{Lem}\label{lem_size_3cores}
Suppose $n\in\N$ is a square, a square plus $1$, or a pronic number. Then 
there exists a $3$-core partition of size $n$.
\end{Lem}
\begin{proof}
If $n=k^2$, consider the partition $\la = (2k-1,2k-3,\ldots, 3,1)$, whose parts are all positive odd numbers up to $2k-1$ with no part repeated. 
If $n=k^2+1$, consider the partition $\mu = (2k-1,2k-3,\ldots, 3,1,1)$, where $\mu$ is obtained from $\la$ above by adding one more box at the bottom in a new row. 
If $n=k^2+k$, consider the partition $\la = (2k,2k-2,\ldots, 4,2)$, whose parts are all even positive numbers up to $2k$ with no part repeated. 
These are three-cores by \cite[Theorems 3 and 4]{Robbins} and it is easy to check they are of size $n$.
\end{proof}

\begin{Exa} The three types of $3$-core partitions from \Cref{lem_size_3cores} for $k=4$ are illustrated below, yielding $3$-cores of sizes $4^2=16$, $4^2+1=17$, and $4^2+4=20$ respectively:
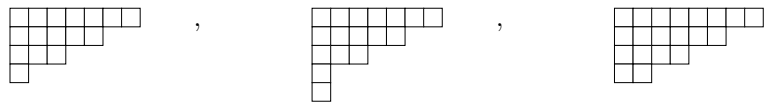
\[
\ds
\Yboxdim{7pt}
\young(<><><><><><><>,<><><><><>,<><><>,<>)\qquad, \qquad \qquad \young(<><><><><><><>,<><><><><>,<><><>,<>,<>)\qquad,\qquad\qquad \young(<><><><><><><><>,<><><><><><>,<><><><>,<><>)\;.
\] 
\end{Exa}

We now introduce the $e$-quotient of a partition by defining the $e$-abacus of a partition. Let $\la$ be  partition.
For each $1\leq j\leq e$, consider the abacus $\bA^{(j)}$ given by the symbol
\[
\bS^{(j)} = \left\{ \left\lfloor\frac{\be-1}{e} \right\rfloor +1 \,\mid \, \be\in \bS(\la) \text{ and }  j = (\be-1) \mod e +1 \right\}.
\]
Then $(\bS^{(1)}, \ldots, \bS^{(e)})$ is a {\em level $e$ symbol}. We depict it by a corresponding {\em e-row abacus}, or simply an {\em e-abacus}, by stacking the $e$ abaci of $\bS^{(1)},\ldots,\bS^{(e)}$ on top of one another, so that $\bA^{(j+1)}$ lies on top of $\bA^{(j)}$ for each $j=1,\ldots,e-1$.
The $e$-abacus of $\la$ corresponds to an $e$-partition called the \textit{$e$-quotient} of $\la$ whose $j$'th component partition is the unique partition whose one-row abacus is, up to a horizontal shift, $\bA^{(j)}$. We denote the $e$-quotient of $\la$ by $\la^{(e)}$. 

\begin{Exa}\label{exa_cores_A}
Let $\la = (8,5^2,2^2,1) \in\cP(23)$, and choose $e=5$.
The $5$-abacus of $\la$ is 
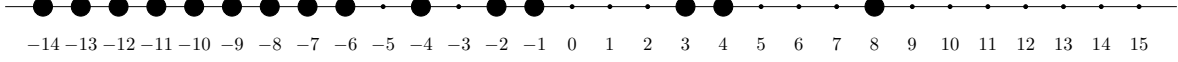
\begin{center}
\begin{tikzpicture}[scale=0.5, bb/.style={draw,circle,fill,minimum size=2.5mm,inner sep=0pt,outer sep=0pt}, wb/.style={draw,circle,fill,minimum size=0.5mm,inner sep=0pt,outer sep=0pt}]
%%%%%%%%%%%%%%%%%%%%%
%%%  On dessine les lignes de l'abaque et tous les emplacements
\draw [line width=0.1mm] (-15,1) -- (16,1);
    \foreach \k in {-14,...,15}
    {
        \node [wb] at (\k,1) {};
    }
%%%%%%%%%%%%%%%%%%%%%
%%%  On place les billes
\foreach \k/\j in { -14/1,-13/1,-12/1,-11/1,-10/1,-9/1,-8/1,-7/1,-6/1,-4/1,-2/1,-1/1,3/1,4/1,8/1}
{
    \node [bb] at (\k,\j) {};
}
%%%%%%%%%%%%%%%%%%%%%
%%%  On fait appraître la graduation
\foreach \k in {-14,...,15}
{
    \node [scale = 0.7] at (\k,0) {$\k$};
}
\end{tikzpicture}
\end{center}
We can figure out $\la^{(5)}$ by subdividing the abacus into packets of $5$,
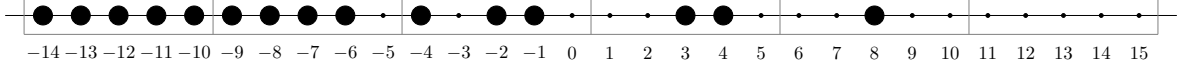
\begin{center}
\begin{tikzpicture}[scale=0.5, bb/.style={draw,circle,fill,minimum size=2.5mm,inner sep=0pt,outer sep=0pt}, wb/.style={draw,circle,fill,minimum size=0.5mm,inner sep=0pt,outer sep=0pt}]
%%%%%%%%%%%%%%%%%%%%%
%%%  On dessine les lignes de l'abaque et tous les emplacements
\draw [line width=0.1mm] (-15,1) -- (16,1);
    \foreach \k in {-14,...,15}
    {
        \node [wb] at (\k,1) {};
    }
%%%%%%%%%%%%%%%%%%%%%
%%%  On place les billes
\foreach \k/\j in { -14/1,-13/1,-12/1,-11/1,-10/1,-9/1,-8/1,-7/1,-6/1,-4/1,-2/1,-1/1,3/1,4/1,8/1}
{
    \node [bb] at (\k,\j) {};
}
%%%%%%%%%%%%%%%%%%%%%
%%%  On fait appraître la graduation
\foreach \k in {-14,...,15}
{
    \node [scale = 0.7] at (\k,0) {$\k$};
}
%%%%%%%%%%%%%%%%%%%%%
%%%  On dessine les rectangles
\draw [line width=0.1mm, color=gray] (-14.5,0.5) -- (15.5,0.5);
\draw [line width=0.1mm, color=gray] (-14.5,1.5) -- (15.5,1.5);
\foreach \k in {-14.5, -9.5, -4.5, 0.5, 5.5, 10.5, 15.5}
{
\draw [line width=0.1mm, color=gray] (\k,1.5) -- (\k,0.5);
}
\end{tikzpicture}
\end{center}
and then rotating each rectangle 90 degrees anticlockwise, assembling them into the $5$-abacus
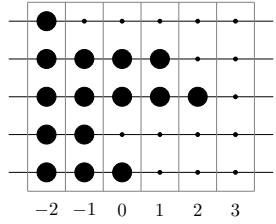
\begin{center}
\begin{tikzpicture}[scale=0.5, bb/.style={draw,circle,fill,minimum size=2.5mm,inner sep=0pt,outer sep=0pt}, wb/.style={draw,circle,fill,minimum size=0.5mm,inner sep=0pt,outer sep=0pt}]
%%%%%%%%%%%%%%%%%%%%%
%%%  On dessine les lignes de l'abaque et tous les emplacements
\foreach \j in {1,2,3,4,5}
{
\draw [line width=0.1mm] (-3,\j) -- (4,\j);
    \foreach \k in {-2,...,3}
    {
        \node [wb] at (\k,\j) {};
    }
}
%%%%%%%%%%%%%%%%%%%%%
%%%  On place les billes
\foreach \k/\j in { 
-2/1,-1/1,0/1,
-2/2,-1/2,
-2/3,-1/3,0/3,1/3,2/3,
-2/4,-1/4,0/4,1/4,
-2/5
}
{
    \node [bb] at (\k,\j) {};
}
%%%%%%%%%%%%%%%%%%%%%
%%%  On fait appraître la graduation
\foreach \k in {-2,...,3}
{
    \node [scale = 0.7] at (\k,0) {$\k$};
}
%%%%%%%%%%%%%%%%%%%%%
%%%  On dessine les rectangles
\draw [line width=0.1mm, color=gray] (-2.5,0.5) -- (3.5,0.5);
\draw [line width=0.1mm, color=gray] (-2.5,5.5) -- (3.5,5.5);
\foreach \k in {-2.5,-1.5,-0.5,0.5,1.5,2.5,3.5}
{
\draw [line width=0.1mm, color=gray] (\k,5.5) -- (\k,0.5);
}
\end{tikzpicture}
\end{center}
We obtain $\la^{(5)}=\bemp$, the empty $5$-partition.
\end{Exa}

We have seen that removing an $e$-hook amounts to moving a bead $e$ positions to the left in $\bA(\la)$.
In turn, this amounts to moving a bead one position to the left in the $e$-abacus of $\la$.
Therefore, the following key property is straightforward.

\begin{Lem}\label{charac_e_cores}
The partition $\la$ is an $e$-core if and only if $\la^{(e)} =\bemp$, the empty $e$-partition.
\end{Lem}

As an illustration, the partition $\la$ in \Cref{exa_cores_A} is a $5$-core as $\la^{(5)}=\bemp$.

\medskip

The $e$-abacus construction thus provides us with a bijection
\begin{equation}
\label{Lit}
\begin{array}{rcl}
\cP & \longrightarrow & \cC_e \times \cP^e
\\
\la &\longmapsto & ( \la_{(e)}, \la^{(e)} )
\end{array}
\end{equation}
which is usually referred to as the \textit{Littlewood decomposition} \cite{Lit1951}. A partition is thus uniquely determined by its $e$-core and $e$-quotient.
Moreover, the sizes of a partition, its $e$-core and its $e$-quotient are related by 
\begin{equation}
\label{size_e_core_e_quot}
|\la| = |\la_{(e)}| + e| \la^{(e)} |.
\end{equation}
This formula reflects the fact that moving a bead to the left in the $e$-abacus of $\la$ corresponds to removing an $e$-hook from $\la$, which reduces the size of $\la$ by $e$. Pushing all beads to the left in the $e$-abacus of $\la$, like tidying up the books on a bookshelf\footnote{Thanks to Roman Gonin for this analogy!}, yields the $e$-abacus of the $e$-core of $\la$.

\begin{Exa}
Let $e=2$ and consider the partition $\la=(9,5^2,4^3,1)$. Under the Littlewood bijection, it goes to the pair $(\la_{(2)},\la^{(2)})=\left((3,2,1),\left((3,1^3),(4,3)\right)\right)$. We $|\la_{(2)}|+2|\la^{(2)}|=6+2\cdot13=32=|\la|$.
\end{Exa}

\subsection{Sums of squares, triangular numbers, and Ramanujan's $\phi$- and $\psi$-functions}
In this section, we recall some classical results from number theory and introduce two functions that will be important for us later. We will repeatedly use three famous theorems from classical number theory:

\medskip

\emph{Legendre's Three-Square Theorem.} Let $n\in\N$. Then $n$ can be written as the sum of three squares if and only if $n\neq 4^j(8k+7)$ for all $j,k\geq 0$.

\medskip

\emph{Lagrange's Four-Square Theorem.} Every $n\in\N$ is the sum of four squares.

\medskip

\emph{Gauss' Eureka Theorem.} Every $n\in\N$ is the sum of three triangular numbers.

\medskip

Next, consider the functions
\[
\ds \phi(q) = \sum_{n\in\Z} q^{n^2} =1+\sum_{n>0}2q^{n^2} \qquad \hbox{ and } \qquad \psi(q)=\sum_{t\geq 0}q^{\frac{t^2+t}{2}}.
\]
These are known as Ramanujan's $\phi$- and $\psi$-functions. Gauss' Eureka Theorem says that the coefficient of $q^n$ in $\psi(q)^3$ is positive for every $n\in\N$. Lagrange's Four-Square Theorem says that the coefficient of $q^n$ in $\phi(q)^4$ is positive for every $n\in\N$, while Legendre's Three-Square Theorem says that the coefficient of $q^n$ in $\phi(q)^3$ is positive unless $n$ belongs to the sequence \href{https://oeis.org/A004215}{A004215}. In particular, if $n\equiv1,2,3,5,$ or $6\bmod 8$, then $n$ is the sum of three squares. Gauss' Eureka Theorem may be deduced as a corollary of Legendre's Three-Square Theorem, and we will often imitate that argument.

\begin{Lem}\label{phi and psi prod}
We have the following product formulas:

  \begin{align*}
      \phi(q) = (-q;q^2)_\infty^2(q^2;q^2)_\infty^2 &=\prod_{n\geq 1}(1+q^{2n-1})^2(1-q^{2n})^2\\
       &= \prod_{n\geq 1} \frac{(1+q^{2n-1})(1-q^{2n})}{(1-q^{2n-1})(1+q^{2n})}
    \end{align*}
and
    \begin{align*}
      \psi(q) = (q;q)_\infty(-q;q)_\infty^2 &= \prod_{n\geq 1}(1-q^n)(1+q^n)^2\\
       &= \prod_{n\geq 1}\frac{(1-q^{2n})^2}{1-q^n}.
    \end{align*}
\end{Lem}
\begin{proof}
 These well-known formulas may be derived from Jacobi's triple product identity.
 The second variant of the formula for $\phi(q)$ is \cite[Eq. (A.1)]{KassReut2018}.%, while the formula for $\psi(q)$ is the formula for $\sfc_2(q)$ in \Cref{gf_cores}.
\end{proof}

The following formulas will be useful later.
No doubt they are well-known but we give proofs here for completeness, as they are easy to establish by manipulation of formal power series or by counting arguments.

\begin{Lem}\label{Sylvie4alt}
It holds that
\[
\psi(q)^2=\psi(q^2)\phi(q).
\]
\end{Lem}
\begin{proof}
Applying the product formulas in \Cref{phi and psi prod}, we have:
\begin{align*}
\psi(q^2)\phi(q) &= \left(\prod_{n\geq 1} \frac{(1-q^{4n})^2}{(1-q^{2n})}\right)\left( \prod_{n\geq 1} \frac{(1+q^{2n-1})(1-q^{2n})}{(1-q^{2n-1})(1+q^{2n})} \right) \\
			&= \prod_{n\geq 1} \frac{(1-q^{2n})^3(1+q^{2n})^2(1+q^{2n-1})}{(1-q^{2n})(1-q^{2n-1})(1+q^{2n})} \times \prod_{n\geq 1}\frac{(1-q^{2n})}{(1-q^{2n})} \\
			&= \prod_{n\geq 1} \frac{(1-q^{2n})^4(1+q^n)}{(1-q^n)(1-q^{2n})}\\
			&=  \prod_{n\geq 1} \frac{(1-q^{2n})^4}{(1-q^n)^2}\\
			&= \psi(q)^2.
\end{align*}
\end{proof}

\begin{Lem}\label{phipsi}
We have $\phi(q)^2 = \phi(q^2)^2 + 4q\psi(q^4)^2$.
\end{Lem}
\begin{proof}
The coefficient of $q^n$ in $\phi(q)^2$ is the number of ways to write
$n=a^2+b^2$ with $a,b\in\Z$.
We split the computation depending on the parity of $a+b$,
and we start by noticing that $a^2+b^2$ has the same parity as $a+b$.
\begin{itemize}
\item On the one hand, the map $(a,b)\mapsto (\al,\be)$ 
with $\al=\frac{a+b}{2}, \be=\frac{a-b}{2}$ is a bijection
from $\{(a,b)\in\Z^2\mid a+b\text{ even}\}$ to $\Z^2$,
and we have 
$a^2+b^2=2(\al^2+\be^2)$.
\item On the other hand, the map $(a,b)\mapsto (\al,\be)$ with $\al=\frac{a+b-1}{2}$ and $\be=\frac{a-b-1}{2}$ is a bijection
from $\{(a,b)\in\Z^2\mid a+b\text{ odd}\}$ to $\Z^2$, and we have
$a^2+b^2=4 ( \frac{\al(\al+1)}{2}+ \frac{\be(\be+1)}{2} ) + 1$.
\end{itemize}
Therefore we obtain
\begin{align*}
\phi(q)^2 
& = \sum_{a+b \text{ even}} q^{a^2+b^2} +  \sum_{a+b \text{ odd}} q^{a^2+b^2}
\\
& = \sum_{\al,\be\in\Z} q^{2(\al^2+\be^2)} +  \sum_{\al,\be\in\Z} q^{4(\frac{\al(\al+1)}{2}+ \frac{\be(\be+1)}{2} ) + 1}
\\
& = \phi(q^2)^2 + q\left( 2\psi(q^4)\right)^2
\end{align*}
where, in the last step, we have used the usual symmetry for the triangular numbers to restrict to 
$\al,\be\in\N$.
\end{proof}

\subsection{The generating function of the $e$-core partitions}
\label{sec_gf_cores}

Let $e\in\Z_{\geq 1}$. Let
\[
\sfc_e(q) = \sum_{n\geq 0} c_e(n)q^n =  \sum_{\la\in\cC_e} q^{|\la|}
\]
be the generating function of the $e$-cores, where $c_e(n)$ denotes the number of $e$-core partitions of $n$. For example, if $e=4$ and $n=5$ then $c_4(5)=3$ as there are exactly three $4$-core partitions of $5$, given by:
\[
\ds
\Yboxdim{7pt}
\young(<><><><>,<>)\quad, \qquad \young(<><><>,<>,<>)\quad,\qquad\young(<><>,<>,<>,<>)\quad.
\]
In the case that $e=2$, we have already observed that $c_2(n)=1$ if $n$ is a triangular number, and $c_2(n)=0$ otherwise. This yields the following equality that we will often use in this paper:
\[
\sfc_2(q)=\sum_{t\in\N}q^{\frac{t^2+t}{2}}=\psi(q).
\]

Combining \Cref{size_e_core_e_quot} with \Cref{gf_part}, we obtain the product formula for the generating function of $e$-core partitions that was established in \cite{GKS1990}:

\begin{equation}
\label{gf_cores}
\sfc_e(q) = \prod \frac{(1-q^{en})^e}{(1-q^n)}.
\end{equation}
We note that taking $e=1$ in this formula yields $\sfc_1(q)=1$, and taking $e=2$ yields the second version of the product formula for $\psi(q)$ in \Cref{phi and psi prod}.

\medskip

From \Cref{gf_cores}, we easily deduce that $\sfc_e(q)$ divides $\sfc_f(q)$ whenever $e$ divides $f$, as follows.
\begin{Lem}\label{factorization gf e-cores}
For any $e,m\geq 1$, \[ \sfc_{em}(q)=\sfc_e(q)\sfc_m(q^e)^e. \]
\end{Lem}
\begin{proof}
\[
\sfc_{em}(q) = \prod_{n\geq 1} \frac{(1-q^{emn})^{em}}{(1-q^n)}
=\prod_{n\geq 1}\frac{(1-q^{en})^e(1-q^{emn})^{em}}{(1-q^n)(1-q^{en})^e}
=\sfc_e(q)\left(\prod_{n\geq 1}
\frac{(1-(q^e)^{mn})^m}{(1-(q^e)^n)}\right)^e
=\sfc_e(q)\sfc_m(q^e)^e.
\]
\end{proof}
In particular, $\sfc_{2d}(q)=\psi(q)\sfc_d(q^2)^2=\sfc_d(q)\psi(q^d)^d$ for any $d\geq 1$ (the latter variant was proved in \cite[Theorem 5]{Ono1994}). By induction on $k$, it follows that $\sfc_{2^k}(q)=\prod\limits_{i=0}^{k-1}\psi(q^{2^i})^{2^i}$ for any $k\geq 1$. In particular, $\sfc_4(q)=\psi(q)\psi(q^2)^2$.

\subsection{Granville and Ono's theorem on positivity of the number of $e$-core partitions of $n$}
\label{sec_GO}

The only $1$-core partition is $\emptyset$. 
For $e\in\{2,3\}$, there are infinitely many integers $n$ such that $c_e(n)= 0$. Indeed, $c_2(n)=0$ whenever $n$ is not a triangular number. As for $e=3$, it turns out that $c_3(n)=0$ for all $n\equiv 3 \bmod 4$ \cite[Corollary 3]{Robbins}. In fact, not only $c_2(n)$ but also $c_3(n)$ is $0$ for almost all $n$ \cite{GO1996}.

\medskip

By contrast, work in the early 1990's on $5$-cores and $7$-cores showed they exhibit a very different behavior: for any $n\in\N$, there exist a $5$-core partition of $n$ and a $7$-core partition of $n$ \cite{ErdmannMichler,GKS1990}. The conjectural statement that $c_e(n)>0$ for all $n\in\N$ whenever $e\geq 4$ was known (with the letter $t$ in place of $e$) as the ``$t$-core conjecture.'' The $t$-core conjecture was settled in the affirmative in two papers by Ono \cite{Ono1994} and Granville-Ono \cite{GO1996} dating from the mid-1990's:

\begin{Thm}\cite[Theorem 1]{GO1996}\label{grono}
It holds that $c_e(n)>0$ for all $n\in\N$, provided that $e\geq 4.$
\end{Thm}

\medskip

The main idea of the papers \cite{Ono1994,GO1996} is to treat the enumeration of $e$-core partitions as a number-theoretic problem, thanks to arithmetic properties of the generating function \Cref{gf_cores}. We summarize the steps of the proof of the $t$-core conjecture, starting with Ono's paper \cite{Ono1994}, which came close. As $c_e(n)\leq c_f(n)$ whenever $e$ divides $f$, it suffices to show that $c_e(n)$ is positive for all $n\in\N$ when $e\geq 5$ is prime and for $e=4,6,9$.  When $p\geq 5$ is prime, the generating function of the $p$-core partitions, $\sfc_p(q)$, is almost a modular form (it is only off by an integral power of $q$) \cite{Ono1994}. A deep result in the theory of modular forms (Deligne's Theorem, stated as \cite[Theorem 1]{Ono1994}) then allowed Ono to deduce in a short argument that for all primes $p\geq 5$, $c_p(n)$ must be positive for all but finitely many $n$, thus establishing \cite[Theorem 2]{Ono1994}. He also claimed to have checked positivity of $c_{11}(q)$. Similar tactics dispatched the $e=4$ and $e=9$ cases \cite[Theorems 3 and 6]{Ono1994}, while the case $e\equiv 2\bmod 4$, $e\geq 6$, succumbed to Gauss's Eureka Theorem 
\cite[Theorem 5]{Ono1994}.

\medskip

In \cite[§2]{GO1996}, Granville and Ono filled the gap in \cite[Theorem 2]{Ono1994} for primes $p\geq 17$ by showing that $c_e(n)>0$ for all $n\in\N$ whenever $e\geq 17$. The argument consists in applying 
Lagrange's Four-Square Theorem in conjunction with a formula found by Garvan, Kim and Stanton expressing the size of an $e$-core by a quadratic polynomial in $e$ variables \cite[Bijection 2]{GKS1990}. To conclude the $t$-core conjecture held, all that was left was the case $e=13$; for this, they used modular forms \cite[Theorem 5]{GO1996}.

\medskip

The presence of modular forms in the story of $e$-cores is intriguing and suggests a powerful explanatory framework. However, while they offer an elegant method of proof and an avenue for understanding, modular forms turn out not to be strictly necessary to prove positivity. Later, Kiming showed that it is possible to deal with the $e=11$ and $e=13$ cases by elementary methods \cite{Kiming}, and modular forms may be circumvented in the $e=4$ and $e=9$ cases by applying Legendre's Three-Square Theorem. 

\subsection{Defect $0$ blocks of finite groups}\label{sec_def0}

We follow \cite{Craven} for the exposition of block theory of finite groups in this section as well as for an overview of finite groups of Lie type and their unipotent characters.

\medskip

Let $G$ be a finite group, let $p$ be a prime, and let $\mathbb{k}$ be a field of characteristic $p$. If $p$ divides the order of $G$ then
 $\mathbb{k}G$ is not semisimple. Instead, the group algebra breaks into a direct sum of indecomposable two-sided ideals called {\em blocks}, which partition the ordinary irreducible characters of $G$. A {\em block of defect $0$} is a block $B\subset \mathbb{k}G$ such that $p^k$ divides the degree $\chi(1)$ for $\chi$ an ordinary irreducible character lying in $B$. Blocks of defect $0$ are characterized by the property that the block consists of a single ordinary irreducible character and they are isomorphic to matrix algebras \cite{Brauer1944}. We may think of the blocks of defect $0$ as making up the part of the representation theory of $G$ that remains semisimple in characteristic $p$. 
 
\medskip

Take $G=S_n$, the symmetric group on $n$ letters. The ordinary irreducible characters of $S_n$ are labeled by partitions of $n$. Let $\mathbb{k}$ be a field of characteristic $p$. The statement known as the Nakayama Conjecture, which was proved by Brauer and Robinson, characterizes the blocks of $\mathbb{k}S_n$ in terms of $p$-cores:
\begin{Thm}\label{NakayamaConjecture1}\cite{Brauer1947,Robinson1947}
Let $\la,\mu\in\cP(n)$ and let $\chi_\la,\chi_\mu$ be the corresponding irreducible characters of $S_n$. Then $\chi_\la$ and $\chi_\mu$ belong to the same block of $\mathbb{k}S_n$ if and only if $\la_{(p)}=\mu_{(p)}$.
\end{Thm}
Consequently, the blocks of defect $0$ in $\mathbb{k}S_n$ are labeled by $p$-core partitions of $n$. Granville and Ono deduced the following corollary of \Cref{grono}:
\begin{Cor}\cite[Corollary 1]{GO1996}
Let $p=\mathrm{char}(\mathbb{k})>0$. Then $\mathbb{k}S_n$ has a defect $0$ block for all $n\in\N$ if and only if $p\geq 5$. 
\end{Cor}
They deduced the analogous result for the alternating groups $A_n$ as well. This allowed them to build off the results of Michler \cite{Michler1986} and Willems \cite{Willems1988} to complete the classification of the finite simple groups with blocks of defect $0$ \cite[Corollary 2]{GO1996}.

\medskip

Now let us explain a second corollary of \Cref{grono} that does not appear in the paper \cite{GO1996}, which is about the positivity of the number of defect $0$ {\em unipotent} blocks of finite general linear groups in non-defining characteristic. 
Throughout the rest of this section, let ${{v}}$ be a power of a prime, and let $\mathbb{F}_{{v}}$ be the field with ${{v}}$ elements. The finite general linear group $\mathrm{GL}_n({{v}})$ consists of all invertible matrices with entries in $\mathbb{F}_{{v}}$, with group operation given by matrix multiplication. It is an example of a {\em finite group of Lie type}. 
Finite groups of Lie type admit a distinguished type of irreducible character called {\em unipotent characters}, defined by Deligne and Lusztig in terms of virtual representations of Deligne-Lusztig varieties \cite{DeligneLusztig}. No algebraic definition of the unipotent characters is known. Unipotent characters admit a parametrization which depends only on Weyl group data and is independent of ${{v}}$. Moreover, the rest of the irreducible characters can be described in terms of them \cite{BonnafeRouquier02,BDR17,Ruhstorfer20}. 
For $\mathrm{GL}_n({{v}})$, the unipotent characters are parametrized by the partitions of $n$.

 \medskip
 
 Let $G({{v}})$ be a finite group of Lie type and consider $\mathbb{k}G({{v}})$, where ${{l}}=\mathrm{char}(\mathbb{k})$ is coprime to ${{v}}$.
The blocks of $\mathbb{k}G({{v}})$ partition the irreducible characters of $G({{v}})$, and a {\em unipotent block} is a block of $\mathbb{k}G({{v}})$ containing a unipotent character. Thus, the unipotent blocks partition the unipotent characters of $G({{v}})$. 
Fong and Srinivasan described this partition combinatorially for $G({{v}})=\mathrm{GL}_n({{v}})$, extending the Nakayama Conjecture to the case of finite general linear groups.
\begin{Thm}\cite{FongSrinivasan1982}
Two unipotent characters labeled by partitions $\la$ and $\mu$ belong to the same block of $\mathbb{k}G_n({{v}})$ if and only if $\la_{(e)}=\mu_{(e)}$, where $e$ is the order of ${{v}}$ modulo ${{l}}$.
\end{Thm}
Unlike in the case of symmetric groups, here $e$ does not need to be prime. Sometimes $e$ is referred to as the {\em quantum characteristic}, as it is the minimal positive integer such that $1+{{v}}+{{v}}^2+\ldots+{{v}}^{e-1}=0$ in $\mathbb{k}$.

 \medskip
 
A {\em defect $0$ unipotent block} of a finite group of Lie type $G({{v}})$ is a unipotent block over $\mathbb{k}$ that is a block of defect $0$. Thus, a defect $0$ unipotent block contains a single, unipotent character, which remains irreducible when reduced modulo ${{l}}$ and does not appear as a composition factor in the ${{l}}$-modular reduction of any other character. We can thus state a second representation-theoretic corollary of \Cref{grono}:
\begin{Cor}\label{def0gln}
Let ${{l}}$ be a prime not dividing ${{v}}$, and let $e$ be the order of ${{v}}$ mod $l$. Let $\mathbb{k}$ be a field of characteristic ${{l}}$. Then $\mathbb{k}\mathrm{GL}_n({{v}})$ has a unipotent block of defect $0$ for every $n\in\N$ if and only if $e\geq 4$.
\end{Cor}

%%%%%%%%%%%%%%
%%%%%%%%%%%%%%%
%%%%%%%%%%%%%%
%%%%%%%%%%%%%%%
%%%%%%%%%%%%%%

\section{Higher-level generalizations of core partitions}
\label{sec_higher_level_cores}

In this section, we study generalizations of $e$-cores and $e$-quotients to levels $\ell>1$. We will work within the framework of {\em charged multipartitions}. A charged multipartition $|\bla,\bs\rangle$, or {\em charged $\ell$-partition} when $\ell\in\Z_{\geq 2}$ is being emphasized, is a pair consisting of an $\ell$-partition $\bla\in\cP^\ell$ and a {\em charge} $\bs\in\Z^\ell$.

\medskip

In level $2$, charged bipartitions were first introduced by Lusztig under the guise of Lusztig symbols in order to parametrize the unipotent characters of finite classical groups, which are the most important class of irreducible characters of those groups \cite{Lusztig1977}. The analogs of cores and quotients for Lusztig symbols are called cocores and coquotients. The combinatorics of cohooks and cocores was studied by Olsson \cite{Olsson1986}. Fong and Srinivasan then showed that cocores describe the unipotent blocks of the finite classical groups in positive, non-defining characteristic \cite{FongSrinivasan1989}. Malle generalized the definition of cocores and coquotients from level $2$ to level $\ell$ for any $\ell\in\Z_{\geq 2}$ in the context of unipotent characters of spetses \cite{Malle1995}. Malle's higher-level spetsial cores and quotients 
include the cocores as the special case $\ell=2$. 

\medskip

More recently, different authors defined and studied several variants of cores and quotients in higher level, motivated by obtaining a description of the blocks of cyclotomic Hecke algebras at roots of unity by a privileged subset of multipartitions, analogous to the description of the blocks of the Hecke algebra of the symmetric group at a root of unity by core partitions (yet another variant of the Nakayama conjecture) \cite{Fayers2007, Fayers2019, JaconLecouvey2020, JaconLecouvey2021, ChlouverakiJacon2023, Lyle2024, CGJ2025}.

\medskip

These two different -- albeit closely linked -- representation-theoretic directions correspond roughly to two different perspectives on how to generalize partitions to higher level. 
In this section, we recall some of these constructions and establish precise relationships
between the two approaches.

\subsection{Charged $\ell$-partitions and $\ell$-abaci}\label{sec_symbols_defs} 

Let $\bla=(\la^1,\la^2,\ldots,\la^\ell)\in\cP^\ell$ and let $\bs=(s_1,s_2,\ldots,s_\ell)\in\Z^\ell$. Fix $j\in\{1,\ldots,\ell\}$. We will associate a unique {\em $\ell$-symbol} to the charged $\ell$-partition $|\bla,\bs\rangle$, and vice versa.
 First, we may associate a symbol to the pair $(\la^j,s_j)$ by shifting the $\be$-numbers of $\la^j$ by $s_j$:
\[
\bS(\la^j,s_j)=\{\be+s_j\;\mid\; \be\in\bS(\la^j) \}=\{ (\la^j)_k-k+1+s_j\;\mid\; k\in\Z_{\geq 1}\}.
\]
(This is the symbol of the {\em charged partition} $|\la^j,s_j\rangle$.)
The {\em symbol} of $|\bla,\bs\rangle$ is then defined to be the $\ell$-tuple of symbols
\[
\bS(\bla,\bs)=\left( \bS(\la^1,s_1),\bS(\la^2,s_2),\ldots,\bS(\la^\ell,s_\ell)\right).
\]

\medskip

If $\bS=(\bS^{(1)},\bS^{(2)},\ldots,\bS^{(\ell)})$ is any $\ell$-tuple of symbols, $\bS$ is called an {\em $\ell$-symbol}.  
We identify $\bS$ with a subset of $\Z\times\{1,\ldots,\ell\}$
via $\be\in\bS^{(j)} \mapsto (\be,j)$.  
For each $j=1,\ldots,\ell$, let $\bA^{(j)}$ be the abacus of $\bS^{(j)}$. Then $\bA=\left(\bA^{(1)},\ldots,\bA^{(\ell)}\right)$ is called an {\em $\ell$-abacus}. We stack the rows with $\bA^{(1)}$ at the bottom and $\bA^{(\ell)}$ at the top, labeling the rows from $1$ to $\ell$ when reading from bottom to top. The abacus $\bA$ has a bead in row $j$ and column $\be$ if and only if $\be\in\bS^{(j)}$. It is simply a convenient variation on the definition of the symbol. We will abuse notation and write things like $\be\in\bA^{(j)}$ by which we mean there is a bead in column $\be$ and row $j$. We will not really distinguish between symbols and their abaci. We let $\bA(\bla,\bs)$ be the $\ell$-abacus identified with the symbol $\bS(\bla,\bs)$.

\medskip

\begin{Exa}\label{exa_3ab}
Let $\ell=3$ and take $\bla=\left((4,3^2,1),(2,1^4),(7^3,5)\right)\in\cP^3$ and $\bs=(1,3,-4)\in\Z^3$. Then 
$\bA(\bla,\bs)$ is
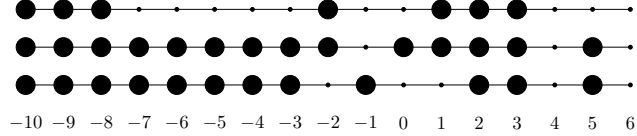
\begin{center}
\begin{tikzpicture}[scale=0.5, bb/.style={draw,circle,fill,minimum size=2.5mm,inner sep=0pt,outer sep=0pt}, wb/.style={draw,circle,fill,minimum size=0.5mm,inner sep=0pt,outer sep=0pt}]
%%%%%%%%%%%%%%%%%%%%%
%%%  On dessine les lignes de l'abaque et tous les emplacements
\foreach \j in {1,2,3}
{
\draw [line width=0.1mm] (-10,\j) -- (6,\j);
    \foreach \k in {-10,...,6}
    {
        \node [wb] at (\k,\j) {};
    }
}
%%%%%%%%%%%%%%%%%%%%%
%%%  On place les billes
\foreach \k/\j in {-10/3,-9/3,-8/3,-2/3,1/3,2/3,3/3,   -10/2,-9/2,-8/2,-7/2,-6/2,-5/2,-4/2,-3/2,-2/2,0/2,1/2,2/2,3/2,5/2,    -10/1,-9/1,-8/1,-7/1,-6/1,-5/1,-4/1,-3/1,-1/1,2/1,3/1,5/1 }
{
    \node [bb] at (\k,\j) {};
}
%%%%%%%%%%%%%%%%%%%%%
%%%  On fait appraître la graduation
\foreach \k in {-10,...,6}
{
    \node [scale = 0.7] at (\k,0) {$\k$};
}
%%%%%%%%%%%%%%%%%%%%%
%%%  On dessine les rectangles
%\draw [line width=0.1mm, color=gray] (-4.5,0.5) -- (2.5,0.5);
%\draw [line width=0.1mm, color=gray] (-4.5,3.5) -- (2.5,3.5);
%\foreach \k in {-3.5, -1.5, 0.5, 2.5}
%{
%\draw [line width=0.1mm, color=gray] (\k,3.5) -- (\k,0.5);
%}
\end{tikzpicture}
\end{center}
\end{Exa}

On the other hand, starting from an $\ell$-abacus $\bA$ we obtain a unique charged $\ell$-partition $|\bla,\bs\rangle$. To find the charge $\bs=(s_1,\ldots,s_\ell)\in\Z^\ell$, we swipe all the beads to the left so there are no gaps and read off the position of the rightmost bead in row $j$ to obtain $s_j$ for each $j=1,\ldots,\ell$. To find $\bla=(\la^1,\ldots,\la^\ell)\in\cP^\ell$, in row $j$ for each $j=1,\ldots,\ell$ we count how many spaces are to the left of each bead, starting from the rightmost bead and working leftwards, and these are the parts of $\la^j$. The procedures of finding the $\ell$-abacus of a charged $\ell$-partition and finding the charged $\ell$-partition of an abacus are inverse to each other, and so we have a bijection between $\ell$-symbols/$\ell$-abaci on the one hand and charged $\ell$-partitions on the other.

\begin{Exa}
Taking $\bA$ to be the $3$-abacus in \Cref{exa_3ab} and swiping its beads all the way to the left yields the abacus
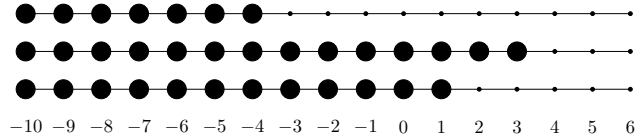
\begin{center}
\begin{tikzpicture}[scale=0.5, bb/.style={draw,circle,fill,minimum size=2.5mm,inner sep=0pt,outer sep=0pt}, wb/.style={draw,circle,fill,minimum size=0.5mm,inner sep=0pt,outer sep=0pt}]
%%%%%%%%%%%%%%%%%%%%%
%%%  On dessine les lignes de l'abaque et tous les emplacements
\foreach \j in {1,2,3}
{
\draw [line width=0.1mm] (-10,\j) -- (6,\j);
    \foreach \k in {-10,...,6}
    {
        \node [wb] at (\k,\j) {};
    }
}
%%%%%%%%%%%%%%%%%%%%%
%%%  On place les billes
\foreach \k/\j in {-10/3,-9/3,-8/3,-7/3,-6/3,-5/3,-4/3,   -10/2,-9/2,-8/2,-7/2,-6/2,-5/2,-4/2,-3/2,-2/2,-1/2,0/2,1/2,2/2,3/2,    -10/1,-9/1,-8/1,-7/1,-6/1,-5/1,-4/1,-3/1,-2/1,-1/1,0/1,1/1 }
{
    \node [bb] at (\k,\j) {};
}
%%%%%%%%%%%%%%%%%%%%%
%%%  On fait appraître la graduation
\foreach \k in {-10,...,6}
{
    \node [scale = 0.7] at (\k,0) {$\k$};
}
%%%%%%%%%%%%%%%%%%%%%
%%%  On dessine les rectangles
%\draw [line width=0.1mm, color=gray] (-4.5,0.5) -- (2.5,0.5);
%\draw [line width=0.1mm, color=gray] (-4.5,3.5) -- (2.5,3.5);
%\foreach \k in {-3.5, -1.5, 0.5, 2.5}
%{
%\draw [line width=0.1mm, color=gray] (\k,3.5) -- (\k,0.5);
%}
\end{tikzpicture}
\end{center}
from which we find that the charge $\bs$ of $\bA$ is $(1,3,-4)$.
\end{Exa}

\begin{Not}
We will let $\cU^\ell$ denote the set of all charged $\ell$-partitions. We will often abuse notation and identify $\cU^\ell$ with the set of all $\ell$-abaci. We remark that while $|\bla,\bs\rangle$ is nothing but a pair $(\bla,\bs)\in\cP^\ell\times\Z^\ell$, the ``physics-y'' notation $|\bla,\bs\rangle$ captures the shape of the abacus with which we identify it.
\end{Not}

\subsection{Level-rank duality}\label{sec_LR}

Fix positive integers $\ell,e\geq 2$ and consider the bijective map 
\[
\begin{array}{lll}
 \Z\times \{ 1,\ldots, \ell\} & \lra & \Z\times \{ 1,\ldots, e\} 
\\
 (k,j) & \longmapsto & (\ell\left(\left\lfloor \frac{k-1}{e}\right\rfloor+1\right) -j+1, (k-1)\mod e +1).
\end{array}
\]
Identifying $\cU^\ell$ as before with a subset of $\Z\times\{1,\ldots,\ell\}$  and  $\cU^e$ with a subset of $\Z\times\{1,\ldots,e\}$, we get a bijective map
\[
\begin{array}{llll}
\LR : & \cU^\ell & \lra & \cU^e
\end{array} 
\]
called the level-rank duality which sends $\ell$-abaci to $e$-abaci \cite{Uglov1999}. The $e$-abacus $\LR(\bA)$ is called the \textit{level-rank dual} of $\bA$.
It can be visualized by tiling $\bA$ into rectangles of width $e$ and height $\ell$ and rotating each rectangle 90 degrees counterclockwise. The convention is that $0$ is in the rightmost column of its rectangle.

\begin{Exa}
\label{exaLR}
Let $\ell=2$ and $e=3$, and consider the charged bipartition $|\bla,\bs\rangle$ for $\bla = \left((1^2), (3,2)\right)\in\cP^2$ and $\bs=(0,-1)\in\Z^2$, whose abacus $\bA$ is
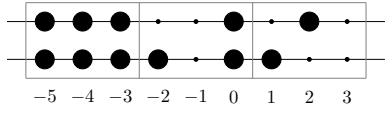
\begin{center}
\begin{tikzpicture}[scale=0.5, bb/.style={draw,circle,fill,minimum size=2.5mm,inner sep=0pt,outer sep=0pt}, wb/.style={draw,circle,fill,minimum size=0.5mm,inner sep=0pt,outer sep=0pt}]
%%%%%%%%%%%%%%%%%%%%%
%%%  On dessine les lignes de l'abaque et tous les emplacements
\foreach \j in {1,2}
{
\draw [line width=0.1mm] (-6,\j) -- (4,\j);
    \foreach \k in {-5,...,3}
    {
        \node [wb] at (\k,\j) {};
    }
}
%%%%%%%%%%%%%%%%%%%%%
%%%  On place les billes
\foreach \k/\j in {-5/1,-4/1,-3/1,-2/1,0/1,1/1,    -5/2,-4/2,-3/2,0/2,2/2 }
{
    \node [bb] at (\k,\j) {};
}
%%%%%%%%%%%%%%%%%%%%%
%%%  On fait appraître la graduation
\foreach \k in {-5,...,3}
{
    \node [scale = 0.7] at (\k,0) {$\k$};
}
%%%%%%%%%%%%%%%%%%%%%
%%%  On dessine les rectangles
\draw [line width=0.1mm, color=gray] (-5.5,0.5) -- (3.5,0.5);
\draw [line width=0.1mm, color=gray] (-5.5,2.5) -- (3.5,2.5);
\foreach \k in {-5.5, -2.5, 0.5, 3.5}
{
\draw [line width=0.1mm, color=gray] (\k,2.5) -- (\k,0.5);
}
\end{tikzpicture}
\end{center}
where we have grouped the beads in $(e\times \ell)$-rectangles.
The level-rank dual $\LR(\bA)$ is then
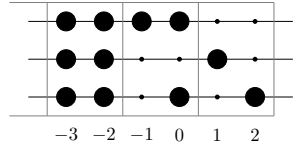
\begin{center}
\begin{tikzpicture}[scale=0.5, bb/.style={draw,circle,fill,minimum size=2.5mm,inner sep=0pt,outer sep=0pt}, wb/.style={draw,circle,fill,minimum size=0.5mm,inner sep=0pt,outer sep=0pt}]
%%%%%%%%%%%%%%%%%%%%%
%%%  On dessine les lignes de l'abaque et tous les emplacements
\foreach \j in {1,2,3}
{
\draw [line width=0.1mm] (-4,\j) -- (3,\j);
    \foreach \k in {-3,...,2}
    {
        \node [wb] at (\k,\j) {};
    }
}
%%%%%%%%%%%%%%%%%%%%%
%%%  On place les billes
\foreach \k/\j in {-3/3,-2/3,-1/3,0/3,   -3/2,-2/2,1/2,    -3/1,-2/1,0/1,2/1 }
{
    \node [bb] at (\k,\j) {};
}
%%%%%%%%%%%%%%%%%%%%%
%%%  On fait appraître la graduation
\foreach \k in {-3,...,2}
{
    \node [scale = 0.7] at (\k,0) {$\k$};
}
%%%%%%%%%%%%%%%%%%%%%
%%%  On dessine les rectangles
\draw [line width=0.1mm, color=gray] (-4.5,0.5) -- (2.5,0.5);
\draw [line width=0.1mm, color=gray] (-4.5,3.5) -- (2.5,3.5);
\foreach \k in {-3.5, -1.5, 0.5, 2.5}
{
\draw [line width=0.1mm, color=gray] (\k,3.5) -- (\k,0.5);
}
\end{tikzpicture}
\end{center}
which is the abacus of $|\bmu,\bt\rangle$ where $\bmu = ((2,1), (2), \emptyset)\in\cP^3$ and $\bt = (0,-1,0)\in\Z^3$.
\end{Exa}

In the rest of the paper, we will use the following convenient notation.
\begin{Not}\label{not_set} Let $\ell\in\Z_{\geq1}$.
For all $X\subseteq \Z^\ell$ and for all $s\in\Z$, let
$X[s]=\{(x_1,\ldots, x_\ell)\in X \mid \sum_{j=1}^\ell x_j=s\}.$
\end{Not}

For a fixed $\bs\in\Z^\ell$, let
$\cU_\bs^\ell = \{ |\bla,\bs\rangle \; ; \; \bla\in\cP^\ell\}$ be the set of all charged $\ell$-partitions with charge $\bs$. 
This forms the basis of the so-called \textit{level $\ell$ Fock space} $\cF^\ell_\bs$ associated to $\bs$. Now, for any charge $\bs=(s_1,s_2,\ldots,s_\ell)\in\Z^\ell$ we can consider the sum of its components 
\[ |\bs|:=s_1+s_2+\ldots +s_\ell.\] 
For any $s\in\Z$, using \Cref{not_set} we denote by $\Z^\ell[s]$ the set of all $\bs\in\Z^\ell$ such that $|\bs|=s$. Consider the set of all charged $\ell$-partitions with charge in $\Z^\ell[s]$:
\[
\cU_s^\ell:=\bigsqcup_{\bs\in\Z^\ell[s]}\cU_\bs^\ell
\]
Uglov showed that 
the level-rank duality map $\LR$ in fact restricts to a bijection
\begin{equation}
\label{LR}
\begin{array}{cccc}
\LR :&\cU_s^\ell &  \longrightarrow & \cU_s^e,
\end{array}
\end{equation}
(again denoted by $\LR$)  \cite{Uglov1999}.
For example, in \Cref{exaLR} we observe that $s_1+s_2=-1=t_1+t_2+t_3$, as indicated by Uglov's theorem.

\begin{Rem}
It is not true that $\LR$ takes $\cU^\ell_\bs$ to $\cU^e_\bt$ for some fixed $\bt\in\Z^e$. For instance, fix $\bs=(0,-1)$ as in \Cref{exaLR}. If we take $\bla=(\emptyset,\emptyset)$ then, letting $\bA$ denote the abacus of $|\bla,\bs\rangle$, we find that the charge of $\LR(\bA)$ is $(0,0,-1)$, not $(0,-1,0)$ as in \Cref{exaLR}. If we take $\bla=\left((3^4,2^2),(5,3,2^2,1)\right)$ then we find that the charge of $\LR(\bA)$ is $(1,-1,-1)$. If we take $\bla=\left( (5,3,2^2,1^2),(6,5^2,4^2)\right)$ then we find that the charge of $\LR(\bA)$ is $(-4,3,0)$. So we emphasize that $\LR$ scatters $\cU^\ell_\bs$ across many $\cU^e_\bt$ for different charges $\bt$, while always preserving the sum of the charge (equal to $-1$ in this case).
\end{Rem}

\subsection{Cores and quotients of charged multipartitions}
\label{sec_cores_l}

We now introduce $e$-cores and $e$-quotients of charged multipartitions following the work of Jacon and Lecouvey \cite{JaconLecouvey2020,JaconLecouvey2021}.
In order to do this, we view the level-rank duality \Cref{LR} as an analogue of the Littlewood decomposition \Cref{Lit}. Fix $\ell,e\in\Z_{\geq 2}$. 

\medskip
First, the definition of the level-rank dual $\LR(\bA)$ of an $\ell$-abacus $\bA$ directly generalizes the definition of the $e$-abacus of a partition in \Cref{sec_Lit} from level $1$ to level $\ell$. In level $1$, we tile the one-row abacus of $\la$ by $(e\times 1)$ rectangles, then rotate the rectangles by $90$ degrees counterclockwise to obtain the $e$-abacus of $\la$. We then read off the corresponding $e$-partition to obtain the $e$-quotient of $\la$. Let us define the $e$-quotient of a charged $\ell$-partition $|\bla,\bs\rangle$ in the same way.

\begin{Def}\label{def_equot_ell}
Let $|\bla,\bs\rangle$ be a charged $\ell$-partition for some $\ell\geq 2$ and let $\bA$ be its abacus. We define the {\em $e$-quotient of $|\bla,\bs\rangle$} to be the unique $e$-partition $ \bmu$ such that $\LR(\bA)=\bA(\bmu,\bt)$ for some charge $\bt\in\Z^e$. We denote the $e$-quotient of $|\bla,\bs\rangle$ by $|\bla,\bs\rangle^{(e)}$.
\end{Def}

\medskip

Moreover, we can define the $e$-core of $|\bla,\bs\rangle$
analogously to how we computed the $e$-core $\la_{(e)}$ of a partition $\la$ in \Cref{sec_Lit} using the $e$-abacus of $\la$. \footnote{We note: we should not define the $e$-core to be simply a multipartition, but a charged multipartition. One simple reason for this is that we would like it to be possible for object to be its own $e$-core, and thus the $e$-core of an object should be again of the same type. There are other reasons, which should become apparent soon.}

\begin{Def}\label{def_ecore_ell} The $e$-core of a charged $\ell$-partition $|\bla,\bs\rangle$ is constructed by the following procedure. Let $\bA$ be the $\ell$-abacus of $|\bla,\bs\rangle$.
First compute the $e$-abacus $\LR(\bA)=:\bA(\bmu,\bt)$,
swipe all its beads to the left as much as possible to obtain the $e$-abacus $\bA':=\bA((\emptyset,\emptyset,\ldots,\emptyset),\bt)$,
then compute the $\ell$-abacus $\LR^{-1}(\bA')$. Then $\LR^{-1}(\bA')$ is the $\ell$-abacus of a charged $\ell$-partition $|\bla,\bs\rangle_{(e)}$
that we call the \textit{$e$-core of $|\bla,\bs\rangle$}. 
\end{Def}

Thanks to bijection \Cref{LR}, we thus obtain a bijection analogous to the Littlewood decomposition \Cref{Lit}:
\begin{equation}
\label{Lit_l}
\begin{array}{ccc}
\cU_s^\ell &   \overset{1:1}{\longleftrightarrow} &  \cC_{e,s}^\ell \times  \cP^e
\\
|\bla,\bs\rangle & \longleftrightarrow & \left( |\bla,\bs\rangle_{(e)} \,,\,  |\bla,\bs\rangle^{(e)} \right).
\end{array}
\end{equation}

If $|\bla,\bs\rangle_{(e)}=|\bla,\bs\rangle$, then we call $|\bla,\bs\rangle$ an \textit{$e$-core}. 
The set of all $e$-cores in $\cU^\ell_{\bs}$ for a fixed charge $\bs\in\Z^\ell$ is denoted by
$\cC_{e,\bs}^\ell$,
and for a fixed $s\in\Z$, we define
\[ 
\cC_{e,s}^\ell = \bigsqcup_{\bs\in\Z^\ell[s]} \cC_{e,\bs}^\ell,
\]
the set of all $e$-cores in $\cU^\ell_s$. 
By definition, $|\bla,\bs\rangle\in\cU^\ell$ is an $e$-core if and only if $\LR(\bA)$ is the abacus of the empty $e$-partition $(\emptyset,\emptyset,\ldots,\emptyset)$ with some charge, where $\bA$ is the $\ell$-abacus of $|\bla,\bs\rangle$. That is, charges in $\Z^e[s]$ parametrize $e$-cores in $\cU^\ell_s$.

\begin{Rem}
If we set
\begin{equation}\label{flotw_domain}
\sD_e^\ell = \{ (s_1,\ldots, s_\ell)\in\Z^{\ell} \mid s_1\leq\cdots\leq s_\ell \leq s_1+e\}, \end{equation}
and if $|\bla,\bs\rangle$ is an $e$-core,
then it is easy to see that $\bs\in \sD_e^\ell$ \cite[Proposition 2.14]{JaconLecouvey2021}.
\end{Rem}
\medskip

We can determine whether $|\bla,\bs\rangle\in\cU^\ell$ is an $e$-core by properties of its abacus (without computing its level-rank dual). This is elegantly described using the 
\textit{$e$-extended} abacus of $|\bla,\bs\rangle$. Consider the $(\ell+1)$-symbol
\[
\bS_+(\bla,\bs) = \left(\bS(\la^1,s_1),\bS(\la^2,s_2),\ldots, \bS(\la^\ell,s_\ell), \bS(\la^1,s_1+e)\right)
\]
and let $\bA_+(\bla,\bs)$ be the corresponding $(\ell+1)$-abacus. Notice that $\bA_+(\bla,\bs)$ is the abacus of $|\bla,\bs\rangle$ with an additional row  given by duplicating the bottom row of $\bA$, shifting it to the right by $e$ positions, and placing it on top of the abacus.

\begin{Exa}\label{ext_ab}
\label{abacus}
Let $\ell=2$ and $e=4$, and consider $|\bla,\bs\rangle \in\cU^2$ with $\bla=\left((1^2), (3,1^2)\right)$ and $\bs=(1,2)$.
Its $e$-extended abacus $\bA_+\left(\bla,\bs\rangle\right)$ is
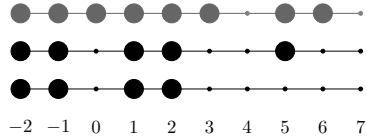
\begin{center}
\begin{tikzpicture}[scale=0.5, bb/.style={draw,circle,fill,minimum size=2.5mm,inner sep=0pt,outer sep=0pt}, gbb/.style={draw,circle,fill,minimum size=2.5mm,inner sep=0pt,outer sep=0pt, color=black!60}, wb/.style={draw,circle,fill,minimum size=0.5mm,inner sep=0pt,outer sep=0pt}, gwb/.style={draw,circle,fill,minimum size=0.5mm,inner sep=0pt,outer sep=0pt, color=black!60}]
%%%%%%%%%%%%%%%%%%%%%
%%%  On dessine les lignes de l'abaque et tous les emplacements
\foreach \j in {1,2}
{
\draw [line width=0.1mm] (-2,\j) -- (7,\j);
    \foreach \k in {-2,...,7}
    {
        \node [wb] at (\k,\j) {};
    }
}
\draw [line width=0.1mm, color=black!60] (-2,3) -- (7,3);
    \foreach \k in {-2,...,7}
    {
        \node [gwb, color=gray] at (\k,3) {};
    }
%%%%%%%%%%%%%%%%%%%%%
%%%  On place les billes
\foreach \k/\j in {-2/1,-1/1,1/1,2/1,    -2/2,-1/2,1/2,2/2,5/2   }
{
    \node [bb] at (\k,\j) {};
}
\foreach \k/\j in {-2/3,-1/3,0/3,1/3,2/3,3/3,5/3,6/3}
{
    \node [gbb] at (\k,\j) {};
}
%%%%%%%%%%%%%%%%%%%%%
%%%  On fait appraître la graduation
\foreach \k in {-2,...,7}
{
    \node [scale = 0.7] at (\k,0) {$\k$};
}
\end{tikzpicture}
\end{center}
Rows are indexed from bottom to top, in particular the grey (topmost) row corresponds to $\bS(\la^1,s_1+e)$.
\end{Exa}

By \cite[Proposition 2.16]{JaconLecouvey2021} the following criterion on $\bA_+(\bla,\bs)$ detects whether $|\bla,\bs\rangle$ is an $e$-core.
\begin{Lem}\label{charac_e_cores_l} 
The charged $\ell$-partition $|\bla,\bs\rangle$ is an $e$-core if and only if 
the rows of $\bA_+(\bla,\bs)$ are nested, meaning, $\bS(\la^1,s_1)\subseteq\bS(\la^2,s_2)\subseteq\ldots\subseteq\bS(\la^\ell,s_\ell)\subseteq\bS(\la^1,s_1+e)$.
\end{Lem}

\begin{Exa}
The extended abacus in \Cref{ext_ab} is nested, so the charged bipartition $|\left((1^2),(3,1^2)\right),(1,2)\rangle$ is a $4$-core.
\end{Exa}

\medskip

Now, recall from \Cref{sec_part} that the $e$-core $\la_{(e)}$ of a partition $\la$ can also be computed by one of the following equivalent procedures:
\begin{enumerate}
\item recursively moving beads $e$ steps to the left in $\bA(\la)$,
\item recursively sliding beads $1$ step to the left in the $e$-abacus of $\la$,
\item recursively sliding up beads in consecutive rows of the $e$-extended abacus $\bA_+(\la)$,
\item recursively removing $e$-hooks in the Young diagram of $\la$.
\end{enumerate}

Similarly, the $e$-core $|\bla,\bs\rangle_{(e)}$ of a charged multipartition $|\bla,\bs\rangle$ 
can be computed by one of the following three equivalent procedures:
\begin{enumerate}
\item[$(1')$] 
recursively sliding up beads in consecutive rows of $\bA(\bla,\bs)$, and moving beads $e$ spots to the left from row $\ell$ to row $1$ of $\bA(\bla,\bs)$,
\item[$(2')$] recursively sliding beads $1$ step to the left in $\LR\left(\bA(\bla,\bs)\right)$,
\item[$(3')$] 
recursively sliding up beads in consecutive rows of $\bA_+(\bla,\bs)$.
\end{enumerate}
We see that the algorithms in $(2)$ and $(2')$ coincide, and that $(1')$ and $(3')$ are natural analogues of $(1)$ and $(3)$.
We will see in the next section that $(1')$ can also be viewed as a ``hook removal'' procedure for special charges (by using a shift in the charge).

\begin{Exa}\label{exa_bead_slide}
Take $\ell=3$ and $e=6$. Let $\bla=\left( (2,1^3), (4,2,1), (3) \right)\in\cP^3$ and $\bs=(0,-1,-1)\in\Z^3$.
Computing the $e$-core $|\bla,\bs\rangle_{(e)}$ can be achieved by Procedure $(1')$ above.
The abacus of $|\bla,\bs\rangle$ is
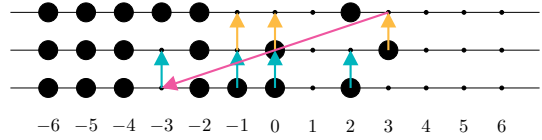
\begin{center}
\begin{tikzpicture}[scale=0.5, bb/.style={draw,circle,fill,minimum size=2.5mm,inner sep=0pt,outer sep=0pt}, wb/.style={draw,circle,fill,minimum size=0.5mm,inner sep=0pt,outer sep=0pt}]
%%%%%%%%%%%%%%%%%%%%%
%%%  On dessine les lignes de l'abaque et tous les emplacements
\foreach \j in {1,...,3}
{
\draw [line width=0.1mm] (-7,\j) -- (7,\j);
    \foreach \k in {-6,...,6}
    {
        \node [wb] at (\k,\j) {};
    }
}
%%%%%%%%%%%%%%%%%%%%%
%%%  On place les billes
\foreach \k/\j in { -6/1,-5/1,-4/1,-2/1,-1/1,0/1,2/1,    -6/2,-5/2,-4/2,-2/2,0/2,3/2,  -6/3,-5/3,-4/3,-3/3,-2/3,2/3 }
{
    \node [bb] at (\k,\j) {};
}
%%%%%%%%%%%%%%%%%%%%%
%%%  On fait appraître la graduation
\foreach \k in {-6,...,6}
{
    \node [scale = 0.7] at (\k,0) {$\k$};
}
% %%%%%%%%%%%%%%%%%%%%%
% %%%  On dessine les rectangles
% \draw [line width=0.1mm, color=gray] (-5.5,0.5) -- (6.5,0.5);
% \draw [line width=0.1mm, color=gray] (-5.5,3.5) -- (6.5,3.5);
% \foreach \k in {-5.5, 0.5, 6.5}
% {
% \draw [line width=0.1mm, color=gray] (\k,3.5) -- (\k,0.5);
% }

%%%%%%%%%% Les fleches
\draw[-{Triangle[length=2mm,width=2mm]}, thick, Aquamarine]
  (-3,1) -- (-3,2);
\draw[-{Triangle[length=2mm,width=2mm]}, thick, Aquamarine]
  (-1,1) -- (-1,2);
\draw[-{Triangle[length=2mm,width=2mm]}, thick, Aquamarine]
  (0,1) -- (0,2);
\draw[-{Triangle[length=2mm,width=2mm]}, thick, Aquamarine]
  (2,1) -- (2,2);
  %%%%
\draw[-{Triangle[length=2mm,width=2mm]}, thick, Dandelion]
  (3,2) -- (3,3);
\draw[-{Triangle[length=2mm,width=2mm]}, thick, Dandelion]
  (0,2) -- (0,3);
\draw[-{Triangle[length=2mm,width=2mm]}, thick, Dandelion]
  (-1,2) -- (-1,3);
\draw[-{Triangle[length=2mm,width=2mm]}, thick, VioletRed]
  (3,3) -- (-3,1);
\end{tikzpicture}
\end{center}
where we have drawn all possible (recursive) bead operations described in $(1')$ with colored arrows.
We obtain
$|\bla,\bs\rangle_{(6)} = |\left( (1),(2,1),(1) \right),(-3,0,1)\rangle$.

\end{Exa}

To conclude this section, we mention that Jacon and Lecouvey actually defined a core to be simply a multipartition instead of a charged multipartition.

\begin{Def}\cite{JaconLecouvey2020,JaconLecouvey2021}
Fix $\bs\in\Z^\ell$. A multipartition $\bla\in\cP^\ell$ is called an {\em $(e,\bs)$-core} if $|\bla,\bs\rangle$ is an $e$-core.
\end{Def}

This language is convenient when considering the set of all $e$-cores with fixed charge $\bs\in\Z^\ell$, which we have denoted above by $\cC^\ell_{e,\bs}$. Obviously $\cC^\ell_{e,\bs}$ is in bijection with the set of $(e,\bs)$-cores, and it makes sense to just think about this set as a subset of $\ell$-partitions. However, this perspective is less optimal when considering the $e$-core {\em of} a charged multipartition, as the charge usually changes when computing the $e$-core. 

\medskip

One immediate consequence of \Cref{charac_e_cores_l} worth stating for its implications for positivity questions is a divisibility result for $(e,\bs)$-cores identical to the analogous statement for $e$-core partitions:

\begin{Lem}\label{ecore_div}
Suppose $e\mid f$, then $\bla$ is an $(e,\bs)$-core implies that $\bla$ is an $(f,\bs)$-core.
\end{Lem}

\subsection{Spetsial cores and quotients of charged multipartitions}\label{spetsial combinat}

In this section, we introduce
Malle's spetsial cores and quotients for level $\ell$ symbols (i.e. charged $\ell$-partitions and their abaci). We then establish a relationship between Malle's combinatorics and the ``Uglovian" combinatorics from the previous section. 

\medskip

By way of introduction, let us quickly recall the notion of a cocore of a level $2$ symbol, which is a very classical concept among all the higher-level generalizations of partitions and their $e$-cores.
 Consider a level $2$ symbol $\bS=(\bS^{(1)},\bS^{(2)})\subset \Z\times\{1,2\}$ and let $\bA$ be its $2$-abacus. For $d\in\Z_{\geq 1}$, 
 a {\em $d$-cohook} in $\bS$ is some $\beta\in\bS^{(j)}$, $j\in\{1,2\}$, such that $\beta-d\notin\bS^{(j+1)}$ (where $j+1$ is taken mod $2$). Removing the $d$-cohook from $\bS$ consists in removing $\be$ from $\bS^{(j)}$ and including $\be-d$ in $\bS^{(j+1)}$. This is done on the abacus $\bA$ by moving the bead (corresponding to $\beta$) $d$ positions to the left and to the opposite row, as illustrated below for $d=3$:
 
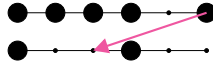
\begin{center}
\begin{tikzpicture}[scale=0.5, bb/.style={draw,circle,fill,minimum size=2.5mm,inner sep=0pt,outer sep=0pt}, wb/.style={draw,circle,fill,minimum size=0.5mm,inner sep=0pt,outer sep=0pt}]
%%%%%%%%%%%%%%%%%%%%%
%%%  On dessine les lignes de l'abaque et tous les emplacements
\foreach \j in {1,2}
{
\draw [line width=0.1mm] (0,\j) -- (5,\j);
    \foreach \k in {0,...,5}
    {
        \node [wb] at (\k,\j) {};
    }
}
%%%%%%%%%%%%%%%%%%%%%
%%%  On place les billes
\foreach \k/\j in {0/1,3/1,    0/2,1/2,2/2,3/2,5/2 }
{
    \node [bb] at (\k,\j) {};
}
%\TikZ{[scale=.5]
%\draw
%(5,2)node[fill,red,circle,inner sep=2pt]{}
%(5,1)node[fill,circle,inner sep=.5pt]{}
%(4,2)node[fill,circle,inner sep=.5pt]{}
%(4,1)node[fill,circle,inner sep=.5pt]{}
%(3,2)node[fill,circle,inner sep=2pt]{}
%(3,1)node[fill,circle,inner sep=2pt]{}
%(2,2)node[fill,circle,inner sep=2pt]{}
%(2,1)node[fill,red,circle,inner sep=.5pt]{}
%(1,2)node[fill,circle,inner sep=2pt]{}
%(1,1)node[fill,circle,inner sep=.5pt]{}
%(0,2)node[fill,circle,inner sep=2pt]{}
%(0,1)node[fill,circle,inner sep=2pt]{}
%;
\draw[-{Triangle[length=2mm,width=2mm]}, thick, VioletRed]
  (5,2) -- (2,1);
\end{tikzpicture}
\end{center}
We say that $\bS$ is a {\em$d$-cocore} if no $d$-cohooks can be removed from $\bS$.
 Given any level $2$ symbol $\bS$, {\em the $d$-cocore of $\bS$} is the unique level $2$ symbol obtained by recursively removing all $d$-cohooks from $\bS$. In the previous example, the abacus of the $3$-cocore is obtained after the removal of the single $3$-cohook:
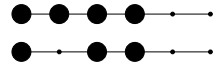
\begin{center}
\begin{tikzpicture}[scale=0.5, bb/.style={draw,circle,fill,minimum size=2.5mm,inner sep=0pt,outer sep=0pt}, wb/.style={draw,circle,fill,minimum size=0.5mm,inner sep=0pt,outer sep=0pt}]
%%%%%%%%%%%%%%%%%%%%%
%%%  On dessine les lignes de l'abaque et tous les emplacements
\foreach \j in {1,2}
{
\draw [line width=0.1mm] (0,\j) -- (5,\j);
    \foreach \k in {0,...,5}
    {
        \node [wb] at (\k,\j) {};
    }
}
%%%%%%%%%%%%%%%%%%%%%
%%%  On place les billes
\foreach \k/\j in {0/1,2/1,3/1,    0/2,1/2,2/2,3/2 }
{
    \node [bb] at (\k,\j) {};
}
\end{tikzpicture}
\end{center}
Observe that the case $d=1$ is non-trivial, unlike in the case of partitions (removing $1$-cohooks does not amount to simply removing all boxes from the associated bipartition one by one).
 \medskip
 
 Malle generalized the combinatorics of cohooks and cocores to any level $\ell>1$ in \cite{Malle1995} as part of the description of unipotent character degrees for Spetses of the complex reflection groups $G(\ell,1,n)$. In fact, Malle essentially constructed a spetsial version of the level-rank duality (he did not use this language).
  
 \begin{Not}
 We will use Greek letters $\bsig,\btau,\ldots$ for a charge in $\Z^\ell$ whenever working in the spetsial framework, while we will use Roman letters $\bs,\bt,\ldots$ for a charge in $\Z^\ell$ when working in the Uglovian framework.
 \end{Not}

 We now explain Malle's spetsial level-rank duality by describing it on the abacus of a charged multipartition. Fix $d\in\Z_{\geq 1}$, and fix $\ell\in\Z_{\geq 2}$.
 Let $\bsig\in\Z^\ell$ be a charge and $\bla\in\cP^\ell$. Let $\bA$ be the $\ell$-abacus of $|\bla,\bsig\rangle$.
 
 \medskip
 
 {\em Step 0.} 
 Tile $\bA$ by ``big rectangles'' of width $d\ell$ and height $\ell$, in such a way that $0$ is in the rightmost column of its rectangle.\footnote{We have shifted Malle's convention; his big rectangles are aligned so that $0$ is in the leftmost column of its rectangle. This only cyclically permutes the resulting core and quotient. As his symbols are actually only considered up to a cyclic permutation, this difference in convention is irrelevant.} Then, subdivide each big rectangle into ``$d$-dominoes:'' small rectangles of width $d$ and height $1$. Thus, each big rectangle is tiled by an $\ell\times\ell$ grid of horizontal $d$-dominoes. Within each big rectangle, there are $\ell$ domino-columns, each consisting of $\ell$ stacked $d$-dominoes. Let us number the domino-columns from left to right by $1$ to up to $\ell$.
 
 \medskip
 
  {\em Step 1.} For each $j=1,\ldots,\ell$, and within every big rectangle, rotate the $j$'th domino-column cyclically in the upwards direction by $j-1$. (So, if a domino falls off the top of its domino-column, it reappears at the bottom of the same domino-column.)
  
  \medskip
  
  {\em Step 2.} Rotate every $d$-domino counterclockwise $90$ degrees. This transforms each big rectangle from an $\ell d \times \ell$ rectangle into an $\ell\times\ell d$ rectangle, which is tiled by vertical $d$-dominoes (small $1\times d$ rectangles). The position of $0$ in the new abacus is given by the column consisting of those rotated dominoes that contained the position $0$ in $\bA$. 
  
  \medskip
  
We will call the resulting $d\ell$-abacus the {\em spetsial level-rank dual} of $\bA$ and denote it $\sLR(\bA)$.   The inverse map $\sLR^{-1}$ is (obviously) given by rotating the vertical $d$-dominoes $90$ degrees clockwise, then applying the cyclic rotation $(\down 0\;,\down1\;,\ldots, \down(\ell-1))$ to the domino-columns.
  
  \begin{Exa}\label{exa_sLR}
  Let $\ell=3$ and $d=2$. We will find $\sLR(\bA)$ for the abacus $\bA$ below:
  
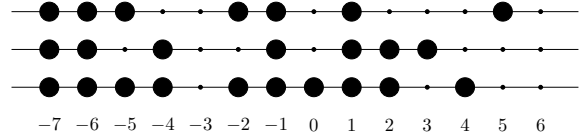
\begin{center}
\begin{tikzpicture}[scale=0.5, bb/.style={draw,circle,fill,minimum size=2.5mm,inner sep=0pt,outer sep=0pt}, wb/.style={draw,circle,fill,minimum size=0.5mm,inner sep=0pt,outer sep=0pt}]
%%%%%%%%%%%%%%%%%%%%%
%%%  On dessine les lignes de l'abaque et tous les emplacements
\foreach \j in {1,2,3}
{
\draw [line width=0.1mm] (-8,\j) -- (7,\j);
}
%   \foreach \k in {-7,...,6}
%    {
 %       \node [wb] at (\k,\j) {};
%    }
\foreach \k/\j in {-3/3,3/3,4/3,    0/2,5/2,6/2,    }
{
    \node [wb] at (\k,\j) {};
}
\foreach \k/\j in {0/3,6/3,    -5/2,    -3/1,3/1}
{
    \node [wb] at (\k,\j) {};
}
\foreach \k/\j in {-4/3,2/3,   -3/2,-2/2,4/2,   5/1,6/1 }
{
    \node [wb] at (\k,\j) {};
}
%%%%%%%%%%%%%%%%%%%%%
%%%  On place les billes
\foreach \k/\j in {-2/3,    -7/2,-6/2,-1/2,    -5/1,-4/1,1/1,2/1 }
{
    \node [bb] at (\k,\j) {};
}
\foreach \k/\j in {-7/3,-6/3,-1/3,5/3,    -4/2,1/2,2/2,   -2/1,4/1 }
{
    \node [bb] at (\k,\j) {};
}
\foreach \k/\j in {-5/3,1/3,    3/2,    -7/1,-6/1,-1/1,0/1 }
{
    \node [bb] at (\k,\j) {};
}
%%%%%%%%%%%%%%%%%%%%%
%%%  On fait appraître la graduation
\foreach \k in {-7,...,6}
{
    \node [scale = 0.7] at (\k,0) {$\k$};
}
%%%%%%%%%%%%%%%%%%%%%
%%%  On dessine les rectangles

\end{tikzpicture}
\end{center}
  During the operations, we will not draw the abacus runners.
  
  \medskip
  
  {\em Step 0.} We mark the big rectangles of width $6$ and height $3$, subdividing each into nine $2$-dominoes. The three colors along the three cyclic, descending domino-diagonals are a visual aid for the operations to carry out next.
  
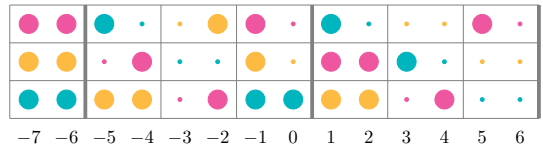
\begin{center}
\begin{tikzpicture}[scale=0.5, bb/.style={draw,circle,fill,minimum size=2.5mm,inner sep=0pt,outer sep=0pt}, wb/.style={draw,circle,fill,minimum size=0.5mm,inner sep=0pt,outer sep=0pt}]
%%%%%%%%%%%%%%%%%%%%%
%%%  On dessine les lignes de l'abaque et tous les emplacements
%\foreach \j in {1,2,3}
%{
%\draw [line width=0.1mm] (-8,\j) -- (7,\j);
%}
%   \foreach \k in {-7,...,6}
%    {
 %       \node [wb] at (\k,\j) {};
%    }
\foreach \k/\j in {-3/3,3/3,4/3,    0/2,5/2,6/2,    }
{
    \node [wb,Dandelion] at (\k,\j) {};
}
\foreach \k/\j in {0/3,6/3,    -5/2,    -3/1,3/1}
{
    \node [wb,VioletRed] at (\k,\j) {};
}
\foreach \k/\j in {-4/3,2/3,   -3/2,-2/2,4/2,   5/1,6/1 }
{
    \node [wb,Aquamarine] at (\k,\j) {};
}
%%%%%%%%%%%%%%%%%%%%%
%%%  On place les billes
\foreach \k/\j in {-2/3,    -7/2,-6/2,-1/2,    -5/1,-4/1,1/1,2/1 }
{
    \node [bb,Dandelion] at (\k,\j) {};
}
\foreach \k/\j in {-7/3,-6/3,-1/3,5/3,    -4/2,1/2,2/2,   -2/1,4/1 }
{
    \node [bb,VioletRed] at (\k,\j) {};
}
\foreach \k/\j in {-5/3,1/3,    3/2,    -7/1,-6/1,-1/1,0/1 }
{
    \node [bb,Aquamarine] at (\k,\j) {};
}
%%%%%%%%%%%%%%%%%%%%%
%%%  On fait appraître la graduation
\foreach \k in {-7,...,6}
{
    \node [scale = 0.7] at (\k,0) {$\k$};
}
%%%%%%%%%%%%%%%%%%%%%
%%%  On dessine les rectangles
\draw [line width=0.1mm, color=gray] (-7.5,0.5) -- (6.5,0.5);
\draw [line width=0.1mm, color=gray] (-7.5,1.5) -- (6.5,1.5);
\draw [line width=0.1mm, color=gray] (-7.5,2.5) -- (6.5,2.5);
\draw [line width=0.1mm, color=gray] (-7.5,3.5) -- (6.5,3.5);
\foreach \k in {-7.5,-5.5,-3.5, -1.5, 0.5, 2.5,4.5,6.5}
{
\draw [line width=0.1mm, color=gray] (\k,3.5) -- (\k,0.5);
}
\foreach \k in {-5.5, 0.5, 6.5}
{
\draw [line width=0.5mm, color=gray] (\k,3.5) -- (\k,0.5);
}
\end{tikzpicture}
\end{center}
  {\em Step 1.} We perform the cyclic rotation $(\up0\;, \up1\;,\up2)$ of the domino-columns in each big rectangle. This amounts to aligning the colors along rows while fixing the first domino-column in each big rectangle.
    
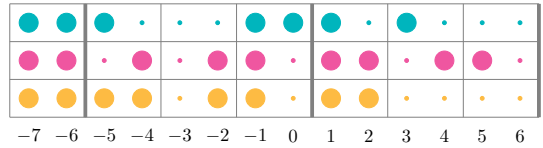
\begin{center}
\begin{tikzpicture}[scale=0.5, bb/.style={draw,circle,fill,minimum size=2.5mm,inner sep=0pt,outer sep=0pt}, wb/.style={draw,circle,fill,minimum size=0.5mm,inner sep=0pt,outer sep=0pt}]
%%%%%%%%%%%%%%%%%%%%%
%%%  On dessine les lignes de l'abaque et tous les emplacements
%\foreach \j in {1,2,3}
%{
%\draw [line width=0.1mm] (-8,\j) -- (7,\j);
%}
%   \foreach \k in {-7,...,6}
%    {
 %       \node [wb] at (\k,\j) {};
%    }
\foreach \k/\j in {-3/1,3/1,4/1,    0/1,5/1,6/1,    }
{
    \node [wb,Dandelion] at (\k,\j) {};
}
\foreach \k/\j in {0/2,6/2,    -5/2,    -3/2,3/2}
{
    \node [wb,VioletRed] at (\k,\j) {};
}
\foreach \k/\j in {-4/3,2/3,   -3/3,-2/3,4/3,   5/3,6/3 }
{
    \node [wb,Aquamarine] at (\k,\j) {};
}
%%%%%%%%%%%%%%%%%%%%%
%%%  On place les billes
\foreach \k/\j in {-2/1,    -7/1,-6/1,-1/1,    -5/1,-4/1,1/1,2/1 }
{
    \node [bb,Dandelion] at (\k,\j) {};
}
\foreach \k/\j in {-7/2,-6/2,-1/2,5/2,    -4/2,1/2,2/2,   -2/2,4/2 }
{
    \node [bb,VioletRed] at (\k,\j) {};
}
\foreach \k/\j in {-5/3,1/3,    3/3,    -7/3,-6/3,-1/3,0/3 }
{
    \node [bb,Aquamarine] at (\k,\j) {};
}
%%%%%%%%%%%%%%%%%%%%%
%%%  On fait appraître la graduation
\foreach \k in {-7,...,6}
{
    \node [scale = 0.7] at (\k,0) {$\k$};
}
%%%%%%%%%%%%%%%%%%%%%
%%%  On dessine les rectangles
\draw [line width=0.1mm, color=gray] (-7.5,0.5) -- (6.5,0.5);
\draw [line width=0.1mm, color=gray] (-7.5,1.5) -- (6.5,1.5);
\draw [line width=0.1mm, color=gray] (-7.5,2.5) -- (6.5,2.5);
\draw [line width=0.1mm, color=gray] (-7.5,3.5) -- (6.5,3.5);
\foreach \k in {-7.5,-5.5,-3.5, -1.5, 0.5, 2.5,4.5,6.5}
{
\draw [line width=0.1mm, color=gray] (\k,3.5) -- (\k,0.5);
}
\foreach \k in {-5.5, 0.5, 6.5}
{
\draw [line width=0.5mm, color=gray] (\k,3.5) -- (\k,0.5);
}
\end{tikzpicture}
\end{center}
  {\em Step 2.} We rotate each domino $90$ degrees counterclockwise to obtain $\sLR(\bA)$.
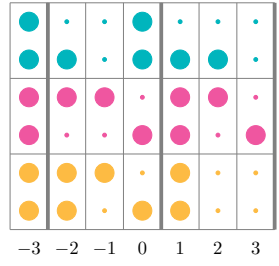
  \begin{center}
\begin{tikzpicture}[scale=0.5, bb/.style={draw,circle,fill,minimum size=2.5mm,inner sep=0pt,outer sep=0pt}, wb/.style={draw,circle,fill,minimum size=0.5mm,inner sep=0pt,outer sep=0pt}]
%%%%%%%%%%%%%%%%%%%%%
%%%  On dessine les lignes de l'abaque et tous les emplacements
%\foreach \j in {1,2,3,4,5,6}
%{
%\draw [line width=0.1mm] (-4,\j) -- (4,\j);
%}
%   \foreach \k in {-7,...,6}
%    {
 %       \node [wb] at (\k,\j) {};
%    }
\foreach \k/\j in {-1/1,2/1,3/1,    0/2,2/2,3/2}
{
    \node [wb,Dandelion] at (\k,\j) {};
}
\foreach \k/\j in {-2/3,-1/3,2/3,    0/4,3/4}
{
    \node [wb,VioletRed] at (\k,\j) {};
}
\foreach \k/\j in {-1/5,3/5,   -2/6,-1/6,1/6,2/6,3/6 }
{
    \node [wb,Aquamarine] at (\k,\j) {};
}
%%%%%%%%%%%%%%%%%%%%%
%%%  On place les billes
\foreach \k/\j in {-3/1,-2/1,0/1,1/1,    -3/2,-2/2,-1/2,1/2}
{
    \node [bb,Dandelion] at (\k,\j) {};
}
\foreach \k/\j in {-3/3,0/3,1/3,3/3,    -3/4,-2/4,-1/4, 1/4,2/4}
{
    \node [bb,VioletRed] at (\k,\j) {};
}
\foreach \k/\j in {-3/5,-2/5,0/5,1/5,2/5,   -3/6,0/6}
{
    \node [bb,Aquamarine] at (\k,\j) {};
}
%%%%%%%%%%%%%%%%%%%%%
%%%  On fait appraître la graduation
\foreach \k in {-3,...,3}
{
    \node [scale = 0.7] at (\k,0) {$\k$};
}
%%%%%%%%%%%%%%%%%%%%%
%%%  On dessine les rectangles
\draw [line width=0.1mm, color=gray] (-3.5,0.5) -- (3.5,0.5);
\draw [line width=0.1mm, color=gray] (-3.5,2.5) -- (3.5,2.5);
\draw [line width=0.1mm, color=gray] (-3.5,4.5) -- (3.5,4.5);
\draw [line width=0.1mm, color=gray] (-3.5,6.5) -- (3.5,6.5);
\foreach \k in {-3.5,-2.5,-1.5, -0.5,0.5, 1.5, 2.5,3.5}
{
\draw [line width=0.1mm, color=gray] (\k,6.5) -- (\k,0.5);
}
\foreach \k in {-2.5, 0.5, 3.5}
{
\draw [line width=0.5mm, color=gray] (\k,6.5) -- (\k,0.5);
}
\end{tikzpicture}
\end{center}
  Thus $\sLR(\bA)$ is given by the $6$-abacus (now erasing the rectangles and drawing the runners):
  
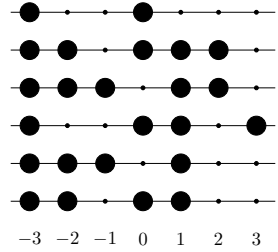
\begin{center}
\begin{tikzpicture}[scale=0.5, bb/.style={draw,circle,fill,minimum size=2.5mm,inner sep=0pt,outer sep=0pt}, wb/.style={draw,circle,fill,minimum size=0.5mm,inner sep=0pt,outer sep=0pt}]
%%%%%%%%%%%%%%%%%%%%%
%%%  On dessine les lignes de l'abaque et tous les emplacements
\foreach \j in {1,2,3,4,5,6}
{
\draw [line width=0.1mm] (-3.5,\j) -- (3.5,\j);
}
%   \foreach \k in {-7,...,6}
%    {
 %       \node [wb] at (\k,\j) {};
%    }
\foreach \k/\j in {-1/1,2/1,3/1,    0/2,2/2,3/2}
{
    \node [wb] at (\k,\j) {};
}
\foreach \k/\j in {-2/3,-1/3,2/3,    0/4,3/4}
{
    \node [wb] at (\k,\j) {};
}
\foreach \k/\j in {-1/5,3/5,   -2/6,-1/6,1/6,2/6,3/6 }
{
    \node [wb] at (\k,\j) {};
}
%%%%%%%%%%%%%%%%%%%%%
%%%  On place les billes
\foreach \k/\j in {-3/1,-2/1,0/1,1/1,    -3/2,-2/2,-1/2,1/2}
{
    \node [bb] at (\k,\j) {};
}
\foreach \k/\j in {-3/3,0/3,1/3,3/3,    -3/4,-2/4,-1/4, 1/4,2/4}
{
    \node [bb] at (\k,\j) {};
}
\foreach \k/\j in {-3/5,-2/5,0/5,1/5,2/5,   -3/6,0/6}
{
    \node [bb] at (\k,\j) {};
}
%%%%%%%%%%%%%%%%%%%%%
%%%  On fait appraître la graduation
\foreach \k in {-3,...,3}
{
    \node [scale = 0.7] at (\k,0) {$\k$};
}
%%%%%%%%%%%%%%%%%%%%%
%%%  On dessine les rectangles
\end{tikzpicture}
\end{center}
  
  \end{Exa}

   From the spetsial level-rank duality, the definitions of spetsial cores and quotients fall out in parallel to \Cref{def_ecore_ell,def_equot_ell}. In both definitions below, let $|\bla,\bsig\rangle$ be a charged $\ell$-partition for some $\ell\geq 2$ and let $\bA$ be its abacus. Let $d\in\Z_{\geq 1}$. 
  \begin{Def}\label{def_spetsialdquot}
We define the {\em spetsial $d$-quotient} $|\bla,\bsig\rangle^{[d]}$ of $|\bla,\bsig\rangle$ to be the unique $\ell d$-partition $ \bmu$ such that $\sLR(\bA)=\bA(\bmu,\btau)$ for some charge $\btau\in\Z^{\ell d}$. 
\end{Def}
 
 \begin{Def}\label{def_spetsialdcore} The {\em spetsial $d$-core }$|\bla,\bsig\rangle_{[d]}$ of $|\bla,\bsig\rangle$ is constructed by the following procedure. 
First compute the $\ell d$-abacus $\sLR(\bA)=:\bA(\bmu,\btau)$,
swipe all its beads to the left as much as possible to obtain the $\ell d$-abacus $\bA':=\bA(\bemp,\btau)$,
then compute the $\ell$-abacus $\sLR^{-1}(\bA')$. Then $\sLR^{-1}(\bA')=:\bA(|\bla,\bsig\rangle_{[d]})$.
\end{Def}

\begin{Exa}\label{exa_spetsialcorequotient}
In \Cref{exa_sLR}, the abacus $\bA$ is $\bA(\bla,\bsig)$ for $\bla=\left((2,1^5),(4^3,3,1),(6,3,2^2)\right)\in\cP^3$ and $\bsig=(2,-1,-1)\in\Z^3$. From $\sLR(\bA)$ we see that $|\bla,\bsig\rangle^{[2]}=\left((1^2),(1),(3,2^2),(1^2),(1^3),(2)\right)\in\cP^6$ is the spetsial $2$-quotient. We likewise find that the charge of $\sLR(\bA)$ is $\btau=(0,0,0,1,1,-2)\in\Z^6$. To find the spetsial $2$-core $|\bla,\bsig\rangle_{[2]}$, we apply $\sLR^{-1}$ to $\bA(\bemp,\btau)$:
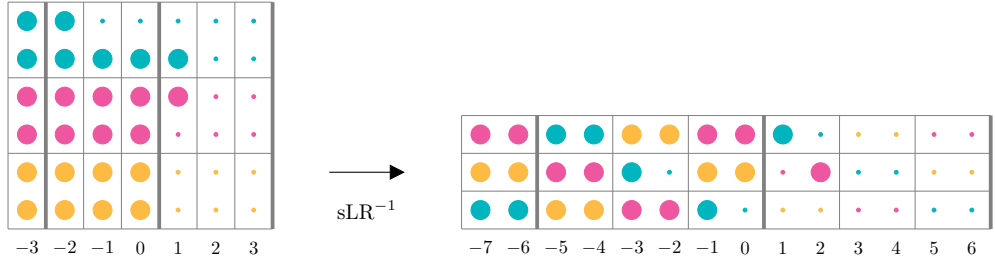
\begin{center}
\begin{tikzpicture}[scale=0.5, bb/.style={draw,circle,fill,minimum size=2.5mm,inner sep=0pt,outer sep=0pt}, wb/.style={draw,circle,fill,minimum size=0.5mm,inner sep=0pt,outer sep=0pt}]
%%%%%%%%%%%%%%%%%%%%%
%%%  On dessine les lignes de l'abaque et tous les emplacements
%\foreach \j in {1,2,3,4,5,6}
%{
%\draw [line width=0.1mm] (-4,\j) -- (4,\j);
%}
%   \foreach \k in {-7,...,6}
%    {
 %       \node [wb] at (\k,\j) {};
%    }
\foreach \k/\j in {1/1,2/1,3/1,    1/2,2/2,3/2}
{
    \node [wb,Dandelion] at (\k,\j) {};
}
\foreach \k/\j in {1/3,2/3,3/3,    2/4,3/4}
{
    \node [wb,VioletRed] at (\k,\j) {};
}
\foreach \k/\j in {2/5,3/5,   -1/6,0/6,1/6,2/6,3/6 }
{
    \node [wb,Aquamarine] at (\k,\j) {};
}
%%%%%%%%%%%%%%%%%%%%%
%%%  On place les billes
\foreach \k/\j in {-3/1,-2/1,-1/1,0/1,    -3/2,-2/2,-1/2,0/2}
{
    \node [bb,Dandelion] at (\k,\j) {};
}
\foreach \k/\j in {-3/3,-2/3,-1/3,0/3,    -3/4,-2/4,-1/4, 0/4,1/4}
{
    \node [bb,VioletRed] at (\k,\j) {};
}
\foreach \k/\j in {-3/5,-2/5,-1/5,0/5,1/5,   -3/6,-2/6}
{
    \node [bb,Aquamarine] at (\k,\j) {};
}
%%%%%%%%%%%%%%%%%%%%%
%%%  On fait appraître la graduation
\foreach \k in {-3,...,3}
{
    \node [scale = 0.7] at (\k,0) {$\k$};
}
%%%%%%%%%%%%%%%%%%%%%
%%%  On dessine les rectangles
\draw [line width=0.1mm, color=gray] (-3.5,0.5) -- (3.5,0.5);
\draw [line width=0.1mm, color=gray] (-3.5,2.5) -- (3.5,2.5);
\draw [line width=0.1mm, color=gray] (-3.5,4.5) -- (3.5,4.5);
\draw [line width=0.1mm, color=gray] (-3.5,6.5) -- (3.5,6.5);
\foreach \k in {-3.5,-2.5,-1.5, -0.5,0.5, 1.5, 2.5,3.5}
{
\draw [line width=0.1mm, color=gray] (\k,6.5) -- (\k,0.5);
}
\foreach \k in {-2.5, 0.5, 3.5}
{
\draw [line width=0.5mm, color=gray] (\k,6.5) -- (\k,0.5);
}
\draw[-{Triangle[length=2mm,width=2mm]}]
  (5,2) -- (7,2)
;
\node [scale=0.8] at (6,1){$\sLR^{-1}$};
%%%%%%%%%
%% next abacus%%
\foreach \k/\j in {3/3,4/3,    5/2,6/2,   1/1,2/1 }
{
    \node [wb,Dandelion] at (\k+16,\j) {};
}
\foreach \k/\j in {5/3,6/3,    1/2,    3/1,4/1}
{
    \node [wb,VioletRed] at (\k+16,\j) {};
}
\foreach \k/\j in {2/3,   -2/2,3/2,4/2,   0/1,5/1,6/1 }
{
    \node [wb,Aquamarine] at (\k+16,\j) {};
}
%%%%%%%%%%%%%%%%%%%%%
%%%  On place les billes
\foreach \k/\j in {-3/3,-2/3,    -7/2,-6/2,-1/2,0/2,    -5/1,-4/1}
{
    \node [bb,Dandelion] at (\k+16,\j) {};
}
\foreach \k/\j in {-7/3,-6/3,-1/3,0/3,   -5/2,-4/2,2/2,   -3/1,-2/1 }
{
    \node [bb,VioletRed] at (\k+16,\j) {};
}
\foreach \k/\j in {-5/3,-4/3,1/3,    -3/2,    -7/1,-6/1,-1/1 }
{
    \node [bb,Aquamarine] at (\k+16,\j) {};
}
%%%%%%%%%%%%%%%%%%%%%
%%%  On fait appraître la graduation
\foreach \k in {-7,...,6}
{
    \node [scale = 0.7] at (\k+16,0) {$\k$};
}
%%%%%%%%%%%%%%%%%%%%%
%%%  On dessine les rectangles
\draw [line width=0.1mm, color=gray] (8.5,0.5) -- (22.5,0.5);
\draw [line width=0.1mm, color=gray] (8.5,1.5) -- (22.5,1.5);
\draw [line width=0.1mm, color=gray] (8.5,2.5) -- (22.5,2.5);
\draw [line width=0.1mm, color=gray] (8.5,3.5) -- (22.5,3.5);
\foreach \k in {-7.5,-5.5,-3.5, -1.5, 0.5, 2.5,4.5,6.5}
{
\draw [line width=0.1mm, color=gray] (\k+16,3.5) -- (\k+16,0.5);
}
\foreach \k in {-5.5, 0.5, 6.5}
{
\draw [line width=0.5mm, color=gray] (\k+16,3.5) -- (\k+16,0.5);
}
;
\end{tikzpicture}
\end{center}
Therefore $|\bla,\bsig\rangle_{[2]}=|\left(\emptyset,(2,1^2),\emptyset\right),(-1,0,1)\rangle$.
\end{Exa}

We now make explicit the precise relationship between ordinary level-rank duality $\LR$ and spetsial level-rank duality $\sLR$. To go between them, we introduce the following shift of the charge. Set 
\[
\brho=(0,1,2,\ldots,\ell-1)\in\Z^\ell
\]
and for $d\in\Z_{\geq 1}$, consider $d\brho=(0,d,2d,\ldots,(\ell-1)d)$. Then consider the $d$-dilated $\brho$-shift of the charge as a map on charged $\ell$-partitions:
\begin{align*}
\mathsf{t}_{d\brho}&:\cU^\ell\lra\cU^\ell\\
&|\bla,\bsig\rangle\mapsto |\bla,\bsig+d\brho\rangle
\end{align*}
Next, for any finite sequence $\ba=(a_1,\ldots,a_\ell)$, write $\ba_{\mathrm{rev}}= (a_\ell,\ldots,a_1)$ for $\ba$ in the reverse order. Set $e=\ell d$. Define $\brho^d:=(0,0,\ldots,0,1,1,\ldots,1,\ldots,\ell-1,\ell-1,\ldots,\ell-1)\in\Z^e$, i.e. each entry of $\brho$ is repeated $d$ times. Then consider the shift of the charge by $\brho^d_{\mathrm{rev}}$ on charged $e$-partitions:
\begin{align*}
\mathsf{t}_{d\brho}&:\cU^e\lra\cU^e\\
&|\bmu,\btau\rangle\mapsto |\bmu,\btau+\brho^d_{\mathrm{rev}}\rangle
\end{align*}
\begin{Prop}\label{prop_lr_vs_slr}
Let $\ell\in\Z_{\geq 2}$ and $d\in\Z_{\geq 1}$, and set $e=\ell d$.
The following diagram commutes.
\[
\begin{tikzcd}
 \cU^\ell \arrow[r, "\sLR"] \arrow[d, "\mathsf{t}_{d\brho}"]
&  \cU^e  \arrow[d, "\mathsf{t}_{\brho^d_{\mathrm{rev}}}"] \\
  \arrow[r, "\LR"]
  \cU^\ell   &\cU^e
  \end{tikzcd}
\]
\end{Prop}
\begin{proof}
All maps preserve the tiling by $d$-dominoes, acting by translations or cyclic permutations or $90$ degree rotations of the dominoes. It thus suffices to keep track of where a given $d$-domino goes under the maps $\LR\circ\mathsf{t}_{d\brho}$ and $\mathsf{t}_{\brho^d_{\mathrm{rev}}}\circ\sLR$, and check it is the same under both maps. This reduces to verifying the $d=1$ case, then substituting $d$-dominoes for the beads in the $d=1$ case (and remembering that under both maps, the dominoes are rotated $90$ degrees counterclockwise exactly once).

\medskip

We now check the statement for $d=1$. 
We follow where $(\be,j)\in\Z\times\{1,\ldots,\ell\}$ goes under the maps $\mathsf{t}_{\brho_{\mathrm{rev}}}\circ\sLR$ and $\LR\circ\mathsf{t}_{\brho}$. Write $\be=k\ell+i$ for some $i\in\{1,2,\ldots,\ell\}$.
First, $\mathsf{t}_\brho(\be,j)=(\be+j-1,j)$. We have 
\begin{align*}
\LR\left(\mathsf{t}_\brho(\be,j)\right)=\LR(\be+j-1,j) &=\left(\ell\left( \left\lfloor \frac{\be+j-2}{\ell} \right\rfloor +1\right) -j+1, (\be+j-2)\bmod\ell+1\right)\\
&=\left(\ell\left( \left\lfloor \frac{k\ell+i+j-2}{\ell} \right\rfloor +1\right) -j+1, (k\ell+i+j-2)\bmod\ell+1\right)\\
&=\left((k+1)\ell+\ell\left( \left\lfloor \frac{i+j-2}{\ell} \right\rfloor\right) -j+1, (i+j-2)\bmod\ell+1\right).
\end{align*}
On the other hand, the map $\sLR$ rotates column $k\ell+i$ by $i-1$, sending $(\be,j)$ to $(\be, (j+i-2)\bmod \ell+1)$. The map $\mathsf{t}_{\brho_{\mathrm{rev}}}$ sends $(\be,j)\in\Z\times\{1,\ldots,\ell\}$ to $(\be+\ell-j,j)$. If $0\leq i+j-2\leq \ell-1$ then $i-\left((j+i-2\bmod \ell)+1\right)=-j+1$, while if $\ell\leq i+j-2\leq 2\ell-2$ then $i-\left((j+i-2\bmod \ell)+1\right)=\ell-j+1$, whence 
\begin{align*}
\mathsf{t}_{\brho_{\mathrm{rev}}}\left(\sLR(\be,j)\right)& =(\be+\ell-1-(j+i-2)\bmod \ell,(j+i-2)\bmod \ell+1)\\
&=((k+1)\ell +i-((j+i-2)\bmod \ell+1),(j+i-2)\bmod \ell+1)\\
&= \left((k+1)\ell+\ell\left( \left\lfloor \frac{i+j-2}{\ell} \right\rfloor\right) -j+1, (i+j-2)\bmod\ell+1\right)\\
&=\LR\left(\mathsf{t}_\brho(\be,j)\right).
\end{align*}
This concludes the proof.
\end{proof}

As the maps $\LR$, $\mathsf{t}_{d\brho}$, and $\mathsf{t}_{\brho^d_{\mathrm{rev}}}$ are all bijections, we deduce:
\begin{Cor}
The spetsial level-rank duality $\sLR:\cU^\ell\lra\cU^{\ell d}$ is a bijection for any $d\in\Z_{\geq 1}$.
\end{Cor}
In fact, we can say more. Note that $|\brho|=|\brho_{\mathrm{rev}}|=\frac{\ell(\ell-1)}{2}$. For any $\si\in\Z$, the maps $\mathsf{t}_{d\brho}$ and $\mathsf{t}_{\brho^d_{\mathrm{rev}}}$ thus restrict to bijections $\cU^\ell_\si\lra\cU^\ell_{\si+d\frac{\ell(\ell-1)}{2}}$ and $\cU^e_\si\lra\cU^e_{\si+d\frac{\ell(\ell-1)}{2}}$, respectively. For any $\si\in\Z$,  by \Cref{LR} the commutative diagram in \Cref{prop_lr_vs_slr} therefore restricts to a commutative diagram
\begin{equation}\label{commlrdualities}
\begin{tikzcd}
 \cU^\ell_\si \arrow[r, "\sLR"] \arrow[d, "\mathsf{t}_{d\brho}"]
&  \cU^e_\si  \arrow[d, "\mathsf{t}_{\brho^d_{\mathrm{rev}}}"] \\
  \arrow[r, "\LR"]
  \cU^\ell_{\si+d\frac{\ell(\ell-1)}{2}}   &\cU^e_{\si+d\frac{\ell(\ell-1)}{2}}
  \end{tikzcd}
\end{equation}
from which we deduce:
\begin{Cor}\label{sLR}
Let $\si\in\Z$ and $d\in\Z_{\geq 1}$. The spetsial level-rank duality yields a bijection $\sLR:\cU^\ell_\si\lra\cU^{\ell d}_\si$.
\end{Cor}
Set
\[
\cC_{[d,\si]}^\ell  = \{|\bla,\bsig\rangle\in\cU^\ell_\si\;\mid\; |\bla,\bsig\rangle_{[d]}=|\bla,\bsig\rangle\}.
\]
Thanks to \Cref{sLR}, we thus obtain a second bijection analogous to the Littlewood decomposition \Cref{Lit}:

\begin{equation}
\label{Lit_l_bis}
\begin{array}{ccc}
\cU^\ell_\si &  
\overset{1:1}{\longleftrightarrow} &  
\cC_{[d,\si]}^\ell \times  \cP^e
\\
|\bla,\bsig\rangle & \longrightarrow & \left( |\bla,\bsig\rangle_{[d]} \,,\, |\bla,\bsig\rangle^{[d]} \right).
\end{array}
\end{equation}
where as before, $e=\ell d$.
In the case that $\si\equiv 1\bmod \ell$, we recover \cite[Corollary 3.5]{Malle1995}

\begin{Exa}\label{exa_lr_vs_slr} Let $\ell=2$ and consider the charged bipartition $|\bla,\bsig\rangle$ for $\bla=\left((3^3,2^3,1^3),(10,6,3^2)\right)$ and $\bsig=(3,-2)$. Let $d=3$. We illustrate \Cref{prop_lr_vs_slr} on the abacus of $|\bla,\bsig\rangle$. Note that $3\brho=(0,3)$ and $\brho^3_{\mathrm{rev}}=(1,1,1,0,0,0)$.
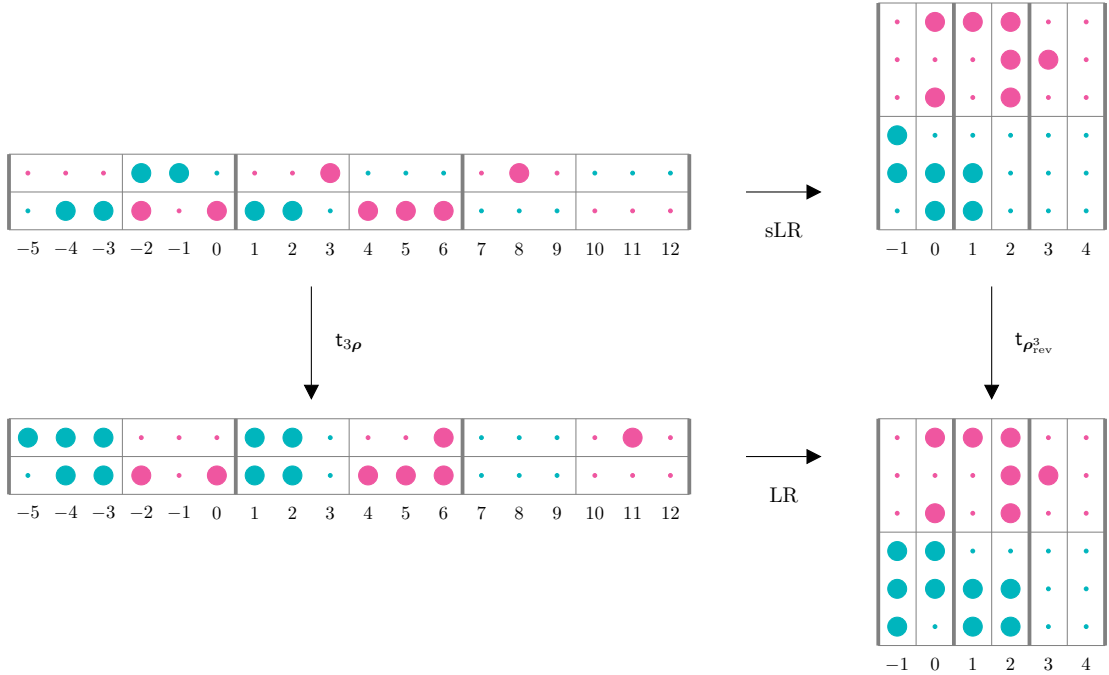
\begin{center}
\begin{tikzpicture}[scale=0.5, bb/.style={draw,circle,fill,minimum size=2.5mm,inner sep=0pt,outer sep=0pt}, wb/.style={draw,circle,fill,minimum size=0.5mm,inner sep=0pt,outer sep=0pt}]
%%%%%%%%%%%%%%%%%%%%%
%%%%%% top left abacus %%%%%
%%%%%%%%%%%%%%%%%%%%

%%%%%%%%%%%%%%%%%%%%
%%%%% On dessine les emplacements sans bille
\foreach \k/\j in {-5/2,-4/2,-3/2,1/2,2/2,7/2,9/2,      -1/1,10/1,11/1,12/1}
{
    \node [wb,VioletRed] at (\k-8,\j) {};
}
\foreach \k/\j in { 0/2,4/2,5/2,6/2,10/2,11/2,12/2,    -5/1,3/1,7/1,8/1,9/1}
{
    \node [wb,Aquamarine] at (\k-8,\j) {};
}
%%%%%%%%%%%%%%%%%%%%%
%%%  On place les billes
\foreach \k/\j in {3/2,8/2,    -2/1,0/1,4/1,5/1,6/1}
{
    \node [bb,VioletRed] at (\k-8,\j) {};
}
\foreach \k/\j in {-2/2,-1/2,     -4/1,-3/1,1/1,2/1}
{
    \node [bb,Aquamarine] at (\k-8,\j) {};
}
%%%%%%%%%%%%%%%%%%%%%
%%%  On fait appraître la graduation
\foreach \k in {-5,...,12}
{
    \node [scale = 0.7] at (\k-8,0) {$\k$};
}
%%%%%%%%%%%%%%%%%%%%%
%%%  On dessine les rectangles
\draw [line width=0.1mm, color=gray] (-13.5,0.5) -- (4.5,0.5);
\draw [line width=0.1mm, color=gray] (-13.5,1.5) -- (4.5,1.5);
\draw [line width=0.1mm, color=gray] (-13.5,2.5) -- (4.5,2.5);
\foreach \k in {-13.5,-10.5,-7.5, -4.5, -1.5, 1.5, 4.5}
{
\draw [line width=0.1mm, color=gray] (\k,2.5) -- (\k,0.5);
}
\foreach \k in {-13.5, -7.5, -1.5, 4.5}
{
\draw [line width=0.5mm, color=gray] (\k,2.5) -- (\k,0.5);
}

%%%%%%%%%%%%%
%%% top arrow %%%%%
%%%%%%%%%%%%%

\draw[-{Triangle[length=2mm,width=2mm]}]
  (6,1.5) -- (8,1.5)
;
\node [scale=0.8] at (7,0.5){$\sLR$};

%%%%%%%%%%%%%%
%% top right abacus%%%
%%%%%%%%%%%%%%

%%%%%%%%%%%%%%%%%%%%
%%%%% On dessine les emplacements sans bille
\foreach \k/\j in {-1/4,-1/5,-1/6,1/4,1/5,3/4,3/6,0/5,4/4,4/5,4/6}
{
    \node [wb,VioletRed] at (\k+11,\j) {};
}
\foreach \k/\j in { -1/1,0/3,1/3,2/1,2/2,2/3,3/1,3/2,3/3,4/1,4/2,4/3}
{
    \node [wb,Aquamarine] at (\k+11,\j) {};
}
%%%%%%%%%%%%%%%%%%%%%
%%%  On place les billes
\foreach \k/\j in {0/4,0/6,1/6,2/4,2/5,2/6,3/5}
{
    \node [bb,VioletRed] at (\k+11,\j) {};
}
\foreach \k/\j in {-1/2,-1/3, 0/1,0/2,1/1,1/2}
{
    \node [bb,Aquamarine] at (\k+11,\j) {};
}
%%%%%%%%%%%%%%%%%%%%%
%%%  On fait appraître la graduation
\foreach \k in {-1,...,4}
{
    \node [scale = 0.7] at (\k+11,0) {$\k$};
}
%%%%%%%%%%%%%%%%%%%%%
%%%  On dessine les rectangles
\draw [line width=0.1mm, color=gray] (9.5,0.5) -- (15.5,0.5);
\draw [line width=0.1mm, color=gray] (9.5,3.5) -- (15.5,3.5);
\draw [line width=0.1mm, color=gray] (9.5,6.5) -- (15.5,6.5);
\foreach \k in {9.5,10.5,11.5, 12.5, 13.5, 14.5, 15.5}
{
\draw [line width=0.1mm, color=gray] (\k,6.5) -- (\k,0.5);
}
\foreach \k in {9.5, 11.5, 13.5, 15.5}
{
\draw [line width=0.5mm, color=gray] (\k,6.5) -- (\k,0.5);
}
%%%%%%%%%%%%%
%%%% left arrow%%%%%
%%%%%%%%%%%%%

\draw[-{Triangle[length=2mm,width=2mm]}]
  (-5.5,-1) -- (-5.5,-4)
;
\node [scale=0.8] at (-4.5,-2.5){$\mathsf{t}_{3\brho}$};

%%%%%%%%%%%%%%%
%%%bottom left abacus%%%
%%%%%%%%%%%%%%%

%%%%%%%%%%%%%%%%%%%%
%%%%% On dessine les emplacements sans bille
\foreach \k/\j in {-2/2,-1/2,0/2,4/2,5/2,10/2,12/2,      -1/1,10/1,11/1,12/1}
{
    \node [wb,VioletRed] at (\k-8,\j-7) {};
}
\foreach \k/\j in { 3/2,7/2,8/2,9/2,    -5/1,3/1,7/1,8/1,9/1}
{
    \node [wb,Aquamarine] at (\k-8,\j-7) {};
}
%%%%%%%%%%%%%%%%%%%%%
%%%  On place les billes
\foreach \k/\j in {6/2,11/2,    -2/1,0/1,4/1,5/1,6/1}
{
    \node [bb,VioletRed] at (\k-8,\j-7) {};
}
\foreach \k/\j in {-5/2,-4/2,-3/2,1/2,2/2,     -4/1,-3/1,1/1,2/1}
{
    \node [bb,Aquamarine] at (\k-8,\j-7) {};
}
%%%%%%%%%%%%%%%%%%%%%
%%%  On fait appraître la graduation
\foreach \k in {-5,...,12}
{
    \node [scale = 0.7] at (\k-8,-7) {$\k$};
}
%%%%%%%%%%%%%%%%%%%%%
%%%  On dessine les rectangles
\draw [line width=0.1mm, color=gray] (-13.5,-6.5) -- (4.5,-6.5);
\draw [line width=0.1mm, color=gray] (-13.5,-5.5) -- (4.5,-5.5);
\draw [line width=0.1mm, color=gray] (-13.5,-4.5) -- (4.5,-4.5);
\foreach \k in {-13.5,-10.5,-7.5, -4.5, -1.5, 1.5, 4.5}
{
\draw [line width=0.1mm, color=gray] (\k,-6.5) -- (\k,-4.5);
}
\foreach \k in {-13.5, -7.5, -1.5, 4.5}
{
\draw [line width=0.5mm, color=gray] (\k,-6.5) -- (\k,-4.5);
}

%%%%%%%%%%%%%
%%%% bottom arrow%%%%%
%%%%%%%%%%%%%

\draw[-{Triangle[length=2mm,width=2mm]}]
  (6,-5.5) -- (8,-5.5)
;
\node [scale=0.8] at (7,-6.5){$\LR$};

%%%%%%%%%%%%%
%%%% right arrow%%%%%
%%%%%%%%%%%%%

\draw[-{Triangle[length=2mm,width=2mm]}]
  (12.5,-1) -- (12.5,-4)
;
\node [scale=0.8] at (13.65,-2.5){$\mathsf{t}_{\brho^3_{\mathrm{rev}}}$};

%%%%%%%%%%%%%%
%% bottom right abacus%%%
%%%%%%%%%%%%%%

%%%%%%%%%%%%%%%%%%%%
%%%%% On dessine les emplacements sans bille
\foreach \k/\j in {-1/4,-1/5,-1/6,1/4,1/5,3/4,3/6,0/5,4/4,4/5,4/6}
{
    \node [wb,VioletRed] at (\k+11,\j-11) {};
}
\foreach \k/\j in { 0/1,1/3,2/3,3/1,3/2,3/3,4/1,4/2,4/3}
{
    \node [wb,Aquamarine] at (\k+11,\j-11) {};
}
%%%%%%%%%%%%%%%%%%%%%
%%%  On place les billes
\foreach \k/\j in {0/4,0/6,1/6,2/4,2/5,2/6,3/5}
{
    \node [bb,VioletRed] at (\k+11,\j-11) {};
}
\foreach \k/\j in {-1/1,-1/2,-1/3,0/2,0/3, 1/1,1/2,2/1,2/2}
{
    \node [bb,Aquamarine] at (\k+11,\j-11) {};
}
%%%%%%%%%%%%%%%%%%%%%
%%%  On fait appraître la graduation
\foreach \k in {-1,...,4}
{
    \node [scale = 0.7] at (\k+11,-11) {$\k$};
}
%%%%%%%%%%%%%%%%%%%%%
%%%  On dessine les rectangles
\draw [line width=0.1mm, color=gray] (9.5,-4.5) -- (15.5,-4.5);
\draw [line width=0.1mm, color=gray] (9.5,-7.5) -- (15.5,-7.5);
\draw [line width=0.1mm, color=gray] (9.5,-10.5) -- (15.5,-10.5);
\foreach \k in {9.5,10.5,11.5, 12.5, 13.5, 14.5, 15.5}
{
\draw [line width=0.1mm, color=gray] (\k,-4.5) -- (\k,-10.5);
}
\foreach \k in {9.5, 11.5, 13.5, 15.5}
{
\draw [line width=0.5mm, color=gray] (\k,-4.5) -- (\k,-10.5);
}

\end{tikzpicture}
\end{center}
\end{Exa}

Therefore, the spetsial $d$-core $|\bla,\bsig\rangle_{[d]}$ of a charged multipartition $|\bla,\bsig\rangle$ 
can be computed by one of the following equivalent procedures.
\begin{enumerate}
\item[$(1'')$] 
recursively moving beads up and $d$ positions to the left in consecutive rows of $\bA(\bla,\bsig)$,
\item[$(2'')$] recursively sliding beads $1$ step to the left in $\sLR\left(\bA(\bla,\bsig)\right)$.
\end{enumerate}
Indeed, \Cref{def_spetsialdcore} involves exactly Procedure $(2'')$, and Procedure $(1'')$ corresponds to Procedure $(2'')$ under the map $\sLR$.
These are analogous to Procedure $(1')$ and $(2')$ for computing the $e$-core $|\bla,\bs\rangle_{(e)}$ of a charged multipartition $|\bla,\bs\rangle$, see \Cref{sec_cores_l}.
More precisely, when $e=d\ell$, Procedure $(1')$ is given by composing Procedure $(1'')$ with $\mathsf{t}_{d\brho}$, while Procedure $(2')$ is given by composing Procedure $(2'')$ with $\mathsf{t}_{\brho^d_{\mathsf{rev}}}$.

\medskip

In fact, Procedure $(1'')$ should really be viewed as a ``hook removal'' procedure, 
as introduced in Malle's original definition of the spetsial $d$-core.
More precisely, identifying $|\bla,\bsig\rangle$ with its abacus, define a {\em spetsial $d$-hook}, or simply {\em $d$-hook}, of $\bA=(\bA^{(1)},\ldots,\bA^{(\ell)})$ to be some $\be\in\bA^{(j)}$, $j\in\{1,\ldots,\ell\}$, such that $\be-d\notin\bA^{(j+1)}$ (where superscripts are taken mod $\ell$).\footnote{Note that when $\ell=2$, this is exactly the definition of a cohook. Likewise for the cocore.} In \cite{Malle1995}, a spetsial $d$-hook is called a {\em $(d,\zeta^1)$-hook}, where $\zeta$ is a primitive $\ell$'th root of $1$.\footnote{
Malle also considers $(d,\zeta^k)$-hooks for $1<k\leq \ell$, defined by $\be\in\bA^{(j)}$ such that $\be-d\notin\bA^{(j+k)}$.}
Now, {\em removing} such a $d$-hook from $\bA$ consists in removing $\be$ from $\bA^{(j)}$ and adding $\be-d$ to $\bA^{(j+1)}$. The spetsial $d$-core of $\bA$, called the {\em $(d,\zeta)$-core} in \cite{Malle1995}, is obtained by recursively removing all $d$-hooks from $\bA$. 
\begin{Exa}
Take $\ell=3$ and $d=2$.  Let $\bla=\left( (2,1^3), (4,2,1), (3) \right)\in\cP^3$ as in \Cref{exa_bead_slide}, and take $\bsig=(0,-3,-5)\in\Z^3$ so that $\bs$ in \Cref{exa_bead_slide} is given by $\bs=\bsig+2\brho$. We draw all the $2$-hook removals described by Procedure $(1'')$ with colored arrows.
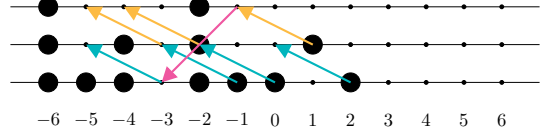
\begin{center}
\begin{tikzpicture}[scale=0.5, bb/.style={draw,circle,fill,minimum size=2.5mm,inner sep=0pt,outer sep=0pt}, wb/.style={draw,circle,fill,minimum size=0.5mm,inner sep=0pt,outer sep=0pt}]
%%%%%%%%%%%%%%%%%%%%%
%%%  On dessine les lignes de l'abaque et tous les emplacements
\foreach \j in {1,...,3}
{
\draw [line width=0.1mm] (-7,\j) -- (7,\j);
    \foreach \k in {-6,...,6}
    {
        \node [wb] at (\k,\j) {};
    }
}
%%%%%%%%%%%%%%%%%%%%%
%%%  On place les billes
\foreach \k/\j in { -6/1,-5/1,-4/1,-2/1,-1/1,0/1,2/1,    -6/2,-4/2,-2/2,1/2,   -6/3,-2/3 }
{
    \node [bb] at (\k,\j) {};
}
%%%%%%%%%%%%%%%%%%%%%
%%%  On fait appraître la graduation
\foreach \k in {-6,...,6}
{
    \node [scale = 0.7] at (\k,0) {$\k$};
}
% %%%%%%%%%%%%%%%%%%%%%
% %%%  On dessine les rectangles
% \draw [line width=0.1mm, color=gray] (-5.5,0.5) -- (6.5,0.5);
% \draw [line width=0.1mm, color=gray] (-5.5,3.5) -- (6.5,3.5);
% \foreach \k in {-5.5, 0.5, 6.5}
% {
% \draw [line width=0.1mm, color=gray] (\k,3.5) -- (\k,0.5);
% }

%%%%%%%%%% Les fleches
\draw[-{Triangle[length=2mm,width=2mm]}, thick, Aquamarine]
  (-3,1) -- (-5,2);
\draw[-{Triangle[length=2mm,width=2mm]}, thick, Aquamarine]
  (-1,1) -- (-3,2);
\draw[-{Triangle[length=2mm,width=2mm]}, thick, Aquamarine]
  (0,1) -- (-2,2);
\draw[-{Triangle[length=2mm,width=2mm]}, thick, Aquamarine]
  (2,1) -- (0,2);
  %%%%
\draw[-{Triangle[length=2mm,width=2mm]}, thick, Dandelion]
  (1,2) -- (-1,3);
\draw[-{Triangle[length=2mm,width=2mm]}, thick, Dandelion]
  (-2,2) -- (-4,3);
\draw[-{Triangle[length=2mm,width=2mm]}, thick, Dandelion]
  (-3,2) -- (-5,3);
\draw[-{Triangle[length=2mm,width=2mm]}, thick, VioletRed]
  (-1,3) -- (-3,1);
\end{tikzpicture}
\end{center}
We obtain
$|\bla,\bsig\rangle_{[2]} = |\left( (1),(2,1),(1) \right),(-3,-2,-3)\rangle$. Observe that the operations in the abacus in \Cref{exa_bead_slide} are obtained from the picture above by simply translating row $j$ by $2(j-1)$ for $j=1,2,3$, and that $|\bla,\bs\rangle_{(6)}=\mathsf{t}_{2\brho}(|\bla,\bsig\rangle_{[2]})$.
\end{Exa}

Spetsial cores may thus be viewed as a particular case of 
cores in the sense of \Cref{sec_cores_l}.
\begin{Cor} \label{compare_cores} Let $\ell\in\Z_{\geq 2}$ and $d\in\Z_{\geq 1}$, and set $e=d\ell$. Let $\bsig\in\Z^\ell$ and set $\bs=\bsig+d\brho$. For any $\bla\in\cP^\ell$, 
\[
|\bla,\bs\rangle_{(e)}=\mathsf{t}_{d\brho}\left(|\bla,\bsig\rangle_{[d]}\right).
\] 
That is, the $e$-core of $|\bla,\bs\rangle$ coincides with the spetsial $d$-core of $|\bla,\bsig\rangle$ whose charge has been shifted by $d\brho$. Moreover, the spetsial $d$-quotient of $|\bla,\bsig\rangle$ is equal to the $e$-quotient of $|\bla,\bs\rangle$:
\[
|\bla,\bsig\rangle^{[d]}=|\bla,\bs\rangle^{(e)}.
\]
\end{Cor}
\begin{proof}
This follows directly from \Cref{prop_lr_vs_slr}.
\end{proof}
We note that if
$|\bla,\bsig\rangle$ is a spetsial $d$-core then $\la^j$ is an $e$-core partition for every $j=1,\ldots,\ell$ 
and $\bsig\in \sDe_d^\ell$ for 
\[
\sDe_d^\ell=\{ (\si_1,\ldots,\si_\ell)\in\Z^{\ell} \mid \si_j\leq \si_{j+1}+d \text{ for all } 1\leq j\leq \ell-1 \text{ and } \si_\ell\leq \si_1+d\}.
\]
Note that, for $\si\in\Z$ and $s=\si+d\frac{\ell(\ell-1)}{2}$,
we have  $\bsig\in\sDe_d^\ell[\si] $ if and only if $ \bs \in\sD_e^\ell[s]$.

\subsection{The rank statistic}\label{sec_rank}

In order to have a full analogue of the Littlewood Decomposition \Cref{Lit}, the naive hope would be for Bijection \Cref{Lit_l} or \Cref{Lit_l_bis} to be compatible with the size of multipartitions. However, the size is a statistic that only sees $\bla$ and does not take account of the charge, unlike the symbols and their cores appearing in the bijection. 
So rather unsurprisingly, it turns out to be more subtle and we need to use the \textit{rank} statistic rather than the size, which we recall und study next.
This section is based on \cite{Olsson1986} and \cite{Malle1995}.
We first recall the definition of the rank of an abacus from \cite{Malle1995} and some basic properties, 
providing a few details not present in the original papers.
By viewing the abacus as a charged multipartition, we reformulate the rank of $|\bla,\bsig\rangle$ as the size of $\bla$ plus an $\N$-valued function of the charge $\bsig$.

\medskip

The rank of a charged $\ell$-partition $|\bla,\bsig\rangle$ is first defined using its symbol or abacus $\bA(\bla,\bsig)$. 
If \linebreak $\bS = (\bS^{(1)}, \ldots, \bS^{(\ell)})$ is an $\ell$-symbol, we will denote the multiset of its $\beta$-numbers by $\{\!\{\bS\}\!\}$. If $\bA$ is an $\ell$-abacus, we then define $\{\!\{\!\bA\!\}\!\}=\{\!\{\bS\}\!\}$ where $\bS$ is the symbol identified with $\bA$ by the correspondence $\be\in\bS^{(j)}$ if and only if there is a bead in row $j$ and column $\be$ of $\bA$. 

\begin{Exa}\label{ab_multiset}
For $\bA$ as below we have $\{\!\{\!\bA\!\}\!\}=\{\!\{4,3,2,1,1,0,-2,-2,-3,-4,-4,-5,-5,\ldots\}\!\}$, where we use double braces notation to distinguish a multiset from a set.
 \begin{center}
\begin{tikzpicture}[scale=0.5, bb/.style={draw,circle,fill,minimum size=2.5mm,inner sep=0pt,outer sep=0pt}, gbb/.style={draw,circle,fill,minimum size=2.5mm,inner sep=0pt,outer sep=0pt, color=black!60}, wb/.style={draw,circle,fill,minimum size=0.5mm,inner sep=0pt,outer sep=0pt}, gwb/.style={draw,circle,fill,minimum size=0.5mm,inner sep=0pt,outer sep=0pt, color=black!60}]
%%%%%%%%%%%%%%%%%%%%%
%%%  On dessine les lignes de l'abaque et tous les emplacements
\foreach \j in {1,2}
{
\draw [line width=0.1mm] (-5,\j) -- (4,\j);
    \foreach \k in {-5,...,4}
    {
        \node [wb] at (\k,\j) {};
    }
}
%%%%%%%%%%%%%%%%%%
%%%  On place les billes
\foreach \k/\j in {-5/1,-4/1,-2/1,1/1,2/1,    -5/2,-4/2,-3/2,-2/2,0/2, 1/2, 3/2, 4/2   }
{
    \node [bb] at (\k,\j) {};
}
%%%%%%%%%%%%%%%%%%%%%
%%%  On fait appraître la graduation
\foreach \k in {-5,...,4}
{
    \node [scale = 0.7] at (\k,0) {$\k$};
}
\end{tikzpicture}
\end{center}
 \end{Exa}

\begin{Def}\label{def_rk_abacus}
Following \cite[(3.2)]{Malle1995}, we define the \textit{rank} of an
$\ell$-abacus $\bA $ by
\[
\rk (\bA) = \sum_{\substack{\beta\in\{\!\{\!\bA\!\}\!\} \\ \beta \geq m}}  \beta - cm  -  \left\lfloor \frac{ (c - 1) ( c - \ell+1) }{ 2\ell} \right\rfloor
\]
where
\begin{itemize}
\item $m=\max\{z\in\mathbb{Z}\;\mid\; \hbox{ there exists }j\in\{1,\ldots,\ell\}\hbox{ such that }z\notin\bA^{(j)}\}$,
\item $c$ is the cardinality of the multiset $\{\!\{\beta\in\{\!\{\bA\}\!\}\;\mid\;\beta\geq m\}\!\}$. The integer $c$ is called the \textit{content} of $\bA$ \cite{Malle1995}. This is the number of beads in $\bA$ that lie in or to the right of column $m$ in $\bA$.
\end{itemize}
If $|\bla,\bsig\rangle\in\cU^\ell$, we define $\rk(|\bla,\bsig\rangle)=\rk(\bA)$ where $\bA=\bA(\bla,\bsig)$.
\end{Def}
We remark that for $\ell=1$, $c$ is the number of non-zero parts in the partition $\la$. 

\begin{Rem}
The convention used by \cite{Malle1995} and \cite{Chevie} is that $m=0$. This amounts to requiring that an arbitrary abacus be shifted horizontally by $m$ so that its leftmost empty space is at position $0$ in order for the rank of the abacus to be computed, which is natural in those authors' conventions where they work with equivalence classes of symbols by horizontal shift and cyclic rotation. We have incorporated the shift by $m$ into our formula for the rank which is valid for any abacus, hence the middle term $-cm$. 
\end{Rem}

\begin{Exa}
Let $\bA$ be the abacus in \Cref{ab_multiset}. Then $m=-3$ and $c=9$. As $\ell=2$, we have
\[
\rk(\bA)= 2+1+(-2) +4+3+1+0+(-2)+(-3)-(-3)(9)-\left\lfloor \frac{(9-1)(9-2+1)}{2\cdot 2}\right\rfloor=15.
\]
\end{Exa}

The following lemma shows that we can in fact truncate $\bA$ anywhere to the left of the leftmost space in $\bA$ in order to compute the rank of $\bA$.

\begin{Lem}\label{truncation}
Let $\bA$ be an $\ell$-abacus and let  $m$ be as in \Cref{def_rk_abacus}.
Let $k\leq m$. Then
\[
\rk(\bA)=\sum_{\substack{\beta \in\{\!\{\!\bA\!\}\!\}\\ \beta\geq k}}\beta-kc_k-\left\lfloor\frac{(c_k-1)(c_k-\ell+1)}{2\ell}\right\rfloor
\]
where $c_k$ is the cardinality of the multiset $\{\!\{\beta\in\{\!\{\!\bA\!\}\!\}\;\mid\;\beta\geq k \}\!\}$.
\end{Lem}

\begin{proof}
By downwards induction on $k$. The base case $k=m$ is the definition of the rank, with $c_m=c$. Now let $k\leq m$ and assume the formula holds for $k$. We will show it holds for $k-1$. We have \linebreak
$\rk(\bA)=\sum\limits_{\substack{\beta\in\{\!\{\!\bA\!\}\!\}\\ \beta\geq k}}\beta-kc_k-\left\lfloor \frac{(c_k-1)(c_k-\ell+1)}{2\ell}\right\rfloor$. Since $k-1<m$, it follows from the definition of $m$ that $k-1\in \bA^{(j)}$ for all $j=1,\ldots,\ell$. It follows that $c_{k-1}=c_k+\ell$ and $\sum\limits_{\substack{\beta\in\{\!\{\!\bA\!\}\!\}\\\beta\geq k-1}}\beta=\sum\limits_{\substack{\beta\in\{\!\{\!\bA\!\}\!\}\\ \beta\geq k}}\beta+\ell(k-1)$, whence:
\begin{align*}
\sum\limits_{\substack{\beta\in\{\!\{\!\bA\!\}\!\}\\\beta\geq k-1}}\beta-(k-1)c_{k-1}-\left\lfloor \frac{(c_{k-1}-1)(c_{k-1}-\ell+1)}{2\ell}\right\rfloor&=\sum\limits_{\substack{\beta\in\{\!\{\!\bA\!\}\!\}\\ \beta\geq k}}\beta+\ell(k-1)-(k-1)(c_k+\ell)-\left\lfloor\frac{(c_k+\ell-1)(c_k+1)}{2\ell} \right\rfloor\\
&=\sum\limits_{\substack{\beta\in\{\!\{\!\bA\!\}\!\}\\ \beta\geq k}}\beta-c_kk+c_k-\left\lfloor\frac{(c_k-1)(c_k-\ell+1)}{2\ell}+c_k \right\rfloor\\
&=\sum\limits_{\substack{\beta\in\{\!\{\!\bA\!\}\!\}\\ \beta\geq k}}\beta-kc_k-\left\lfloor \frac{(c_k-1)(c_k-\ell+1)}{2\ell}\right\rfloor\\
&=\rk(\bA).
\end{align*}
\end{proof}

The following properties of the rank function are straightforward.
\begin{Lem}\label{lem_invariance}
The rank is invariant under 
\begin{itemize}
\item permutation of the abacus rows, that is $\rk(\bA^{(1)},\ldots, \bA^{(\ell)}) = \rk(\bA^{(\pi(1))},\ldots, \bA^{(\pi(\ell))})$ for all $\pi\in\mathfrak{S}_\ell$,
\item simultaneous horizontal shift of all rows of an abacus by a fixed amount, that is, \linebreak $\rk(\bA^{(1)},\ldots, \bA^{(\ell)}) = \rk(\bA^{(1)}+k,\ldots, \bA^{(\ell)}+ k)$  for all $k\in\Z$,
\item permutation of the entries of a given abacus column. 
\end{itemize}
\end{Lem}

Since the rank of an abacus or symbol is the rank of a charged $\ell$-partition $|\bla,\bsig\rangle$, it is natural to wonder if there is a formula for the rank in terms of $\bla$ and $\bsig$. 
For instance, when $\ell=1$, we simply have $\rk(\bA(\la,\si))=|\la|$ for any $\si\in\Z$. 
We give such a formula in the following proposition.
Let $\bA=\bA(\bla,\bsig) = (\bA^{(1)},\ldots, \bA^{(\ell)})$ and 
set $c_j=|\{\beta\in \bA^{(j)}\;:\;\beta\geq m\}|$, where $m$ is defined as above.  
When $m=0$ we have $c_j=\si_j+1$.
Moreover, if $\bsig=(\si_1,\ldots,\si_\ell)$, define
\begin{equation}\label{def_rk_charge}
Y:=\sum\limits_{i=1}^{\ell} \sum\limits_{\substack{ j=1 \\ j\neq i}}^\ell \si_i(\si_i-\si_j)
\mand
\rk (\bsig) = \left\lceil \frac{Y-\ell+1}{2\ell}\right\rceil.
\end{equation}

\begin{Prop}\label{rk_formula}
For all $\bla\in\cP^\ell$ and $\bsig\in\Z^\ell$, we have
$\rk(|\bla,\bsig\rangle) = |\bla| + \rk (\bsig).$
\end{Prop}

\begin{proof}
Let $\si=\sum_{i=1}^\ell \si_i$. Then $c=\sum_{i=1}^\ell c_i=\sum_{i=1}^\ell(\si_i+1)=\si+\ell$. 
By Lemma \ref{lem_invariance}, we may assume that $m=0$. We then have:
\[
\rk(|\bla,\bsig\rangle)=\rk(\bA)=\sum\limits_{\substack{\beta\in\{\!\{\!\bA\!\}\!\} \\ \beta\geq 0}}\beta-\left\lfloor \frac{(c-1)(c-\ell+1)}{2\ell}\right\rfloor
=\sum\limits_{\substack{\beta\in\{\!\{\!\bA\!\}\!\} \\ \beta\geq 0}}\beta-\left\lfloor \frac{(\si+\ell-1)(\si+1)}{2\ell}\right\rfloor.
\]
We find 
\begin{align*}
 (\si+\ell-1)(\si+1)&=\si^2+\ell\si+\ell-1=\sum\limits_{i=1}^\ell \si_i^2+2\sum\limits_{\substack{i,j=1 \\ i<j}}^\ell \si_i\si_j+\ell\sum\limits_{i=1}^\ell \si_i+ \ell-1 \\
& =\ell\left(\sum\limits_{i=1}^\ell \si_i^2+\si_i\right)+2\sum\limits_{\substack{i,j=1 \\ i<j}}^\ell \si_i\si_j -(\ell-1)\sum\limits_{i=1}^\ell \si_i^2 + \ell-1\\
&=\ell\left(\sum\limits_{i=1}^\ell \si_i^2+\si_i\right)-Y+\ell-1.
 \end{align*}
This yields
 \begin{align*}
 \rk(|\bla,\bsig\rangle)&=\sum\limits_{\substack{\beta\in\{\!\{\!\bA\!\}\!\} \\ \beta\geq 0}}\beta-\left\lfloor \frac{\sum\limits_{i=1}^\ell \si_i^2+\si_i}{2}-\left(\frac{Y-\ell+1}{2\ell}\right)\right\rfloor =\sum\limits_{\substack{\beta\in\{\!\{\!\bA\!\}\!\} \\ \beta\geq 0}}\beta- \frac{\sum\limits_{i=1}^\ell \si_i^2+\si_i}{2}-\left\lfloor-\left(\frac{Y-\ell+1}{2\ell}\right)\right\rfloor \\
 &=\sum\limits_{i=1}^\ell\left(\sum\limits_{\substack{\beta\in[A^{(i)}] \\ \beta\geq 0}}\beta-\frac{(\si_i^2+\si_i)}{2}\right)+\left\lceil\frac{Y-\ell+1}{2\ell}\right\rceil=\sum\limits_{i=1}^\ell \rk(|\lambda^i,\si_i\rangle)+\left\lceil\frac{Y-\ell+1}{2\ell}\right\rceil \\
 &= \sum\limits_{i=1}^\ell |\lambda^i|+\left\lceil\frac{Y-\ell+1}{2\ell}\right\rceil =|\bla|+ \left\lceil\frac{Y-\ell+1}{2\ell}\right\rceil =|\bla|+\rk(\bsig).
 \end{align*}
 This concludes the proof.
\end{proof}
\begin{Cor}\label{rk bsig}
We have $\rk(|\bemp,\bsig\rangle)=\rk(\bsig)$ for any $\bsig\in\Z^\ell$.
\end{Cor}
As a corollary, we recover the following known formulas for the rank when $\ell=1$ and $\ell=2$, and we also have a similar formula when $\ell=3$.
\begin{Cor}\label{rk_ell=123}
\begin{enumerate}
\item Assume $\ell=1$. Then $\rk(|\la,s\rangle) = |\la|$ for any partition $\lambda$ and any $s\in\Z$.

\item Assume $\ell=2$ and write $\bsig=(\si_1,\si_2)$.
Then $\ds \rk(|\bla,\bsig\rangle) = |\bla| + \left\lfloor \frac{(\si_1-\si_2)^2}{4} \right\rfloor.$

\item Assume $\ell=3$ and write $\bsig=(\si_1,\si_2,\si_3)$. Then 
\[
\rk(|\bla,\bsig\rangle)=|\bla|+\left\lfloor \frac{(\si_1-\si_2)^2+(\si_3-\si_1)(\si_3-\si_2)}{3}\right\rfloor.
\]
\end{enumerate}

\end{Cor}

\begin{proof}
The formula for $\ell=1$ is immediate from \Cref{rk_formula}, while the formula for $\ell=2$ follows from it together with the observation that the squares mod $4$ are $0$ and $1$. When $\ell=3$ we have \linebreak$Y-\ell+1=2(\sum_i \si_i^2-\sum_{i<j} \si_i\si_j)-2$, whence $\left\lceil \frac{ Y-2}{6}\right\rceil =  \left \lceil \frac{\sum_i \si_i^2-\sum_{i<j} \si_i\si_j-1}{3}\right\rceil  = \left\lfloor \frac{\sum_i \si_i^2-\sum_{i<j} \si_i\si_j}{3}\right\rfloor = \left\lfloor \frac{(\si_1-\si_2)^2+(\si_3-\si_1)(\si_3-\si_2)}{3}\right\rfloor$ by congruence considerations mod $3$. We apply \Cref{rk_formula} to conclude.
\end{proof}

In level $2$, the formula of Corollary \ref{rk_ell=123} allows us to give an alternative expression for the rank
depending on the parity of $\si_2-\si_1$, which often appears in the literature.

\begin{Cor}\label{rk_t}
Assume $\ell=2$.
We have
\[
\rk(|\bla,\bsig\rangle)=\begin{cases} 
|\bla|+t^2+t \quad\hbox{ if }\si_2-\si_1=2t+1 \text{ for some }t\in\mathbb{Z},\\ 
|\bla|+t^2 \qquad\; \;\hbox{ if }\si_2-\si_1=2t  \text{ for some }t \in\mathbb{Z}.
\end{cases}
\]
\end{Cor}

\begin{proof}
We have
\begin{align*}
\rk(|\bla,\bsig\rangle) & = |\bla| + \left\lfloor \frac{(\si_1-\si_2)^2}{4} \right\rfloor
\\& =
\left\{
\begin{array}{l}
|\bla| + \left\lfloor \frac{(2t+1)^2}{4} \right\rfloor
= |\bla| + \left\lfloor t^2+t+ \frac{1}{4} \right\rfloor
 = |\bla| + t^2+t \text{ \quad in the first case,}
\\
|\bla| + \left\lfloor \frac{(2t)^2}{4} \right\rfloor
= |\bla| + \left\lfloor t^2\right\rfloor
 = |\bla| + t^2 \text{ \quad in the second case.}
\end{array}
\right.
\end{align*}
\end{proof}

Let $\bA=\bA(\bla,\bsig)$ and assume that $\bsig=(0,\ldots,0)$ or any permutation of $(0,\ldots,0,\pm 1)$. Clearly $\rk(\bsig)=0$, 
and thus by \Cref{rk_formula} we have $\rk(\bA) = |\bla|$, but see also \cite[Section 3A]{Malle1995}.

\begin{Rem}
\Cref{rk_formula} implies a formula for the size of a multipartition $\bla$ as a function of the abacus or symbol of $|\bla,\bs\rangle$, $\bs\in\Z^\ell$. See \cite{Jacon2024} for another way to get such a formula using level-rank duality.
\end{Rem}

The following fact is mentioned without proof in both \cite{Olsson1986} and \cite{Malle1995}. 
It serves as a key lemma for controlling the rank statistic when applying the Littlewood bijection \Cref{Lit_l_bis}, and is therefore crucial for establishing the generating function of the cocores and their higher-level analogues the spetsial $d$-cores.
\begin{Lem}\label{hook_rank}
Removing a spetsial $d$-hook in $\bA$ decreases its rank by $d$.
\end{Lem}
\begin{proof}
Let $\beta\in \bA^{(i)}$ such that $\beta-d\notin \bA^{(i+1)}$ (taking the indices mod $\ell$). Removing this $d$-hook yields the abacus $\bA'$ which differs from $\bA$ by having $\bA'^{(i)}=\bA^{(i)}\setminus \{\beta\}$ and $(\bA')^{(i+1)}=\bA^{(i+1)}\sqcup\{\beta-d\}$, and $(\bA')^{(j)}=\bA^{(j)}$ for all $j\notin\{i,i+1\}$. As in the definition of the rank, let $m=\max\{z\in\mathbb{Z}\;\mid\; z\notin \bA^{(j)}\hbox{ for some }j\in\{1,\ldots,\ell\}\}$. Similarly, let $m'=\max\{z\in\mathbb{Z}\;\mid\; z\notin (\bA')^{(j)}\hbox{ for some }j\in\{1,\ldots,\ell\}\}$. Necessarily $\beta \geq m+d$.\\

{\em Case 1: $\beta>m+d$}.  Then $m'=m$ and $c'=c$, so it is immediate from the formula for the rank that $\rk (\bA')=\rk (\bA)-d$.\\

{\em Case 2: $\beta=m+d$ and $m'=m$.} Then $c=c'$ and again it is immediate from the formula for the rank that $\rk (\bA')=\rk (\bA)-d$.\\

{\em Case 3: $\beta=m+d$ and $m'>m$.} Then we have $m'=m+k$ for some $1\leq k\leq d$. It follows that $c'=c-\ell k$. We compute:
\begin{align*}
\rk(\bA')&=\sum_{\substack{\beta\in\{\!\{\!\bA'\!\}\!\}  \\ \beta\geq m'}}\beta-c'm'-\left\lfloor \frac{(c'-1)(c'-\ell+1)}{2\ell}\right\rfloor\\
&= \sum_{\substack{\beta\in\{\!\{\!\bA\!\}\!\} \\ \beta\geq m}}\beta -\left(\sum_{i=0}^{k-1}\ell(m+i)\right)-d-(c-\ell k)(m+k)-\left\lfloor \frac{((c-1)-\ell k)((c-\ell+1)-\ell k)}{2\ell} \right\rfloor\\
&=\sum_{\substack{\beta\in\{\!\{\!\bA\!\}\!\} \\ \beta\geq m}}\beta-\ell k m -\ell\left(\frac{k(k-1)}{2}\right)-d-cm+\ell km-kc+\ell k^2-\left\lfloor \frac{(c-1)(c-\ell+1)}{2\ell} -kc+\ell\left(\frac{k(k+1)}{2}\right)\right\rfloor \\
&=\sum_{\substack{\beta\in\{\!\{\!\bA\!\}\!\} \\ \beta\geq m}}\beta-cm-d+\ell \left(\frac{k(k+1)}{2}\right)-kc-\left\lfloor \frac{(c-1)(c-\ell+1)}{2\ell}\right\rfloor  +kc-\ell\left(\frac{k(k+1)}{2}\right)\\
&=\sum_{\substack{\beta\in\{\!\{\!\bA\!\}\!\} \\ \beta\geq m}}\beta-cm-\left\lfloor \frac{(c-1)(c-\ell+1)}{2\ell}\right\rfloor -d\\
&=\rk(\bA)-d.
\end{align*}
\end{proof}

\medskip

Recall Bijection \Cref{Lit_l_bis}, which associates to any $|\bla,\bsig\rangle \in\cU^\ell_{\si}$ the pair consisting of its spetsial $d$-core $|\bla,\bs\rangle_{[d]}$, which is again a charged $d$-partition in $\cU^\ell_{\si}$, and its spetsial $d$-quotient $|\bla,\bsig\rangle^{[d]}$, which is an $e$-partition for $e=\ell d$.
By \Cref{hook_rank}, this bijection respects the rank in the following way:

\begin{Prop}\label{rk_vs_lit} Let $|\bla,\bsig\in\cU^\ell$ and let $d\in\Z_{\geq 1}$. Set $e=d\ell$ and $\bmu=|\bla,\bsig\rangle^{[d]}\in\cP^e$. 
We have \[ \rk(|\bla,\bsig\rangle) = \rk \left( |\bla,\bsig\rangle_{[d]} \right) + d |\bmu |. \]
\end{Prop}
Combining this statement with \Cref{compare_cores} yields a similar statement for $e$-cores and $e$-quotients provided that $\ell$ divides $e$.
\begin{Cor}\label{rk_vs_lit_bis} Let $e\in\Z_{\geq 2}$ and suppose $\ell$ divides $e$. Set $d=e/\ell$. Let $|\bla,\bs\rangle\in\cU^\ell$.
 Set $\bmu=|\bla,\bs\rangle^{(e)}\in\cP^e$ and set $|\bnu,\br\rangle=|\bla,\bs\rangle_{(e)}\in\cU^\ell$. Then
\[ \rk(|\bla,\bs-d\brho\rangle)= \rk(|\bnu,\br-d\brho\rangle)+d|\bmu|. \]
\end{Cor}

We have finally established the full analogue of the Littlewood Decomposition for charged multipartitions, which controls what happens to the rank under the bijections \Cref{Lit_l,Lit_l_bis}.

\section{Enumerating cocores}
\label{sec_enum_cocores}

The goal of this section is to derive the generating functions for the $d$-cocores (one for ``type $B/C$ symbols," and one for ``type $D/{^2}D$ symbols"), and prove positivity of this generating function whenever $d\geq 2$.

\subsection{Rank generating function of charged bipartitions}\label{rk gf level 2}

Let $\si\in\Z$ and consider the generating function 
\[
\sfu^2_\si(q) = \sum_{ |\bla,\bsig\rangle \in \cU^2_\si } q^{\rk(|\bla,\bsig\rangle)}.
\] 
As we will see in the more general setting of \Cref{sec_rk_gf}, in fact $\sfu^2_\si$ only depends on the parity of $\si$. We summarize this in the following lemma.

\begin{Lem}
We have
$$\sfu^2_\si(q) = 
\left\{
\begin{array}{ll}
\sfu^2_1(q) & \text{\quad if $\si$ is odd,}
\\
\sfu^2_0(q) & \text{\quad if $\si$ is even.}
\end{array}
\right.
$$
\end{Lem}

This yields our first product formulas.
\begin{Lem}\label{lem_f_prod}
We have 
$\sfu^2_1(q) = 2 \psi(q^2) \sfp(q)^2$
and $\sfu^2_0(q) = \phi(q) \sfp(q)^2.$
\end{Lem}

\begin{proof} Assume first that $\bs=(\si_1,\si_2)\in\Z^\ell[1]$. We have $\bsig=(-t,t+1)$ by setting $t=-\si_1$.
We apply \Cref{rk_t} (1) and obtain
\begin{align*}
\sfu^2_1(q)
&
=  \sum_{\substack{\bla\in\cP^2 \\ t \in\Z }} q^{ |\bla| + t^2+t }
=  2\, \left(\sum_{\bla\in\cP^2} q^{ |\bla|} \right)  \left(\sum_{t \in\N } (q^2)^{\frac{t(t+1)}{2}} \right)
= 2 \sfp(q)^2 \sfc_2(q^2).
\end{align*}

Assume now that  $\bs=(\si_1,\si_2)\in\Z^\ell[0]$. We have $\bsig=(-t,t)$ by setting $t=-\si_1$.
We apply \Cref{rk_t} (2) and obtain
\[
\sfu^2_0(q)
 =  \sum_{\substack{\bla\in\cP^2 \\ t \in\Z }} q^{ |\bla| + t^2} 
 =  \left(\sum_{\bla\in\cP^2} q^{ |\bla|} \right) \left(\sum_{t \in\Z } q^{t^2 } \right) 
= \sfp(q)^2 \phi(q).
\]
\end{proof}

\subsection{Product formula and positivity for the $d$-cocores}

Consider the rank generating function of the $d$-cocores.
Similarly to the previous subsection, we actually consider two separate series, depending on the parity of $|\bsig|$. Set

$$\sfz^2_{d,1}(q) = \sum_{|\bla,\bsig\rangle \in \cC_{[d,1]}^2 } q^{\rk(|\bla,\bsig\rangle)} 
\mand
\sfz^2_{d,0}(q) = \sum_{|\bla,\bsig\rangle \in \cC_{[d,0]}^2 } q^{\rk(|\bla,\bsig\rangle)}.$$

\begin{Thm}\label{thm_prod_cocores} Let $d$ be a positive integer. 
The following product formulas hold.
\begin{enumerate}
    \item $\sfz^2_{d,1}(q) = 2\psi(q^2) \sfc_d(q)^2.$
\item $\sfz^2_{d,0}(q) = \phi(q) \sfc_d(q)^2.$
\end{enumerate}
\end{Thm}

\begin{proof}
\begin{enumerate}
    \item 
We have $\sfu^2_1(q) = 2\psi(q^2) \sfp(q)^2$  by \Cref{lem_f_prod}.
On the other hand, 
from \Cref{rk_vs_lit}, we obtain
$$\sfu^2_1 (q) =\sfz^2_{d,1}(q)\sfp(q^{d})^{2d}.$$
Therefore, we have
\begin{align*}
\sfz^2_{d,1} (q) 
= \frac{\sfu^2_1(q)}{\sfp(q^{d})^{2d}} 
= \ds
\ds 2\psi(q^2) \left(\frac{ \sfp(q)}{\sfp(q^{d})^d}\right)^2
= 2\psi(q^2) \sfc_d(q)^2.
\end{align*}
\item 
We have $\sfu^2_0(q) = \phi(q) \sfp(q)^2$  by \Cref{lem_f_prod}.
Similarly as before,
from \Cref{rk_vs_lit}, we obtain
\[
\sfu^2_0 (q) =\sfz^2_{d,0}(q)\sfp(q^{d})^{2d}.
\]
Therefore, we have
\begin{align*}
\sfz^2_{d,0} (q) 
= \frac{\sfu^2_0(q)}{\sfp(q^{d})^{2d}} 
= \ds \phi(q) \left(\frac{ \sfp(q)}{\sfp(q^{d})^d}\right)^2
= \phi(q) \sfc_d(q)^2.
\end{align*}
\end{enumerate}
\end{proof}

We are ready to determine positivity of the generating functions of cocores.
Recall that, for $\si\in\{0,1\}$,  
$$\sfz^2_{d,\si}(q) = \sum_{n\in\N} z^2_{d,\si}(n) q^n$$
where $z^2_{d,\si}(n)$ denotes the number of $d$-cocores of rank $n$ with charge summing to $\si$.

\medskip

We will use the following lemma on several occasions.
\begin{Lem}\label{pos phipsipsi}
The coefficient of $q^n$ in $\phi(q)\psi(q)^2=\phi(q)^2\psi(q^2)$ is positive for all $n\in\N$.
\end{Lem}
\begin{proof} By \Cref{Sylvie4alt}, $\phi(q)\psi(q)^2=\phi(q)^2\psi(q^2)$. We must show that for all $n\in\N$,
\[
n=w^2+y^2+z^2+z
\]
for some $w,y,z\in \N$. 
We have
\begin{align*}
n=w^2+y^2+z^2+z 
&\iff 4n=4w^2+4y^2+4z^2+4z &   
\\
& \iff  4n+1=(2w)^2+(2y)^2+(2z+1)^2.
\end{align*}
By Legendre's Three-Square Theorem, there exist
 $a,b,c\in\N$ such that 
$$4n+1=a^2+b^2+c^2.$$
Reducing this equation modulo $4$, we see that
two of these squares must be even and the other odd,
which is exactly what we wanted.
\end{proof}

The following is an analogue of \Cref{grono} on the positivity of the generating function of $e$-core partitions. 
\begin{Thm}\label{thm_positivity_cocores}
Let $d$ be a positive integer. Then:
\begin{enumerate}
\item $z^2_{d,1}(n)>0$ for all $n\in\N$ if and only if $d\geq 2$,
\item $z^2_{d,0}(n)>0$ for all $n\in\N$ if and only if $d\geq 2$. 
\end{enumerate}
\end{Thm}

\begin{proof}
\begin{enumerate}
\item The idea is to use \Cref{grono} for $d\geq 4$,
and to treat the remaining cases by hand, which are dealt with using Legendre's Three-Square Theorem.
\begin{itemize}
\item  Assume $d=1$. Then $\sfc_d(q)=1$, 
so $a_{d,1}(n)=0$ except when $n$ is a pronic number. 
\item  Assume $d=2$. 
We have 
\[
\sfz^2_{2,1}(q)=\psi(q^2)\psi(q)^2=\psi(q^2)(\psi(q^2)\phi(q))=\phi(q)\psi(q^2)^2
\]
by applying \Cref{Sylvie4alt}. We must show that any $n\in\N$ may be written as $n=x^2+y^2+y+z^2+z$ for some $x,y,z\in\N$. Observe that
\begin{align*}
n=x^2+y^2+y+z^2+z 
&\iff 4n=4x^2+4y^2+4y+4z^2+4z &   
\\
& \iff  4n+2=(2x)^2+(2y+1)^2+(2z+1)^2.
\end{align*}
By Legendre's Three-Square Theorem, there exist
 $a,b,c\in\N$ such that 
$$4n+2=a^2+b^2+c^2.$$
Reducing this equation modulo $4$, we see that
one of these squares must be even and the other two odd, which concludes the proof.
\item  Assume $d=3$.
We have $\sfz^2_{3,1}(q) = 2\sfc_2(q^2)\sfc_3(q)^2$,
so $\frac{1}{2}z^2_{3,1}(n)$ is the number of ways to write $n$ as
the sum of a pronic number and two integers that are sizes of a $3$-core. Let $n\in\N$.
By \Cref{lem_size_3cores}, it suffices to write $n=w^2+y^2+z^2+z$
for some $w,y,z\in \N$, which we can do by \Cref{pos phipsipsi}.

    \item Assume $d\geq 4$.
We have, by \Cref{thm_prod_cocores}, $\sfz^2_{d,1}(q)=2\sfc_d(q)^2\sfc_2(q^2)
=2\sfc_d(q)^2\left(1+\sum\limits_{t> 0}q^{t^2+t}\right)$. 
Thus, for all $n\in\N$, 
$z^2_{d,1}(n)$ is at least twice the coefficient of $q^n$ in $\sfc_d(q)^2$,
so in particular $z^2_{d,1}(n)>c_{d}(n)>0$ by \Cref{grono}. 
\end{itemize}
\item We use a similar strategy as in (1).
\begin{itemize}
\item  Assume $d=1$. Then clearly $z^2_{d,0}(n)>0$ if and only if $n$ is a square. 
\item Assume $d=2$. 
We have \[
\sfz^2_{2,0}(q) = \phi(q)\sfc_2(q)^2=\phi(q)\psi(q)^2=\phi(q)^2\psi(q^2)
\]
by \Cref{Sylvie4alt}. This has positive coefficients by \Cref{pos phipsipsi}.
\item Assume $d=3$.
We have this time $\sfz^2_{3,0}(q) = \phi(q)\sfc_3(q)^2$,
so by \Cref{lem_size_3cores} it suffices to write $n$ as the sum of
a square, another square and a pronic number.
This is exactly the situation of (1), Case $d=3$, so we are done.
\item Assume $d\geq 4$.
Recall that \Cref{thm_prod_cocores} gives us
$\sfz^2_{d,0}(q) = \phi(q)\sfc_d(q)^2
=\left(1+\sum\limits_{t> 0} 2q^{t^2}\right)\sfc_d(q)^2$. 
In particular, for all $n\in\N$, 
$z^2_{d,0}(n)\geq c_{d}(n)>0$ by \Cref{grono}. 
\end{itemize}
\end{enumerate}
\end{proof}

For reasons that will become clear later (see \Cref{sec_RT}), 
we also want to study positivity of the following generating functions:

$$\overline{\sfz^2_{d,1}}(q) 
= \frac{1}{2} \sum_{n\in\N}z^2_{d,1}(n) q^n 
= \sfc_2(q^2) \sfc_d(q)^2 $$
and 
$$
\begin{array}{ccccc}
\sfz^{2,+}_{d,0}(q) 
& =&\ds \sum_{n\in\N} z^{2,+}_{d,0}(n) q^n & =&\ds \left(\sum_{t\in 2\Z} q^{t^2}\right) \sfc_d(q)^2
\\
\sfz^{2,-}_{d,0}(q) 
& =&\ds \sum_{n\in\N} z^{2,-}_{d,0}(n) q^n & =& \ds\left(\sum_{t\in 2\Z+1} q^{t^2}\right) \sfc_d (q)^2.
\end{array}
$$
Thus we have $\sfz^2_{d,0}(q)=\sfz^{2,+}_{d,0}(q)+\sfz^{2,-}_{d,0}(q)$ and $\sfz^2_{d,1}(q)=2\overline{\sfz^2_{d,1}}(q)$. Obviously, we have the following immediate corollary of \Cref{thm_positivity_cocores}.
\begin{Cor}\label{pos_typeBCcocores}
The coefficients of $\overline{\sfz^2_{d,1}}(q)$ are all positive if and only if $d\geq 2$.
\end{Cor}

Regarding $\sfz^{2,\pm}_{d,0}(q)$, we introduce the following notion, inspired by \cite[Theorem 1.1]{OhSun2009}.

\begin{Def}
An \textit{Oh-Sun} number is a pronic number $n=k(k+1)$ such that all prime divisors of $2k+1$ are congruent to $1$ modulo $4$.
\end{Def}

\begin{Thm}\label{thm_typesd_cocores}
\begin{enumerate}
\item  $z^{2,+}_{d,0}(n)>0$ for all $n\in\N$ if and only if $d\geq 2$.
\item  $z^{2,-}_{d,0}(n)>0$ for all $n\in\N\setminus\{0\}$ if and only if $d\geq 3$.
Moreover, $z^{2,-}_{2,0}(n)=0$ if and only if $n=0$ or $n$ is an Oh-Sun number. 
\end{enumerate}
\end{Thm}

\begin{proof}
\begin{enumerate}

\item This situation is similar to Theorem \ref{thm_positivity_cocores}(2) but considering only even square terms of $\phi(q)$.

\begin{itemize}
\item \textit{Case $d=1$.} We  have $\sfz^{2,+}_{d,0}(n) > 0 $ if and only if $n$ is an even square.
\item 
\textit{Case $d=2$.}
We want to write 
$n = (2w)^2 + \frac{y^2+y}{2} + \frac{z^2+z}{2}$
which is equivalent to 
$8n+2 = 32 w^2 + (2y+1)^2+ (2z+1)^2.$
But any expression of the form $8n+2=32w^2+a^2+b^2$
with $a,b\in\N$ forces $a,b$ to be odd.
So it suffices to show that for all $n\in\N$,
$8n+2=32w^2+a^2+b^2$ for some $w,a,b\in\N$.
In other terms, it suffices to know
that the quadratic form $32X^2+Y^2+Z^2$ is $(8,2)$-universal, which 
in fact was proved in 
\cite{PW2018}. 
\item \textit{Case $d=3$.}
We need to show that $n=(2a)^2+b+c$ where $c_3(b),c_3(c)>0$. Clearly this is true if $n=0$, so assume $n>0$. If $n\equiv1$ or $2\bmod 4$ then $n$ is a sum of three squares by Lagrange's theorem, and one of these squares must be even as the squares mod $4$ are $0$ and $1$. The other two squares are sizes of $3$-cores by Lemma \ref{lem_size_3cores}, so we are done in this case. If $n\equiv 3\bmod 4$, write $n-1=x^2+y^2+z^2$ with $x$ even. Then $n=x^2+y^2+(z^2+1)$ does the job, since not only $y^2$ but also $z^2+1$ is the size of a $3$-core, by Lemma \ref{lem_size_3cores}. If $n\equiv 0\bmod 4$, write $n-2=x^2+y^2+z^2$ with $x$ even. Then $n=x^2+(y^2+1)+(z^2+1)$ is the desired expression.
\item \textit{Case $d\geq4$.} We have $\sfz^{2,+}_{d,0}(n) > 0 $ for all $n\geq 0$ by \Cref{grono}.
\end{itemize}

\item Note first that
$\sfz^{2,-}_{d,0}(0)=0$ for all $d\geq 1$,
since $\sfz^{2,-}_{d,0}>0$ implies that an odd square appears in the decomposition of $n$.
So we restrict the study to $n>0$.

\begin{itemize}
\item \textit{Case $d=1$.} We  have $\sfz^{2,-}_{d,0}(n) > 0 $ if and only if $n$ is an odd square.
\item \textit{Case $d=2$.} We  have $\sfz^{2,-}_{d,0}(n) > 0 $ if and only if $n$ is
a sum of an odd square and two triangular numbers.
By \cite[Theorem 1.1]{OhSun2009}, this is equivalent to $n$ not being 
an Oh-Sun number.

\item \textit{Case $d=3$.} We must show for any positive integer $n$, $n=(2a+1)^2+b+c$ for some $a,b,c\in\mathbb{N}$ such that $c_3(b),c_3(c)\neq 0$. First suppose that $n\equiv1,2,3,5$ or $6\bmod 8$. Then $n$ is a sum of three squares by Lagrange's three-square theorem. As any even square is congruent to $0$ or $4$ mod $8$, at least one of these squares must be odd. The remaining two squares are sizes of $3$-cores by Lemma \ref{lem_size_3cores}, so we are done in this case. Next, suppose $n\equiv 7 \bmod 8$ or $n\equiv 4 \bmod 8$. We may write $n-1=x^2+y^2+z^2$ with $x$ odd. Then $n=x^2+(y^2+1)+z^2$, and by Lemma \ref{lem_size_3cores} both $y^2+1$ and $z^2$ are sizes of $3$-cores. Similarly, if $n\equiv 0\bmod 8$, write $n-2= x^2+y^2+z^2$ with $x$ odd. Then $n=x^2+(y^2+1)+(z^2+1)$, with $y^2+1$ and $z^2+1$ being the sizes of two $3$-cores by Lemma \ref{lem_size_3cores} again.
\item \textit{Case $d\geq4$.} We  have $\sfz^{2,-}_{d,0}(n) > 0 $ for all $n>0$ by \Cref{grono}.

\qedhere
\end{itemize}
\end{enumerate}

\end{proof}

\subsection{Defect $0$ unipotent blocks of finite classical groups}\label{sec_RT}
Our goal now is to interpret the combinatorial results we have obtained as results about the modular representation theory of finite groups of Lie type. 
We have seen that Granville and Ono's theorem stating that $c_e(n)>0$ for all $n\in\N$ provided $e\geq 4$ (\Cref{grono}) may be reformulated as the statement that $\mathbb{k}\mathrm{GL}_n({{v}})$ has a defect $0$ unipotent block for all $n\in\N$ provided that the order of ${{v}}\bmod \mathrm{char}(\mathbb{k})$ is greater than or equal to $4$ (\Cref{def0gln}). In this section, we deduce analogous statements for finite symplectic and (special) orthogonal groups from our results on the generating functions of cocores. 

\medskip

By a {\em finite classical group} $G_n({{v}})$, we mean one of the following matrix groups with entries in a finite field $\mathbb{F}_{{v}}$, classified by Dynkin type as follows:
\renewcommand{\arraystretch}{1.2}
\[
\begin{array}{lllll}
\hline 
 \text{Group} & \mathrm{SO}_{2n+1}({{v}}) &  \mathrm{Sp}_{2n}({{v}}) & \mathrm{O}_{2n}^+({{v}}) & \mathrm{O}_{2n}^-({{v}})\\[3pt]\hline
 \text{Type} & B_n & C_n & D_n & {}^2 D_n \\ \hline
 \end{array}
 \]
We will always assume that ${{v}}$ is a power of an odd prime. 

\medskip

Let $G_n({{v}})$ be a finite classical group. The unipotent characters of $G_n({{v}})$ are partitioned into {\em Harish-Chandra series} which record if and how they are obtained by Harish-Chandra induction from proper Levi subgroups. Each Harish-Chandra series is labeled by a unique {\em cuspidal unipotent character}, which is a unipotent character that cannot be obtained by Harish-Chandra induction from a proper Levi subgroup. In the case of $\mathrm{GL}_n({{v}})$, the only cuspidal is the trivial character of the trivial group $S_0=\{1\}$. The theory becomes more interesting outside type $A$.  The cuspidal unipotent characters for the finite classical groups may be parametrized by the number of solutions by $t\in\Z$ to the equations in the table below:
\[
\begin{array}{lllll}
\hline
 \text{Group} & \mathrm{SO}_{2n+1}({{v}}) &  \mathrm{Sp}_{2n}({{v}}) & \mathrm{O}_{2n}^+({{v}}) & \mathrm{O}_{2n}^-({{v}})\\[3pt]\hline
 \text{Unip. Cuspidal(s)}  &  n=t^2+t,  \text{ $t$ even} &  n=t^2+t,  \text{ $t$ odd}  & n=t^2,  \text{ $t$ even} & n=t^2,  \text{ $t$ odd} 
 \\ \hline
\end{array}
\]
Thus, the groups $\mathrm{SO}_{2n+1}({{v}})$ and $ \mathrm{Sp}_{2n}({{v}})$ each have a unique unipotent cuspidal if and only if $n$ is a pronic number; the group $ \mathrm{O}_{2n}^+({{v}})$ has a unique unipotent cuspidal when $n=0$, and otherwise has two distinct unipotent cuspidals if and only if $n$ is an even square bigger than $0$; while the group $ \mathrm{O}_{2n}^-({{v}})$ has two distinct unipotent cuspidals if and only if $n$ is an odd square \cite{FongSrinivasan1989,Wald04}.

\medskip

Within a given Harish-Chandra series, the unipotent characters are labeled by the irreducible representations of a Weyl group.
In the case of the finite classical groups described above, these Weyl groups are all of type $B$, with the exception of the principal series of $ \mathrm{O}_{2n}^+({{v}})$ which is of type $D_n$.
As mentioned earlier, Lusztig introduced $2$-symbols to parametrize the unipotent characters of the finite classical groups \cite{Lusztig1977}. 
We will use charged bipartitions $|\bla,\bsig\rangle$. The difference worth noting is that Lusztig's symbols are modded out by cyclic permutation, whereas we do not consider $|(\la^1,\la^2),(\si_1,\si_2)\rangle$ as the same symbol as $|(\la^2,\la^1),(\si^2,\si^1)\rangle$. Instead, we will assign different charged bipartitions to type $B$ and type $C$ characters.
The parametrization of the charges $\bsig$ that we will use comes from a categorical action, as explained in \cite{DVV2}. The relevant charges $\bsig$ label the unipotent cuspidals by the discussion above. The charge $\bsig$ indexes the Harish-Chandra series of the unipotent character, while the size of $\bla$ records its cuspidal depth.

\medskip

The relevant charges $\bsig$ for the finite classical groups are given by two separate formulas, one for types $B$ and $C$ together and one for types $D$ and ${^2}D$ together, and are parametrized in both cases by $t\in\Z$. For each $t\in\Z$ and $\bla\in\cP^2$, there is a unique symbol $|\bla,(-t,t+1)\rangle$ labeling a unipotent character of type $B$ or $C$. For each $t\in\Z$, $t\neq 0$, and $\bla\in\cP^2$, there is a unique symbol $|\bla,(-t,t)\rangle$ labeling a unipotent character of type $D$ or $^{2}D$. When $t=0$ and $\bla=(\lambda^1,\lambda^2)$, $\la^2\neq\la^1$, the symbols $|(\lambda^1,\lambda^2),(0,0)\rangle$ and $|(\lambda^2,\lambda^1),(0,0)\rangle $ label a single unipotent character of type $D$; while for each partition $\lambda$, there are two unipotent characters associated to the symbol $|(\lambda,\lambda),(0,0)\rangle$, which may be labeled as $|(\lambda,\lambda),(0,0)\rangle^+$ and $|(\lambda,\lambda),(0,0)\rangle^-$. This is because $\bsig=(0,0)$ is the charge for the principal series of $\mathrm{O}_{2n}^+({{v}})$ (induced from the trivial character of the trivial group) and the unipotent characters in the principal series are parametrized by the irreducible representations of the Weyl group $D_n$.

\medskip

The table below summarizes our assignment of charges $\bsig$ depending on $t\in\Z$ 
to the unipotent characters of the four types of finite classical groups and the ranks of the resulting symbols $|\bla,\bsig\rangle$, $\bla\in\cP^2$.
\[
\begin{array}{lllll}
\hline
 \text{Type} & B & C & D & {^2}D\\[3pt]
 \hline
  \text{Charges } \bsig & (-t,t+1),  \text{ $t$ even}  & (-t,t+1),  \text{ $t$ odd} &  (-t,t),   \text{ $t$ even}  &(-t,t),  \text{ $t$ odd} \\
    \text{Rank of }|\bla,\bsig\rangle  & t^2+t+|\bla| &  t^2+t+|\bla| & t^2+|\bla|& t^2+|\bla| 
    \\ \hline
   \end{array}
\]
We set $\mathsf{Unip}(G_n)$ to be the set of unipotent characters of $G_n({{v}})$, for $G_n\in\{B_n,C_n,D_n,{^2}D_n\}$. By definition, $\rk(|\bla,\bsig\rangle)=n$ if and only if $|\bla,\bsig\rangle$ labels a unipotent character of a group of type $G_n$.

\medskip

For $G$ one of the Dynkin types $B$, $C$, $D$ or ${}^2D$, we then set
\[
\mathsf{Unip}(G)=\bigcup_{n\in\N} \mathsf{Unip}(G_n),
\]
the set of unipotent characters of all the groups $G_n$ for $n\in\N$. Thus, if $G\in\{B,C\}$ we have a parametrization of $\mathsf{Unip}(G)$ by charged bipartitions as follows:
\[
\mathsf{Unip}(G)=\begin{cases}
\{|\bla,\bsig\rangle\in\cU^2_1\;\mid\; \bsig=(-t,t+1)\hbox{ for some }t\in2\Z \}& \hbox{ if }G=B,\\
\{|\bla,\bsig\rangle\in\cU^2_1\;\mid\; \bsig=(-t,t+1)\hbox{ for some }t\in2\Z +1 \}& \hbox{ if }G=C.
\end{cases}
\]
If $G=D$ then we have to slightly tweak the parametrization in the case of the principal series of type $D$ as explained above, yielding:
\begin{align*}
\mathsf{Unip}(D)&=
\{|\bla,\bsig\rangle\in\cU^2_0\;\mid\; \bsig=(-t,t)\hbox{ for some }t\in2\Z, t\neq 0 \}\\
&\qquad\qquad\cup\{|\bla,(0,0)\rangle\;\mid\; \la^1\neq\la^2\}/\left(|\la^1.\la^2,(0,0)\rangle\sim|\la^2.\la^1,(0,0)\rangle\right)
\cup\{|\la.\la,(0,0)\rangle^\pm\}.
\end{align*}
If $G={^2}D$, then we have the parametrization
\[
\mathsf{Unip}({^2}D)=\{ |\bla,\bsig\rangle\in\cU^2_0\;\mid\; \bsig=(-t,t)\hbox{ for some }t\in2\Z+1 \}.
\]

\medskip

We then define the generating function of the unipotent characters of type $G\in\{B,C,D,{^2}D\}$ as follows:
\[
\sfu_G({{q}})=\sum_{|\bla,\bsig\rangle\in\mathsf{Unip}(G) } {{q}}^{\mathrm{rk}(|\bla,\bsig\rangle)}.
\]
We thus have 
\[
\sfu_B({{q}})+\sfu_C({{q}})=\sfr^2_1({{q}}),
\]
while $\sfu_D({{q}})+\sfu_{{^2}D}({{q}})$ is very close to being equal to $\sfr^2_0({{q}})$. Note that $\sfu_B({{q}})=\sfu_C({{q}})$, as the abaci labeling type $C$ are the flip upside-down of the abaci labeling type $B$, and this does not affect the rank by \Cref{lem_invariance}.

\medskip

Cocores enter the story when we consider modular representations of these groups, that is, representations of $G_n({{v}})$ over a field $\mathbb{k}$ of characteristic ${{l}}>0$. We always assume that ${{l}}$ is odd and coprime to ${{v}}$, the latter condition being known as {\em non-defining characteristic}. 
Let $e$ be the multiplicative order of ${{v}}$ in $\mathbb{k}^\times$. The modular representation theory of $G_n({{v}})$ in non-defining characteristic falls into two cases:
\begin{itemize}
\item An easier case: when $e$ is odd (the {\em linear prime case}), the representation theory is determined from that of $\mathrm{GL}_m({{v}})$ for various $m\leq n$, which is understood \cite{GruberHiss}.
\item A harder case: when $e$ is even (the {\em unitary prime case}), there are fundamental open questions (such as determining decomposition numbers).\end{itemize}
Fong and Srinivasan proved the analogue of the Nakayama Conjecture for the modular representations of finite classical groups \cite{FongSrinivasan1989}:
\begin{itemize}
\item if $e$ is odd, two unipotent characters belong to the same block of $\mathbb{k}G_n({{v}})$ if and only if their symbols have the same pair of charged $e$-cores obtained by taking the $e$-core of each partition in the associated bipartition separately and keeping the charge the same.
\item if $e$ is even, two unipotent characters belong to the same block of $\mathbb{k}G_n({{v}})$ if and only if their symbols have the same $d$-cocore, where $d=\frac{e}{2}$.
\end{itemize}
(We note that in this theory, symbols i.e. abaci are only defined up to cyclic permutation of the two rows, and the same is true for the cocores, as they are also symbols).

\medskip

We may now rephrase \Cref{thm_positivity_cocores} and \Cref{thm_typesd_cocores} as statements about the existence of defect $0$ unipotent blocks of finite classical groups. For $G\in\{B,C,D,{^{{2}}D}\}$, let $\sfz_{d,G}(q)$ be the generating function of the $d$-cocores that are symbols of type $G$.
\begin{Cor}\label{cor_fglt}
Let ${{v}}$ be a power of an odd prime, and let $G_n({{v}})\in\{ \mathrm{SO}_{2n+1}({{v}}),  \;\mathrm{Sp}_{2n}({{v}}),\; \mathrm{O}_{2n}^+({{v}}), \;\mathrm{O}_{2n}^-({{v}})\}$ be a finite classical group. Consider $\mathbb{k}G_n({{v}})$ for a field $\mathbb{k}$ of odd characteristic ${{l}}>0$ not dividing ${{v}}$, and let $e$ be the order of ${{v}}$ modulo ${{l}}$.
\begin{enumerate}
\item Suppose $e$ is odd (so that ${{l}}$ is a linear prime for $G_n({{v}})$). Then $G_n({{v}})$ has a unipotent $l$-block of defect $0$ for every $n\in\N$ if and only if $e\geq 3$.
\item Suppose $e$ is even (so that ${{l}}$ is a unitary prime for $G_n({{v}})$). 
	\begin{enumerate}
	\item If $G_n({{v}})\in\{\mathrm{SO}_{2n+1}({{v}}), \; \mathrm{Sp}_{2n}({{v}})\}$, then $G_n(v)$ has a unipotent $l$-block of defect $0$ for every $n\in\N$ if and only if $e\geq 4$. 
	\item For every $n\in\N$ there is a defect $0$ unipotent $l$-block of at least one of $\mathrm{O}_{2n}^+({{v}}),\;\mathrm{O}_{2n}^-({{v}})$ if and only if $e\geq 4$. 
	\end{enumerate}
\end{enumerate}
\end{Cor}
\begin{proof} 
\begin{enumerate}
\item This follows from the characterization of the blocks in the linear prime case due to Fong and Srinivasan and \Cref{thm_typesd_cocores}. By \Cref{rk_ell=123}(2), the rank of $|\bla,\bsig\rangle\in\cU^2_1$ is $|\bla|+t^2+t$ for some $t\in\Z$, while the rank of $|\bla,\bsig\rangle\in\cU^2_0$ is $|\bla|+t^2$ for some $t\in\Z$.
By the parametrization of the type $B$ and $C$ symbols, the generating function with respect to the rank of all $|\bla,\bsig\rangle$ of type $B$ or $C$ such that $\bla$ is a pair of $e$-cores is $\psi({{q^2}})\sfc_e({{q}})^2=\overline{\sfz^2_{e,1}}({{q}})$. (We are enumerating all pairs of $e$-cores for each charge $\bsig$ of a given type, say type $B$, and these charges have rank $t^2+t=0,2,6,\ldots$) Likewise, using the parametrization of the type $D$ and $^2D$ symbols, the generating function with respect to rank of 
all $|\bla,\bsig\rangle$ of type $D$ such that $\bla$ is a pair of $e$-cores is
 $\ds \left(\sum_{t\in 2\Z} q^{t^2}\right)\sfc_e({{q}})^2=\sfz^{2,+}_{e,0}({{q}})$, while for type $^2D$ symbols it is $\ds \left(\sum_{t\in 2\Z+1} q^{t^2}\right)\sfc_e({{q}})^2=\sfz^{2,+}_{e,0}({{q}})$. (Again, we take the ranks of all type $D$ charges, which are all $t^2$, $t\in\Z$ even, and then for each charge we count all pairs of $e$-cores. Likewise for type $^2D$, where the charges have rank $t^2$, $t\in\Z$ odd). In types $B,C,{^{{2}}D}$ these are precisely the generating functions of the defect $0$ blocks. In type $D$, we have the issue explained above that when $\bsig=(0,0)$, some symbols label two characters while other symbols label half a character. However, this is not important if we only care about positivity of the coefficients. 
	\begin{enumerate}
	\item Set $d=\frac{e}{2}$. We have $\sfz_{d,B}({{q}})=\sfz_{d,C}({{q}})=\overline{\sfz^2_{d,1}}({{q}})=\psi({{q}}^2)\sfc_d({{q}})^2$. The result then follows from \Cref{pos_typeBCcocores}. 
	\item Set $d=\frac{e}{2}$. We have  $\sfz_{d,D}({{q}})+ \sfz_{d,{^2D}}({{q}})=\sfz^2_{d,0}(q)$. The result then follows from \Cref{thm_positivity_cocores}.
	 \end{enumerate}
\end{enumerate}
\end{proof}
 From \Cref{cor_fglt}, we conclude that  $\mathrm{SO}_{2n+1}({{v}})$ and $\mathrm{Sp}_{2n}({{v}})$ both have a defect $0$ unipotent $l$-block if and only if $e\geq 3$.
 
 \begin{Conj}\label{conj_typesd}
 Let ${{v}}$ be a power of an odd prime, let $\mathbb{k}$ be a field of odd characteristic ${{l}}>0$ not dividing ${{v}}$, and let $e$ be the order of ${{v}}$ modulo ${{l}}$. Then $\mathrm{O}_{2n}^+({{v}})$ has a unipotent $l$-block of defect $0$ for every $n\in\N$ if and only if $e\geq 3$, while $\mathrm{O}_{2n}^-({{v}})$ has a unipotent $l$-block of defect $0$ for every $n\in\N$ if and only if $e=3$ or $e\geq 5$.
 \end{Conj}
 
 The conjecture is true for odd $e$ by \Cref{cor_fglt}, so the conjecture is really about what happens when $e$ is even. The results of \Cref{escore_gf_sec} allow us to establish some cases of this conjecture.  \Cref{thm_gf_escores} says that $\sfc_{4,(0,2)}(q)$ has positive coefficients while $\sfc_{4,(0,0)}(q)$ does not, and $\sfz^2_{2,0}(q)=\sfc_{4,(0,2)}(q) + 2q \sfc_{4,(0,0)}(q)$. Moreover, as $(0,2)=(0,0)+2\brho$, we see that $\sfc_{4,(0,2)}(q)=\sfz_{2,D}({{q}})$, while $(0,0)\sim(1,1)=(-1,1)+2\brho$ and thus $2q\sfc_{4,(0,0)}(q)= \sfz_{2,^2D}({{q}})$.
Thus \Cref{conj_typesd} is true for $e=4$. If $4$ divides $e$, then $\sfc_{e,(0,2)}(q)$ has positive coefficients by \Cref{ecore_div}. This implies that  $\sfz_{\frac{e}{2},D}({{q}})$
has positive coefficients for any $e=8m+4$, $m\geq 0$, while  $\sfz_{\frac{e}{2},{^2D}}({{q}})$ has positive coefficients for any $e=8m$, $m\geq 1$. In the case $e=6$, the conjecture is equivalent to the statement that both $\sfc_{6,(0,1)}(q)$ and $\sfc_{6,(0,3)}(q)$ have positive coefficients, which appears to be true from a computer calculation.

 %%%%%%%%%%%%%%%%%%%%%%%%%%%%%%%%%%%%%
 %%%%%%%%%%%%%%%%%%%%%%%%%%%%%%%%%%%%%

\section{Enumerating spetsial cores}\label{section_spetsialcores}
\label{sec_enum_spets}
In this section we find the generating functions for cuspidal symbols, $\ell$-symbols, and spetsial $d$-cores in $\cU^\ell_0$ and $\cU^\ell_1$ and settle the question of positivity of their coefficients. The generating functions of spetsial $d$-cores are of interest for enumerating the $\Phi_e$-blocks of imprimitive spetses that consist of a single $\ell$-symbol, where $e=d\ell$. However, we also find the combinatorics interesting in that it naturally develops the level $2$ case to higher levels, in the process revealing connections to number theory (sums of squares, Theta series of lattices). We will also use the results from this section to obtain some generating functions of Jacon and Lecouvey's $(e,\bs)$-cores in \Cref{escore_gf_sec}.

\subsection{Rank generating functions of cuspidal symbols}\label{rk_cusp}
We have seen in Lemma \ref{rk_formula} that the rank of a charged $\ell$-partition $|\bla,\bsig\rangle$ may be expressed as $|\bla|$ plus a non-negatively-valued function of the charge $\bsig$, which we denoted $\rk(\bsig)$. If we restrict our attention to symbols whose charges $\bsig$ satisfy $|\bsig|=\si$ for a fixed $\si\in\mathbb{Z}$, then $\rk(\bsig)$ is given by a quadratic polynomial. This is particularly interesting in two cases: $\si=1$ and $\si=0$. These are the cases relevant for parametrizing the unipotent characters of spetses of types $G(\ell,1,n)$ and $G(\ell,\ell,n)$, respectively. In the case $\si=0$ it turns out that $\rk(\bsig)$ is given by a quadratic form, while in the case of $\si=1$ it has additional linear terms.

\medskip

Recall that for $\si\in\mathbb{Z}$, 
$\cU_\si^\ell=\{|\bla,\bsig\rangle\mid \si_1+\si_2+\ldots+\si_\ell=\si\}$ where $\bsig=(\si_1,\ldots,\si_\ell)$. 
In particular, 
\[
\cU_0^\ell=\{|\bla,\bsig\rangle\mid \si_1+\si_2+\ldots+\si_\ell=0\},\quad\quad\quad
\cU_1^\ell=\{|\bla,\bsig\rangle\mid \si_1+\si_2+\ldots+\si_\ell=1\}.
\]
\begin{Def}
Call $|\bla,\bsig\rangle\in\cU^\ell$ a {\em cuspidal symbol} if $\bla$ is the empty $\ell$-partition $\bemp$.
\end{Def}
By  \Cref{rk bsig}, we may identify $\rk(\bsig)$ with the rank of the cuspidal symbol $|\bemp,\bsig\rangle$.
We are going to study the rank generating functions of cuspidal symbols in $\cU_\si^\ell$ for $\si=0$ and $\si=1$.

\begin{Lem}\label{polynomialrankformulas}
Let $\bsig=({\si}_1,{\si}_2,\ldots,{\si}_\ell)\in\mathbb{Z}^\ell$.
\begin{enumerate}
\item Suppose ${\si}_1+{\si}_2+\ldots+{\si}_\ell=1$. Then 
\[
\rk(\bsig)=\sum_{i=1}^{\ell-1} ({\si}_i^2-{\si}_i) + \sum_{1\leq i<j\leq \ell-1} {\si}_i{\si}_j.
\]
\item
Suppose ${\si}_1+{\si}_2+\ldots+{\si}_\ell=0$. Then 
\[
\rk(\bsig)=\sum_{i=1}^{\ell-1} {\si}_i^2 + \sum_{1\leq i<j\leq \ell-1} {\si}_i{\si}_j.
\]
\end{enumerate}
\end{Lem}
\begin{proof}
By Lemma \ref{rk_formula}, we have $\rk(\bsig)=\left\lceil \frac{Y-\ell+1}{2\ell}\right\rceil$ where $Y=\sum\limits_{i=1}^\ell {\si}_i\sum\limits_{j\neq i} ({\si}_i-{\si}_j)$. We may write
\begin{align*}
Y&=\sum_{i=1}^{\ell-1}{\si}_i\sum_{1\leq i<j\leq \ell-1}({\si}_i-{\si}_j)+\sum_{i=1}^{\ell-1}{\si}_i({\si}_i-{\si}_\ell)+{\si}_\ell\sum_{j=1}^{\ell-1}({\si}_\ell-{\si}_j)\\
&=(\ell-1)\sum_{i=1}^{\ell-1}{\si}_i^2-2\left(\sum_{1\leq i<j\leq \ell-1}{\si}_i{\si}_j\right)-2{\si}_\ell\sum_{i=1}^{\ell-1}{\si}_i+(\ell-1){\si}_\ell^2.
\end{align*}
\begin{enumerate}
\item If ${\si}_1+{\si}_2+\ldots+{\si}_\ell=1$, we have ${\si}_\ell=1-\sum\limits_{i=1}^{\ell-1}{\si}_i$. Noting that \[{\si}_\ell^2=1+\sum\limits_{i=1}^{\ell-1}{\si}_i^2+2\sum\limits_{1\leq i<j\leq \ell-1}{\si}_i{\si}_j-2\sum\limits_{i=1}^{\ell-1} {\si}_i,\] 
we find that
\begin{align*}
Y&=(\ell-1)\sum_{i=1}^{\ell-1}{\si}_i^2-2\left(\sum_{1\leq i<j\leq \ell-1}{\si}_i{\si}_j\right)-2\left (1-\sum_{i=1}^{\ell-1}{\si}_i \right)\sum_{i=1}^{\ell-1}{\si}_i+\left(1-\sum_{i=1}^{\ell-1}{\si}_i \right)^2\\
&=(\ell-1)\sum_{i=1}^{\ell-1}{\si}_i^2-2\left(\sum_{1\leq i<j\leq \ell-1}{\si}_i{\si}_j\right)-2\sum_{i=1}^{\ell-1}{\si}_i+2\sum_{i=1}^{\ell-1}{\si}_i^2+4\left(\sum_{1\leq i<j\leq \ell-1}{\si}_i{\si}_j\right)\\
&\qquad\qquad+(\ell-1)\left(1+\sum\limits_{i=1}^{\ell-1}{\si}_i^2+2\sum\limits_{1\leq i<j\leq \ell-1}{\si}_i{\si}_j-2\sum\limits_{i=1}^{\ell-1} {\si}_i\right)\\
&=2\ell\left( \sum_{i=1}^{\ell-1}({\si}_i^2-{\si}_i)+\sum_{1\leq i<j\leq \ell-1}{\si}_i{\si}_j\right)+\ell-1.
\end{align*}
Thus 
\begin{align*}
\rk(\bsig)=\left\lceil \frac{ Y-\ell+1}{2\ell}\right\rceil&=\left\lceil \frac{ 2\ell\left( \sum_{i=1}^{\ell-1}({\si}_i^2-{\si}_i)+\sum_{1\leq i<j\leq \ell-1}{\si}_i{\si}_j\right)+\ell-1 -\ell+1}{2\ell}\right\rceil\\
&=\sum_{i=1}^{\ell-1}({\si}_i^2-{\si}_i)+\sum_{1\leq i<j\leq \ell-1}{\si}_i{\si}_j.
\end{align*}
\item
If ${\si}_1+{\si}_2+\ldots+{\si}_\ell=0$, we have ${\si}_\ell=-\sum\limits_{i=1}^{\ell-1}{\si}_i$ and ${\si}_\ell^2=\sum\limits_{i=1}^{\ell-1}{\si}_i^2+2\left(\sum\limits_{1\leq i<j\leq \ell-1}{\si}_i{\si}_j\right)$. Then
\begin{align*}
Y&=(\ell-1)\sum_{i=1}^{\ell-1}{\si}_i^2-2\left(\sum_{1\leq i<j\leq \ell-1}{\si}_i{\si}_j\right)+2\left(\sum\limits_{i=1}^{\ell-1}{\si}_i\right)^2+(\ell-1)\left(\sum\limits_{i=1}^{\ell-1}{\si}_i\right)^2\\
&=2\ell\left( \sum_{i=1}^{\ell-1}{\si}_i^2+\sum_{1\leq i<j\leq \ell-1}{\si}_i{\si}_j\right),
\end{align*}
from which 
\begin{align*}
\rk(\bsig)=\left\lceil \frac{ Y-\ell+1}{2\ell}\right\rceil&=\left\lceil \frac{ 2\ell\left( \sum_{i=1}^{\ell-1}{\si}_i^2+\sum_{1\leq i<j\leq \ell-1}{\si}_i{\si}_j\right) -\ell+1}{2\ell}\right\rceil\\
&=\sum_{i=1}^{\ell-1}{\si}_i^2+\sum_{1\leq i<j\leq \ell-1}{\si}_i{\si}_j.
\end{align*}

\end{enumerate}
\end{proof}

\medskip
Next, consider the generating function with respect to rank of all $|\bemp,\bsig\rangle\in\cU_1^\ell$:
\[
\sfr^\ell_1(q):=\sum_{\substack{\bsig\in\mathbb{Z}^\ell \\ \si_1+\si_2+\ldots+\si_\ell=1}} q^{\rk(\bsig)}.
\]
Similarly, consider the generating function with respect to rank of all $|\bemp,\bsig\rangle\in\cU_0^\ell$:
\[
\sfr^\ell_0(q):=\sum_{\substack{\bsig\in\mathbb{Z}^\ell \\ \si_1+\si_2+\ldots+\si_\ell=0}} q^{\rk(\bsig)}.
\]

By Lemma \ref{polynomialrankformulas}, we have:
\begin{equation}\label{charge_gf_si=1}
\sfr^\ell_1(q)=\sum_{{\si}_1,{\si}_2,\ldots,{\si}_{\ell-1}\in\mathbb{Z}}q^{\sum\limits_{i=1}^{\ell-1}({\si}_i^2-{\si}_i)+\sum\limits_{1\leq i<j\leq \ell-1}{\si}_i{\si}_j}=\sum_{{\si}_1,{\si}_2,\ldots,{\si}_{\ell-1}\in\mathbb{Z}}q^{\sum\limits_{i=1}^{\ell-1}({\si}_i^2+{\si}_i)+\sum\limits_{1\leq i<j\leq \ell-1}{\si}_i{\si}_j}
\end{equation}
and 
\begin{equation}\label{charge_gf_si=0}
\sfr^\ell_0(q)=\sum_{{\si}_1,{\si}_2,\ldots,{\si}_{\ell-1}\in\mathbb{Z}}q^{\sum\limits_{i=1}^{\ell-1}{\si}_i^2+\sum\limits_{1\leq i<j\leq \ell-1}{\si}_i{\si}_j}.
\end{equation}

\medskip
It is easy to verify the following formulas.
\begin{Lem}\label{rksumsquares} We have
\[
q\sfr^\ell_1(q^2)=\sum_{\substack{{\si}_1,{\si}_2,\ldots,{\si}_\ell\in\mathbb{Z} \\ {\si}_1+{\si}_2+\ldots+{\si}_\ell=1}}q^{{\si}_1^2+{\si}_2^2+\ldots+{\si}_\ell^2} \qquad\hbox{ and }\qquad \sfr^\ell_0(q^2)=\sum_{\substack{{\si}_1,{\si}_2,\ldots,{\si}_\ell\in\mathbb{Z} \\ {\si}_1+{\si}_2+\ldots+{\si}_\ell=0}}q^{{\si}_1^2+{\si}_2^2+\ldots+{\si}_\ell^2}.
\]
\end{Lem}
This lemma shows that the problem of enumerating cuspidal symbols in level $\ell$ is very closely related to the problem of finding all integer solutions to the Diophantine equation
 \[
 x_1^2+x_2^2+\ldots+x_\ell^2=n,
 \]
the problem of how many ways there are to write $n$ as the sum of $\ell$ squares.

\medskip

We now address the question of positivity of the coefficients of $\sfr^\ell_1(q)$ and $\sfr^\ell_0(q)$. Let $r^\ell_1(n)$ be the coefficient of $q^n$ in $\sfr^\ell_1(q)$, and let $r^\ell_0(n)$ be the coefficient of $q^n$ in $\sfr^\ell_0(q)$. Thus,  $r^\ell_1(n)$ is the number of charges $\bsig\in\Z^\ell$ of rank $n$ satisfying $|\bsig|=1$, while  $r^\ell_0(n)$ is the number of charges  $\bsig\in\Z^\ell$ of rank $n$ satisfying $|\bsig|=0$.

\begin{Thm}\label{thm_positivity_cuspidals}
Let $\ell$ be a positive integer. Then:
\begin{enumerate}
\item $r^\ell_1(n)>0$ for all $n\in\N$ if and only if $\ell\geq 4$,
\item $r^\ell_0(n)>0$ for all $n\in\N$ if and only if $\ell\geq 5$.
\end{enumerate}
\end{Thm}
\begin{proof}
\begin{enumerate}
\item 
Let $\ell=4$. We will show that 
\[
\sfr^4_1(q)=4\mathsf{c}_2(q)^3.
\]
It suffices to show that $\sfr^4_1(q^2)=4\mathsf{c}_2(q^2)^3$. We first observe that for all $t\in\mathbb{Z}$, $(-t-1)^2+(-t-1)=t^2+t$, and that $t, -t-1 $ have opposite parity. It follows that 
\[
2\mathsf{c}_2(q^2)=2\sum_{t\geq 0}q^{t^2+t}=\sum_{t\geq 0}q^{t^2+t}+q^{(-t-1)^2+(-t-1)}=\sum_{t\in\mathbb{Z}}q^{t^2+t}.
\]
Therefore,
\begin{align*}
8\mathsf{c}_2(q^2)^3&=\sum_{x,y,z\in\mathbb{Z}}q^{x^2+y^2+z^2+x+y+z}\\
&=\sum_{\substack{x,y,z\in\Z \\ x+y+z\;\mathrm{even}}}q^{x^2+y^2+z^2+x+y+z}+\sum_{\substack{x,y,z \in\Z \\ x+y+z\; \mathrm{odd}}} q^{x^2+y^2+z^2+x+y+z}\\
&=2\sum_{\substack{x,y,z\in\Z \\ x+y+z\;\mathrm{even}}}q^{x^2+y^2+z^2+x+y+z},
\end{align*}
as the map $\Z^3\rightarrow\Z^3$ sending $x$ to $-x-1$, $y$ to $-y-1$ and $z$ to $-z-1$ swaps the two sums in the second line, while $q^{x^2+y^2+z^2+x+y+z}=q^{(-x-1)^2+(-y-1)^2+(-z-1)^2-x-1-y-1-z-1}$. 

For $a,b,c\in\mathbb{Z}$, set $x=a+b$, $y=b+c$, $z=c+a$. This defines a bijection between $\mathbb{Z}^3$ and the set $\{(x,y,z)\in\mathbb{Z}^3\mid x+y+z\hbox{ is even}\}$. Then observe that
\begin{align*} x^2+y^2+z^2+x+y+z &= (a+b)^2+(b+c)^2+(c+a)^2 + (a+b)+(b+c)+(c+a)\\ &= 2(a^2+b^2+c^2+ab+ac+bc+a+b+c).\end{align*}
It follows that
\[
\sfr^4_1(q^2) = \sum_{a,b,c\in\mathbb{Z}}q^{2(a^2+b^2+c^2+ab+ac+bc+a+b+c)} 
			= \sum_{\substack{x,y,z\in\Z \\ x+y+z\;\mathrm{even}}}q^{x^2+x+y^2+y+z^2+z}
			= 4\mathsf{c}_2(q^2)^3.
\]

Therefore, $\sfr^4_1(q)=4\mathsf{c}_2(q)^3.$ As $\mathsf{c}_2(q)$ is the generating function of the triangular numbers, the statement that $r^4_1(n)>0$ for every $n\in\N$ is equivalent to the statement that every positive integer is the sum of three triangular numbers, which is Gauss' Eureka Theorem. We conclude that $r^4_1(n)>0$ for every $n\in\N$. 

\medskip

If $\ell>4$, the statement that $r^\ell_1(n)>0$ for all $n\in\N$ follows by induction on $\ell$ with base case $\ell=4$: for the induction step, set the final variable equal to $0$.

\medskip

As for $\ell<4$, the statement is clear for $\ell=1$ as $\sfr^1_1(q)=1$. When $\ell=2$, we have $\sfr^2_1(q)=\sum\limits_{t\in\Z}q^{t^2+t}=2\mathsf{c}_2(q^2)$. When $\ell=3$, we have $\sfr^3_1(q)=\sum\limits_{a,b\in\Z}q^{a^2+b^2+ab+a+b}=3\mathsf{c}_3(q)$, as was proved in \cite{Hirschhorn2008}. It follows that neither $\sfr^2_1(q)$ nor $\sfr^3_1(q)$ has strictly positive coefficients.
 
\item This was proved in \cite[Propositions 1.8 and 1.9]{CGJL} using the $290$-Theorem. We give a different proof for the cases $\ell\leq 5$ using formulas for the generating functions, and induction on $\ell$ for $\ell>5$. 

\medskip

Let $\ell=5$. Then $\sfr^5_0(q^2)=\sum\limits_{\substack{x,y,z,w,v\in\Z \\ x+y+z+w+v=0}}q^{x^2+y^2+z^2+w^2+v^2}$ by Lemma \ref{rksumsquares}. Fixing $v=-1$, the coefficient of $q^{2n}$ in $\sfr^5_0(q^2)$ is greater than or equal to the coefficient of $q^{2n}$ in 
\[
1+\sum_{x+y+z+w=1}q^{x^2+y^2+z^2+w^2+1}=1+q\sum_{x+y+z+w=1}q^{x^2+y^2+z^2+w^2}=1+q^2\sfr^4_1(q^2),
\]
by Lemma \ref{rksumsquares} again. For all $n\in\N$, the coefficient of $2n$ is positive in this generating function by part (1) above. Thus the same is true for the coefficient of $2n$ in $\sfr^5_0(q^2)$. We conclude that $r^5_0(n)>0$ for all $n\in\N$.

\medskip
If $\ell>5$, the statement that the coefficients of $\sfr^\ell_0(q)$ are strictly positive follows by induction on $\ell$ with base case $\ell=5$: for the induction step, set the final variable equal to $0$.
\medskip

We now deal with the cases $\ell<5$. In the cases $\ell=1$ and $\ell=2$, we have $\sfr^1_0(q)=1$ and $\sfr^2_0(q)=\sum_{t\in\mathbb{Z}}q^{t^2}$, for which the claim is obvious. When $\ell=3$, $\sfr^3_0(q)=\sum_{a,b\in\mathbb{Z}}q^{a^2+b^2+ab}$, and 
\[
\sum_{a,b\in\mathbb{Z}}q^{a^2+b^2+ab}=1+6\sum_{n\geq 1}\left(\frac{q^{3n-2}}{1-q^{3n-2}}-\frac{q^{3n-1}}{1-q^{3n-1}} \right)
\]
thanks to Ramanujan, see  \cite{Hirschhorn2008} for an accessible proof. For $n\geq 1$, $\frac{1}{6}r^3_0(n)$ is therefore the number of divisors of $n$ congruent to $1$ mod $3$ minus the number of divisors of $n$ congruent to $-1$ mod $3$ \cite{Hirschhorn2008}. It follows that $r^3_0(n)=0$ if $n\equiv -1\bmod 3$, as $d_1d_2\equiv -1\bmod 3$ implies that one of $d_1,d_2$ is $1$ mod $3$ and the other is $-1$ mod $3$, so that the divisors of $n$ come in pairs that cancel. 

When $\ell=4$, we have $\sfr^4_0(q)=\sum\limits_{a,b,c\in\mathbb{Z}}q^{a^2+b^2+c^2+ab+bc+ca}$. As in the proof for $\sfr^4_1(q)$, set $x=a+b$, $y=b+c$, $z=c+a$ in order to obtain 
\[ \sfr^4_0(q^2)=\sum_{\substack{x,y,z\in\Z \\ x+y+z\; \mathrm{even}}}q^{x^2+y^2+z^2}. \]
Clearly $x+y+z$ is even if and only if $x^2+y^2+z^2$ is even. By Lagrange's Three-Square Theorem, if $2n=4^a(8b+7)$ for some $a\geq 1$ and $b\geq 0$ then the coefficient of $q^{2n}$ in $ \sfr^4_0(q^2)$ is equal to $0$, and therefore $r^4_0(n)=0$ whenever such an equality holds.
\end{enumerate}
\end{proof}

In the proof of \Cref{thm_positivity_cuspidals}, it came to light that the generating functions $\sfr^\ell_0(q)$ and $\sfr^\ell_1(q)$ for levels $\ell=2,3,4$ have pleasant formulas. These give rise to well-known integer sequences of number-theoretic origin, and we take note of them here as Corollaries to the proof of Theorem \ref{thm_positivity_cuspidals}.  For a lattice $\Lambda\subset \Z^\ell$, and a vector $v=(a_1,a_2,...,a_\ell)\in\Lambda$, let the norm $\Vert v \Vert=a_1^2+a_2^2+\ldots+a_\ell^2$ be the square of its length.
The Theta series of $\Lambda$ is defined as $\Theta_\Lambda(z)=\sum\limits_{v\in\Lambda}q^{\Vert v\Vert}$, where $q=e^{\pi i z}$ and $z$ is a complex variable in the upper half plane \cite[§2.3]{ConwaySloane1999}.

\begin{Cor}($\ell=2$) From \Cref{lem_f_prod} it follows that 
\begin{align*}
\sfr^2_0(q) &= \sum\limits_{t\in\Z}q^{t^2} = 1+2q+2q^4+2q^9+\ldots,\\
 \sfr^2_1(q) &= 2\sum\limits_{t\geq 0}q^{t^2+t} = 2+2q^2+2q^6+2q^{12}+\ldots.
\end{align*}
Thus $r^2_0(n)$ is $1$ if $n=0$, $2$ if $n$ is a square greater than $0$, and $0$ otherwise; this is OEIS sequence \href{https://oeis.org/A000122}{A000122} \cite{oeis}. Likewise $r^2_1(n)$ is $2$ if $n$ is a pronic number and $0$ otherwise, yielding $2$ times OEIS sequence \href{https://oeis.org/A010054}{A010054} \cite{oeis}. We have
\begin{align*}
\sfr^2_0(q)& = \Theta_\Z(z) = \theta_3(z) = \phi(q),\\
\sfr^2_1(q)& = q^{-\frac{1}{4}}\theta_2(z)= 2\psi(q^2),
\end{align*}
where $\theta_3(z)$ and $\theta_2(z)$ are Jacobi $\theta$-functions, see \cite[pp.45,102]{ConwaySloane1999}.
\end{Cor}

\begin{Cor}\label{charge gf level 3}($\ell=3$)
\begin{itemize}
\item When $\ell=3$ and $\si=0$, we have seen that $r^3_0(n)$ is equal to the number of integer solutions to \linebreak $x^2+xy+y^2=n$. This is OEIS sequence \href{https://oeis.org/A004016}{A004016}, and $\sfr^3_0(q)=\Theta_{A_2}(z)$ where $A_2$ stands for the $A_2$ root lattice i.e. the planar hexagonal lattice \cite{oeis}. This lattice describes the maximal-density sphere-packing in dimension 2 (i.e. the optimal way to arrange coins of equal size in the plane). See \cite[Figure 1.3]{ConwaySloane1999}. The OEIS entry gives the following formula for $\sfr_0^3(q)$ in terms of $\phi$ and $\psi$ \cite{oeis}:
\[
\sfr^3_0(q)=\phi(q)\phi(q^3)+4q\psi(q^2)\psi(q^4).
\]
\item When $\ell=3$ and $\si=1$, $r^3_1(n)=3c_3(n)$, that is, three times the number of $3$-core partitions of $n$. This is three times OEIS sequence \href{https://oeis.org/A033687}{A033687}\cite{oeis}, and $\sfr^3_0(q)$ is the Theta series of the planar hexagonal lattice $A_2$ with respect to a deep hole. See \cite[Figure 1.3]{ConwaySloane1999}.
\end{itemize}
\end{Cor}

\begin{Cor}\label{charge gf level 4}($\ell=4$)
\begin{itemize}
\item We saw that $\sfr^4_0(q^2)=\sum\limits_{x+y+z\; \text{even}}q^{x^2+y^2+z^2}$. This is the Theta series of the face-centered cubic lattice, which describes a maximal-density sphere packing in dimension three (the usual way to stack oranges). See \cite[Figure 1.1]{ConwaySloane1999} for an illustration of this sphere-packing and the fcc lattice. We can thus interpret the coefficients of $\sfr^4_0(q)$ as follows: $r^4_0(n)$ is the number of spheres in the fcc lattice packing whose centers are located at a distance of $\sqrt{2n}$ from the origin.  See OEIS sequence \href{https://oeis.org/A004015}{A004015} \cite{oeis} and \cite[pp. 12-13]{ConwaySloane1999} for more on the Theta series of the fcc lattice. In particular, it can be expressed in terms of $\theta_2(z)$ and $\theta_3(z)$ as $\Theta_{\text{fcc}}(z)=\theta_3(4z)^3+3\theta_3(4z)\theta_2(4z)^2=\phi(q^4)^3+12q^2\phi(q^4)\psi(q^8)^2$ \cite[p.112,eq.(66)]{ConwaySloane1999}. It follows that
\[
\sfr^4_0(q)=\phi(q^2)^3+12q\phi(q^2)\psi(q^4)^2.
\]
\item On the other hand, we saw that $\sfr^4_1(q)=4\sfc_2(q)^3=4\psi(q)^3$, and $r^4_1(n)$ is four times the number of ways to write $n$ as a sum of three triangular numbers. This is four times \href{https://oeis.org/A008443}{A008443} \cite{oeis}.
\end{itemize}
\end{Cor}

\begin{Rem}\label{rem_cusp_theta}
In fact, by Lemma \ref{rksumsquares}, 
\[
\sfr^\ell_0(q^2)=\sum\limits_{\substack{v=(a_1,a_2,\ldots,a_\ell)\in\Z^\ell \\ a_1+a_2+\ldots+a_\ell=0}}x^{\Vert v\Vert}
			=\Theta_{A_{\ell-1}}(z),
\]
the Theta series of the $A_{\ell-1}$ root lattice. See  \cite[Equation 56, p. 110]{ConwaySloane1999} for an explicit formula using $\theta_3(z)$.
\end{Rem}

The motivation for studying $\sfr^\ell_1(q)$ and $\sfr^\ell_0(q)$ comes from the speculative representation theory of spetses for the groups $G(\ell,1,n)$ and $G(\ell,\ell,n)$. However, nothing prevents us from considering the generating function of charges $\bsig=(\si_1,\ldots,\si_\ell)\in\Z^\ell$ that sum to any integer we like. For $\si\in\Z$, define
\[
\sfr^\ell_\si(q)=\sum_{\substack{\bsig\in\mathbb{Z}^\ell \\ \si_1+\si_2+\ldots+\si_\ell=\si}} q^{\rk(\bsig)}.
\]
For the purpose of studying the $(e,\bs)$-cores in level $4$ 
 we now work out the rank generating function of the level $4$ charges that sum to $2$.

\begin{Lem}\label{sigma=2 gf}
We have $\sfr^4_2(q^2)=q\left( \frac{\phi(q)^3-\phi(-q)^3}{2}\right).$
\end{Lem}
\begin{proof}
Consider $\bsig=(a,b,c,d)\in\Z^4$ such that $a+b+c+d=2$. We set $d=2-a-b-c$ and compute 
\begin{align*}
\rk(\bsig)   &= \left\lceil \frac{\left(8(a^2+b^2+c^2)+8(ab+bc+ca)-16(a+b+c)+12\right)-3}{8}  \right\rceil\\
		&= a^2+b^2+c^2+ab+bc+ca-2(a+b+c)+2.
\end{align*}
Thus
\[
\sfr^4_2(q^2) = \sum_{a,b,c\in\Z} q^{2(a^2+b^2+c^2+ab+bc+ca-2(a+b+c))+4}.
\]
There is a bijection 
\begin{align*}
f:\;\Z^3\quad&\rightarrow\{(x,y,z)\in\Z^3\;\mid\; x+y+z\hbox{ is odd}\}\\
(a,b,c)&\mapsto(x,y,z)=(a+b-1,b+c-1,c+a-1)
\end{align*}
For $x=a+b-1$, $y=b+c-1$, and $z=c+a-1$, we find that 
\[
x^2+y^2+z^2=(a+b-1)^2+(b+c-1)^2+(c+a-1)^2=2\left(a^2+b^2+c^2+ab+bc+ca-2(a+b+c)\right)+3.
\]
It follows that 
\[
\sfr^4_2(q^2)=q\sum_{\substack{x,y,z\in\Z \\ x+y+z\text{ is odd}}}q^{x^2+y^2+z^2}.
\]

Now, the Theta series of the fcc lattice with respect to an octahedral hole is given by computing the square of the distance of each lattice point in the fcc lattice from a fixed octahedral hole. An example of an octahedral hole is $(0,0,1)$ \cite[Ch. 4, §6.3]{ConwaySloane1999}. As the fcc lattice is given as
\[
\Lambda_{\text{fcc}}=\{(x,y,z)\in\Z^3\;\mid\;x+y+z\text{ is even}\},
\]
it follows that the Theta series of the fcc lattice with respect to an octahedral hole is equal to 
\[
\sum_{\substack{x,y,z\in\Z \\ x+y+z \text{ is odd}}}q^{x^2+y^2+z^2}.
\]
The generating function of the Theta series of the fcc lattice with respect to an octahedral hole is known to be \cite[Ch. 4, §6.3, Eqn.(63)]{ConwaySloane1999}:
\[
\frac{\phi(q)^3-\phi(-q)^3}{2}
\]
and so we conclude that
\[ 
\sfr^4_2(q^2)=q\left(\frac{\phi(q)^3-\phi(-q)^3}{2}\right).
\]
\end{proof}

\begin{Prop}\label{sigma=2 gf cor}
We have
\begin{align*}
\sfr^4_2(q) & = 2q \psi(q^4)\phi(q)^2 + 4q\psi(q^2)^2\phi(q^2)\\
		& = 6q\phi(q^2)^2\psi(q^4) +  8q^2\psi(q^4)^3.
\end{align*}
\end{Prop}

\begin{proof}
We have $\phi(q)=\phi(q^4)+2q\psi(q^8)$ and $\phi(-q)=\phi(q^4)-2q\psi(q^8)$ by \cite[Appendix A]{KassReut2018}.
With \Cref{phipsi}, we can expand and simplify to obtain
$$\phi(q)^3 - \phi(-q)^3 = 4q\psi(q^8)\phi(q^2)^2 + 8q\psi(q^4)^2\phi(q^4).$$
Applying \Cref{sigma=2 gf} gives
$\sfr_2^4(q^2) = 2q^2\psi(q^8)\phi(q^2)^2 + 4q^2\psi(q^4)^2\phi(q^4)$,
and we can substitute $q$ for $q^2$ to obtain the desired formula. Applying \Cref{phipsi} yields the second version of the formula, which is the so-called ``bisection of the series" into the sum of the two series consisting of its odd- and even-power terms.
\end{proof}

The following observations will be useful for us later.
\begin{Lem}\label{tinylemma1}
For any $\si\in\Z$, we have $\sfr^\ell_\si(q)=\sfr^\ell_{-\si}(q)$.
\end{Lem}
\begin{proof} Let $\bsig=(\si_1,\ldots,\si_\ell)\in\Z^\ell$. 
Since $\sum\limits_{i,j}\si_i(\si_i-\si_j)=\sum\limits_{i,j}(-\si_i)(-\si_i-(-\si_j))$, it follows from Lemma \ref{rk_formula} that $\rk(\bsig)=\rk(-\bsig)$. Therefore
\[
\sfr^\ell_\si(q) = \sum_{\substack{\bsig\in\Z^\ell \\ \si_1+\ldots+\si_\ell=\si}}q^{\rk(\bsig)} = \sum_{\substack{\bsig\in\Z^\ell \\ -\si_1-\ldots-\si_\ell=-\si}}q^{\rk(-\bsig)} = \sum_{\substack{ \btau\in\Z^\ell \\
\tau_1+\ldots+\tau_\ell=-\si}}q^{\rk(\btau)} = \sfr^\ell_{-\si}(q).
\] 
\end{proof}

\begin{Lem}\label{tinylemma2}
For any $\si,k\in\Z$, we have $\sfr^\ell_\si(q)=\sfr^\ell_{\si+k\ell}(q)$.
\end{Lem}
\begin{proof}
Recall that for all $\bsig\in\Z^\ell$, it holds that $\rk(\bsig)=\rk(\bsig+(k,k,\ldots,k))$ by Lemma \ref{lem_invariance}. Thus
\[
\sfr^\ell_\si(q) = \sum_{\substack{\bsig\in\Z^\ell \\ \si_1+\ldots+\si_\ell=\si}}q^{\rk(\bsig)}
				= \sum_{\substack{\btau=\bsig+(k,k,\ldots,k) \\ \si_1+\ldots+\si_\ell=\si}}q^{\rk(\btau)}
				=  \sum_{\substack{\btau\in\Z^\ell \\ \tau_1+\ldots+\tau_\ell=\si+k\ell}}q^{\rk(\btau)}
				=\sfr^\ell_{\si+k\ell}(q).
\]
\end{proof}
Thus $\sfr^\ell_\si(q)$ is well-defined on $\si\in\Z/\ell\Z$, and moreover, we only need to consider the first $\left\lceil \frac{\ell +1}{2}\right\rceil$ residues of $\Z/\ell\Z$ for the value of $\si$ in order to see all distinct functions $\sfr^\ell_\si(q)$ that occur. Thus, if $n\in\Z$ then $\sfr^\ell_n(q)$ is equal to one of $\sfr^\ell_0(q)$, $\sfr^\ell_1(q), \ldots,  \sfr^\ell_{\left\lfloor \frac{\ell}{2} \right\rfloor}(q)$. 

\subsection{Rank generating functions of charged $\ell$-partitions and spetsial $d$-cores}
\label{sec_rk_gf}

Consider the generating function of all symbols in $\cU^\ell_\si$ with respect to their rank:
\[
\sfu^\ell_\si(q) = \sum_{ |\bla,\bsig\rangle \in \cU^\ell_\si } q^{\rk(|\bla,\bsig\rangle)}.
\]

\begin{Lem}\label{symbolgf_dec}
We have 
\[
\sfu^\ell_\si(q)=\sfr^\ell_\si(q)\sfp(q)^\ell.
\]
\end{Lem}

\begin{proof} Applying \cref{rk_formula} in the second equality below, we have
\[
\sfu^\ell_\si(q)=\sum_{|\bla,\bsig\rangle\in\cU^\ell_\si}q^{\rk(|\bla,\bsig\rangle)}=\sum_{(\bla,\bsig)\in\cP^\ell\times\Z^\ell[\si]}q^{|\bla|+\rk(\bsig)}=\left(  \sum_{\bsig\in\Z^\ell[\si]}q^{\rk(\bsig)}  \right)\left(  \sum_{\bla\in\cP^\ell}q^{|\bla|}  \right)=\sfr^\ell_\si(q)\sfp(q)^\ell.
\]
\end{proof}

\begin{Lem}\label{rk_gf_reduction} Let $\si,n\in\Z$ and suppose $n\equiv \si \bmod \ell$. We have
\[
\sfu^\ell_{-\si}(q)=\sfu^\ell_\si(q), \qquad\qquad
\hbox{and}
\qquad\qquad
\sfu^\ell_n(q)=\sfu^\ell_\si(q).
\]

\end{Lem}

\begin{proof}
The first equality follows from combining \Cref{tinylemma1,symbolgf_dec}, while the second equality follows from combining \Cref{tinylemma2,symbolgf_dec}.
\end{proof}

\medskip

We proceed to the study of the generating function of the spetsial $d$-cores. Let $d$ be a positive integer and fix $\si\in\Z$. We let 
\[ 
\cC^\ell_{[d,\si]}=\{|\bla,\bsig\rangle\in\cU^\ell_\si \;\mid\; |\bla,\bsig\rangle = |\bla,\bsig\rangle_{[d]}\}
\]
denote the set of spetsial $d$-cores in level $\ell$ whose charges sum to $\si$. Consider its rank generating function:
\[
\sfz^\ell_{d,\si}(q):=\sum_{|\bla,\bsig\rangle\in\cC^\ell_{[d,\si]}}q^{\rk(|\bla,\bsig\rangle)}.
\]

\begin{Prop}\label{gf_spetsialdcores}
Let $\ell\geq 2$ and $d\geq 1$. For any $\si\in\Z$, we have
\[
\sfz^\ell_{d,\si}(q)=\sfr^\ell_\si(q) \sfc_d(q)^\ell.
\]
\end{Prop}

\begin{proof}
On the one hand, we have $\sfu^\ell_\si(q) = \sfz^\ell_{d,\si}(q)\sfp(q^d)^{\ell d}$ by \Cref{rk_vs_lit}. On the other hand, we have $\sfu^\ell_\si(q)=\sfr^\ell_\si(q)\sfp(q)^\ell$ by \Cref{symbolgf_dec}. We thus obtain
\[
\sfz^\ell_{d,\si}(q)=\frac{\sfr^\ell_\si(q)\sfp(q)^\ell}{\sfp(q^d)^{d\ell}}=\sfr^\ell_\si(q)\left(\frac{\sfp(q)}{\sfp(q^d)^d}\right)^\ell=\sfr^\ell_\si(q)\sfc_d(q)^\ell.
\]
\end{proof}

\begin{Rem}\label{red_sig}
Combining \Cref{gf_spetsialdcores} with \Cref{tinylemma1,tinylemma2} we see that the study of $\sfz^\ell_{d,\si}(q)$ can be restricted to $\si\in\{0,1,2,\ldots,\left\lfloor\frac{\ell}{2}\right\rfloor\}$.
\end{Rem}

\begin{Cor}\label{d=1_cor}
In the case of spetsial $1$-cores, we have
\[
\sfz^\ell_{1,\si}(q)=\sfr^\ell_\si(q).
\]
\end{Cor}
\begin{proof}
This is immediate from \Cref{gf_spetsialdcores} as $\sfc_1(q)=1$.
\end{proof}

We have already established in \Cref{thm_prod_cocores} that $z^2_{d,0}(n)>0$ and $z^2_{d,1}(n)>0$ for all $n\in\N$ provided that $2d\geq 4$. We now establish an analogous statement for higher levels $\ell$, where a disparity between the $\si=0$ and $\si=1$ cases appears thanks to \Cref{thm_positivity_cuspidals} above.

\begin{Thm}\label{thm_positivity_spetsialdcores}
Let $\ell\geq 3$ and $d\geq 1$. Let $z^\ell_{d,\si}(n)$ be the coefficient of $q^n$ in the generating function $\sfz^\ell_{d,\si}(q)$ of the spetsial $d$-cores, $\si\in\{0,1\}$. Then:
\begin{enumerate}
\item $z^\ell_{d,1}(n)>0$ for all $n\in\N$ if and only if $\ell d\geq 4$,
\item $z^\ell_{d,0}(n)>0$ for all $n\in\N$ if and only if $\ell d\geq 5$.
\end{enumerate}
\end{Thm}
\begin{proof}
Let $\si\in\{0,1\}$. By \Cref{gf_spetsialdcores}, we have 
\[
\sfz^\ell_{d,\si}(q)=\sfr^\ell_\si(q) \sfc_d(q)^\ell = \sfr^\ell_\si(q)\left( 1+\sum_{n>1}a_nq^n\right) = \left( 1+\sum_{n>1}b_nq^n \right) \sfc_d(q)^\ell 
\]
for some $a_n,b_n\in\N$. It follows that $z^\ell_{d,\si}(n)>0$ whenever $d\geq 4$ by \Cref{grono}, that $z^\ell_{d,1}(n)>0$ whenever $\ell\geq 4$ by \Cref{thm_positivity_cuspidals}, and that $z^\ell_{d,0}(n)>0$ whenever $\ell\geq 5$  by  \Cref{thm_positivity_cuspidals}. This leaves only a few cases to check for small $\ell$ and $d$. 

\medskip

Let $d\in\{2,3\}$. We have:
\[
\sfz^\ell_{d,\si}(q)=\sfr^\ell_\si(q)\sfc_d(q)^\ell=\begin{cases} 
											\sfr^\ell_\si(q)\sfc_2(q)^\ell \qquad\qquad \hbox{ if }d=2 ,\\  
											\sfr^\ell_\si(q)\sfc_3(q)^\ell \qquad \qquad\; \hbox{if }d=3. 
												\end{cases}
\]
In the case $d=2$, we have $z^\ell_{2,\si}(n)>0$ for all $n\in\N$ by Gauss' Eureka Theorem. In the case $d=3$, we have $z^3_{3,1}(n)>0$ for all $n\in\N$ since any $n\in\N$ can be written as $n=a^2+b^2+c^2+c$ for some $a,b,c\in\N$ by the proof of \Cref{thm_positivity_cocores}(1)($d=3$), and thus any $n\in\N$ can be written as $n=x+y+z$ where $x,y,z$ are sizes of $3$-cores. 

\medskip

For the case $d=1$, we consider the two cases $\si=1$ and $\si=0$ for the relevant values of $\ell$ and apply \Cref{d=1_cor}:
\begin{enumerate}
\item (Case 1: $\si=1$). Let $\ell=3$. We have
\[
\sfz^3_{1,1}(q)=\sfr^3_1(q)
\]
and by \Cref{thm_positivity_cuspidals}, it follows that the coefficients of $\sfz^3_{1,1}(q)$ are not all positive.
\item (Case 2: $\si=0$). Let $\ell\in\{3,4\}$. We have
\[
\sfz^\ell_{1,0}(q)=\sfr^\ell_0(q)
\]
and by \Cref{thm_positivity_cuspidals}, it follows that the coefficients of $\sfz^\ell_{1,0}(n)=\sfr^\ell_0(q)$ are not all positive.
\end{enumerate}
\end{proof}

%%%%%%%%%%%%%%%%%

\subsection{Rank-preserving bijection between spetsial $1$-cores and cuspidal symbols}\label{section_d=1}
In the previous section, we observed in \Cref{d=1_cor} that  
$\sfr^\ell_\si(q)=\sfz^\ell_{1,\si}(q)$
for any $\si\in\Z$.
That is, the rank generating function of the set of charges $\Z^\ell[\si]$ for the charged bipartitions in $\cU^\ell_\si$ equals the rank generating function of the spetsial $1$-cores in $\cU^\ell_\si$.
We can upgrade this equality of generating functions to a rank-preserving bijection between spetsial $1$-cores and cuspidal symbols using the spetsial level-rank duality studied in \Cref{spetsial combinat}.

\medskip

By \Cref{sLR}, if $d=1$ then the spetsial level-rank duality is a permutation
\[
\sLR:\cU^\ell_\si\lra\cU^\ell_\si.
\]
Let $\cC^\ell_{[1,\si]}=\{|\bla,\bsig\rangle\in\cU^\ell_\si\;\mid\; |\bla,\bsig\rangle_{[1]}=|\bla,\bsig\rangle\}$ be the subset of spetsial $1$-cores. On the other hand, the subset of cuspidal symbols in $\cU^\ell_\si$ is $\{|\bemp,\btau\rangle\in\cU^\ell_\si\}=\{|\bemp,\btau\rangle\in\cU^\ell \;\mid\;|\btau|=\si\}$ and is naturally identified with $\Z^\ell[\si]$. The map $\sLR$ sends $|\bla,\bsig\rangle$ to $|\bmu,\btau\rangle\in\cU^\ell$ where $\bmu=|\bla,\bsig\rangle^{[d]}$ and $\sLR^{-1}(|\bemp,\btau\rangle)=|\bla,\bsig\rangle_{[d]}$. We have $|\bla,\bsig\rangle\in\cC^\ell_{[1,\si]}$ if and only if $\sLR(|\bla,\bsig\rangle)=|\bemp,\btau\rangle$ for some $\btau\in\Z^\ell[\si]$. Thus $\sLR$ restricts to a bijection
\[
\sLR:\cC^\ell_{[1,\si]}\lra\{|\bemp,\btau\rangle\in\cU^\ell \;\mid\;|\btau|=\si\}.
\]
Let us check that this map preserves the rank. When $d=1$, $\sLR$ cyclically rotates each column $j$ of $\bA=\bA(|\bla,\bsig\rangle)$ in the upwards direction by $j-1\bmod \ell$. This preserves the rank by  \Cref{lem_invariance} (the multiset of $\be$-numbers of $\bA$ does not change). We have proved the following statement.

\begin{Prop}\label{spetsial1core_vs_cuspsymb} When $d=1$, the spetsial level-rank duality $\sLR$ yields a rank-preserving bijection between the spetsial $1$-cores in $\cU^\ell_\si$ and the cuspidal symbols in $\cU^\ell_\si$ (identified with charges in $\Z^\ell[\si]$).
\end{Prop}
%In this way, we can parametrize the cuspidal symbols within a given $\cU^\ell_\si$ by the spetsial $1$-cores in $\cU^\ell_\si$.

\begin{Exa}\label{cusp_bij_example1}
Let $\ell=2$ and for $t\geq 1$ set $\Delta_t=(t,t-1,\ldots,2,1)$, and set $\Delta_0=\emptyset$. Let $\si=1$ (so the charged bipartitions label type $B$ and $C$ unipotent characters). The spetsial $1$-cores are all of the form $|(\Delta_t,\Delta_t),(1,0)\rangle$ or $|(\Delta_t,\Delta_t),(0,1)\rangle$, and the rank of these is $2\left(\frac{t(t-1)}{2}\right)+0=t^2+t$. Applying $\sLR$ to $|(\Delta_t,\Delta_t),(0,1)\rangle$ yields the cuspidal symbol $|\bemp,(-t,t+1)\rangle$ if $t$ is even, and $|\bemp,(t+1,-t)\rangle$ if $t$ is odd. Applying $\sLR$ to $|(\Delta_t,\Delta_t),(1,0)\rangle$ yields the cuspidal symbol $|\bemp,(-t,t+1)\rangle$ if $t$ is odd, and $|\bemp,(t+1,-t)\rangle$ if $t$ is even.  We have $\rk(|\bemp,(-t,t+1)\rangle)=\rk(|\bemp,(t+1,-t)\rangle)=t^2+t$. We remark that the labeling by Dynkin type in \Cref{sec_RT} is preserved under this bijection.
\end{Exa}

\begin{Exa}\label{cusp_bij_example2}
Let $\ell=4$, let $\bla=\left( (2^2), (1^3), (3,1^3), (2) \right)$ and $\bsig=(0,0,0,0)$. The abacus of $|\bla,\bsig\rangle$ is
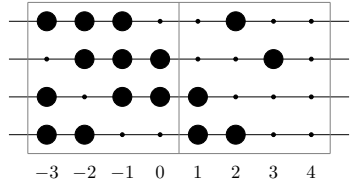
\begin{center}
\begin{tikzpicture}[scale=0.5, bb/.style={draw,circle,fill,minimum size=2.5mm,inner sep=0pt,outer sep=0pt}, wb/.style={draw,circle,fill,minimum size=0.5mm,inner sep=0pt,outer sep=0pt}]
%%%%%%%%%%%%%%%%%%%%%
%%%  On dessine les lignes de l'abaque et tous les emplacements
\foreach \j in {1,2,3,4}
{
\draw [line width=0.1mm] (-4,\j) -- (5,\j);
    \foreach \k in {-3,...,4}
    {
        \node [wb] at (\k,\j) {};
    }
}
%%%%%%%%%%%%%%%%%%%%%
%%%  On place les billes
\foreach \k/\j in {-3/1,-2/1,1/1,2/1,   -3/2,-1/2,0/2,1/2,   -2/3,-1/3,0/3, 3/3,   -3/4,-2/4,-1/4,2/4 }
{
    \node [bb] at (\k,\j) {};
}
%%%%%%%%%%%%%%%%%%%%%
%%%  On fait appraître la graduation
\foreach \k in {-3,...,4}
{
    \node [scale = 0.7] at (\k,0) {$\k$};
}
%%%%%%%%%%%%%%%%%%%%%
%%%  On dessine les rectangles
\draw [line width=0.1mm, color=gray] (-3.5,0.5) -- (4.5,0.5);
\draw [line width=0.1mm, color=gray] (-3.5,4.5) -- (4.5,4.5);
\foreach \k in {-3.5, 0.5, 4.5}
{
\draw [line width=0.1mm, color=gray] (\k,4.5) -- (\k,0.5);
}
\end{tikzpicture}
\end{center}
which is seen to be a spetsial $1$-core of rank $|\bla|=4+3+6+2=15$. 
Applying $\sLR$ yields the abacus of $\bemp$ with charge $(3,2,-4,-1)$, which has rank $3^2+2^2+(-4)^2+3\cdot2+2(-4)+3(-4)=15$, depicted below.
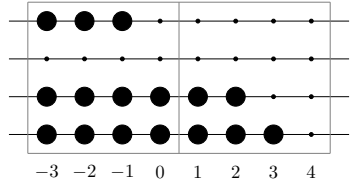
\begin{center}
\begin{tikzpicture}[scale=0.5, bb/.style={draw,circle,fill,minimum size=2.5mm,inner sep=0pt,outer sep=0pt}, wb/.style={draw,circle,fill,minimum size=0.5mm,inner sep=0pt,outer sep=0pt}]
%%%%%%%%%%%%%%%%%%%%%
%%%  On dessine les lignes de l'abaque et tous les emplacements
\foreach \j in {1,2,3,4}
{
\draw [line width=0.1mm] (-4,\j) -- (5,\j);
    \foreach \k in {-3,...,4}
    {
        \node [wb] at (\k,\j) {};
    }
}
%%%%%%%%%%%%%%%%%%%%%
%%%  On place les billes
\foreach \k/\j in {-3/1,-2/2,1/1,2/2,   -3/2,-1/4,0/1,1/2,   -2/4,-1/1,0/2, 3/1,   -3/4,-2/1,-1/2,2/1 }
{
    \node [bb] at (\k,\j) {};
}
%%%%%%%%%%%%%%%%%%%%%
%%%  On fait appraître la graduation
\foreach \k in {-3,...,4}
{
    \node [scale = 0.7] at (\k,0) {$\k$};
}
%%%%%%%%%%%%%%%%%%%%%
%%%  On dessine les rectangles
\draw [line width=0.1mm, color=gray] (-3.5,0.5) -- (4.5,0.5);
\draw [line width=0.1mm, color=gray] (-3.5,4.5) -- (4.5,4.5);
\foreach \k in {-3.5, 0.5, 4.5}
{
\draw [line width=0.1mm, color=gray] (\k,4.5) -- (\k,0.5);
}
\end{tikzpicture}
\end{center}
\end{Exa}

\begin{Rem} This is also possibly interesting for the $(\ell,\bs)$-cores in level $\ell$.
If $s\in\Z$ then the restriction of usual level-rank duality $\LR:\cU^\ell_s\lra \cU^\ell_s$ to $\cC_{\ell,s}^\ell$ yields a bijection $\LR:\cC_{\ell,s}^\ell\rightarrow\Z^\ell[s]$ and for $\si=s-\frac{\ell(\ell-1)}{2}$ we have the commutative diagram by \Cref{commlrdualities}

\[
\begin{tikzcd}
 \cC^\ell_{[1,\si]} \arrow[r, "\sLR"] \arrow[d, "\mathsf{t}_\brho"]
&  \Z^\ell[\si]  \arrow[d, "\mathsf{t}_{\brho_{\mathrm{rev}}}"] \\
  \arrow[r, "\LR"]
  \cC_{\ell,s}^\ell   &\Z^\ell[s] 
  \end{tikzcd}
\]
Let $\bs\in\Z^\ell$ and let $\bla\in\cC^\ell_{\ell,\bs}$. By \Cref{rk_vs_lit_bis} and \Cref{spetsial1core_vs_cuspsymb}, $\rk(|\bla,\bs-\brho\rangle)=\rk(\bt-\brho)$ where $\bt\in\Z^\ell=\LR(|\bla,\bs\rangle)$. Thus $|\bla|=\rk(\bt-\brho)-\rk(\bs-\brho)$ by \Cref{rk_formula}. We then have
\[
\sfc_{\ell,\bs}(q)=q^{-\rk(\bs-\brho)}\sum_{\bt\in\LR(\cC^\ell_{\ell,\bs})}q^{\rk(\bt-\brho)}
\]
where $\sfc_{\ell,\bs}(q)$ is the generating function of the $(\ell,\bs)$-cores with respect to their size. 
\end{Rem}

%%%%%%%%%%%%%%%%%

\subsection{Application to imprimitive spetses}

Can there be representation-theoretic data on the model of a finite group of Lie type -- group order formula, unipotent character degree formulas, Frobenius eigenvalues, Brauer trees, Fourier transform matrices... -- without anything behind it being represented? And if so, what is the meaning of the theory for the classical objects that do exist? These are questions raised by the Spetses program, which has amassed data to suggest the existence of exotic analogues of finite groups of Lie type where ``spetsial" complex reflection groups play the role of Weyl groups. See \cite{BMM1999} for the definitive paper on the Spetses program, \cite{Malle1998} for a short survey overview, and \cite{KMS2024,ChlouverakiMalle2026,Malle2026} for some recent developments.

\medskip

Weyl groups arise from the geometry of algebraic groups. However, Coxeter groups and more generally complex reflection groups may be defined algebraically, and extend the Weyl groups to a larger family of finite groups generated by {\em complex reflections}, invertible linear transformations of a finite-dimensional complex vector space that fix a hyperplane and whose unique eigenvalue different from $1$ is a root of unity. The irreducible complex reflection groups were classified by Shephard and Todd \cite{ShephardTodd1954}, and fall into an infinite series $G(\ell,m,n)$ (where $m|\ell$) and a finite list of exceptional groups. The complex reflection groups $G(\ell,1,n)$ for $\ell,n\in\Z_{\geq 1}$ include the symmetric groups $S_n$ as the case $\ell=1$ and the type $B$ Weyl groups as the case $\ell=2$. The group $G(\ell,1,n)$ consists of the $n\times n$ generalized permutation matrices whose non-zero entries are $\ell$'th roots of unity.

\medskip

The spetsial groups are a subset of the complex reflection groups satisfying the additional conditions in \cite[Proposition 3.10]{Malle1998}. Although the mythical spetsial analogues of finite reductive groups have not been found, there is a polynomial order formula for the group and a polynomial formula for the degrees of its unipotent characters, see \cite[Proposition 3.5]{BMM1999} and \cite[Theorem 4.10]{Malle1998}. For $G(\ell,1,n)$, the unipotent degrees are given by a hook formula \cite[Proposition 3.12]{Malle1995}. 

\medskip

The spetses for $G(\ell,m,n)$ (where $m$ divides $\ell$), are called {\em imprimitive spetses} and were studied by Malle in \cite{Malle1995}. Let us only discuss the case of $G(\ell,1,n)$ here. Recall that $\cU^\ell_1$ is the set of of all charged $\ell$-partitions $|\bla,\bsig\rangle$ such that $|\bsig|=1$. The language of symbols is usually used in the context of spetses -- recall from \Cref{sec_symbols_defs}  that charged $\ell$-partitions are another way of defining symbols. The cyclic group of order $\ell$ acts freely on $\cU^\ell_1$ by cyclic permutation of the rows of the abacus $\cA(|\bla,\bsig\rangle)$. The ``unipotent characters" of the $G(\ell,1,n)$-spets may be parametrized by the equivalence classes of $\cU^\ell_1$ under this action. The rank function is well-defined on an equivalence class by \Cref{lem_invariance}. As each equivalence class has cardinality $\ell$, it follows that the rank generating function of the unipotent characters for the $G(\ell,1,n)$-spets is $\frac{1}{\ell}\sfu^\ell_1({{q}})$, and the rank generating function of the unipotent characters labeled by spetsial $d$-cores is $\frac{1}{\ell}\sfz^\ell_{d,1}({{q}})$. 

\medskip

As blocks are algebraic in their definition, we canot really talk about blocks of spetses (however, see \cite{Craven2014,KMS2024}). The work-around to talk about blocks of the $G(\ell,1,n)$-spets in quantum characteristic $e$ without reference to the group algebra or a category of representations over a field $\mathbb{k}$ is by using cyclotomic polynomials $\Phi_e({{q}})$. Through their role in the polynomial formulas mentioned above, they give rise to the notion of $\Phi_e$-Harish-Chandra series, or $\Phi_e$-HC series for short. By analogy with the unitary prime case in level $2$, we will from now on assume that $e=d\ell$ for some $d\geq 1$. Under this assumption, an $\ell$-symbol is called $\Phi_e$-cuspidal if it is a spetsial $d$-core, see \cite[Section 3C]{Malle1995}. The analogue of the Nakayama Conjecture in this setting then states that for symbols $|\bla,\bsig\rangle,|\bmu,\btau\rangle\in\cU^\ell_1$, $|\bla,\bsig\rangle$ belongs to the same $\Phi_e$-HC series as $|\bmu,\btau\rangle$ if and only if $|\bla,\bsig\rangle_{[d]}=|\bmu,\btau\rangle_{[d]}$. The $\Phi_e$-cuspidals are thus the symbols that are alone in their $\Phi_e$-HC series. A $\Phi_e$-cuspidal is the analogue of a defect $0$ block in this setting.

\medskip

As a corollary of our results on the generating functions of the spetsial $d$-cores, we thus have the following results on the spetses of $G(\ell,1,n)$. 
% Let us use $x$ as the variable in the generating functions, as $q$ is held in reserve for a prime power in this setting.
\begin{Cor}\label{spetsial_results}
Fix a level $\ell$ and some $d\in\Z_{\geq 1}$, and set $e=d\ell$. 
\begin{enumerate}
\item The rank generating function of the $\Phi_e$-cuspidals for the spetses of $G(\ell,1,n)$, $n\in\N$, is $\frac{1}{\ell}\sfr^\ell_1({{q}}) \sfc_d({{q}})^\ell$.
\item The spets of $G(\ell,1,n)$ has a $\Phi_e$-cuspidal for all $n\in\N$ if and only if $e\geq 4$. 
\end{enumerate}
\end{Cor} 
\begin{proof}
\begin{enumerate}
\item The generating function of the $\Phi_e$-cuspidals is $\frac{1}{\ell}\sfz^\ell_{d,1}({{q}})$, and now apply \Cref{gf_spetsialdcores}.
\item This follows directly from the definition of $\Phi_e$-cuspidal and the positivity criterion of \Cref{thm_positivity_spetsialdcores}.
\end{enumerate}
\end{proof}

As interesting special cases of \Cref{spetsial_results}, when $\ell=3$, the rank generating function of the $\Phi_1$-cuspidals for the spets of $G(3,1,n)$, $n\in\N$, is $\sfc_3({{q}})$, the generating function of the $3$-core partitions. This follows from \Cref{charge gf level 3} and \Cref{d=1_cor}. When $\ell=4$, the rank generating function of the $\Phi_1$-cuspidals for the spets of $G(4,1,n)$, $n\in\N$, is $\psi({{q}})^3$. That is, the number of $\Phi_1$-cuspidals for the spets of $G(4,1,n)$ is the number of ways to write $n$ as the sum of three triangular numbers. This follows from \Cref{charge gf level 4} and \Cref{d=1_cor}.

%%%%%%%%%%%%%%%%%%%%%%%%%%%%%%%%%
%%%%%%%%%%%%%%%%%%%%%%%%%%%%%%%%%
%%%%%%%%%%escores
%%%%%%%%%%%%%%%%%%%%%%%%%%%%%%%%%
%%%%%%%%%%%%%%%%%%%%%%%%%%%%%%%%%

\section{Enumerating $(e,\bs)$-cores}\label{escore_gf_sec}

We now turn to the study of the $(e,\bs)$-cores. 
For $e,\ell\in\Z_{\geq 2}$ and $\bs\in\Z^\ell$, we let 
\[
\sfc_{e,\bs}(q)=\sum_{\bla\in\cC^\ell_{e,\bs}}q^{|\bla|} =\sum_{n\geq 0}c_{e,\bs}(n)q^n
\]
be the generating function of the $(e,\bs)$-cores with respect to their size. By definition, $c_{e,\bs}(n)$ is the number of $(e,\bs)$-cores of size $n$.
The goal of this section is to obtain precise formulas for some of the $\sfc_{e,\bs}(q)$ 
and study their positivity. 
We work in the setting that the level $\ell$ divides $e$ and set $d=e/\ell$. Then we use the equality between a generating function of spetsial $d$-cores and a sum of $q$-shifts of generating functions of $(e,\bs)$-cores established below in \Cref{gf_cores_decomp}. For small $\ell$ and $e$, our results on spetsial $d$-cores then yield explicit formulas for the generating functions of $(e,\bs)$-cores.

\subsection{First properties of the generating functions of $(e,\bs)$-cores }

First of all, we can reduce the study to specific charges.
For $\bs,\bs'\in\Z^\ell$, let us write $\bs\sim \bs'$ if $\bs'$
can be obtained from $\bs$ by compositions of the affine cyclic permutation $(s_1,\ldots, s_\ell)\mapsto (s_\ell-e, s_1,\ldots, s_{\ell-1})$
and the translation $(s_1,\ldots, s_\ell)\mapsto (s_1+1,\ldots, s_\ell+1)$.

\begin{Lem}\label{equivcharges}
Let $\bs,\bs'\in\sD_e^\ell$ such that $\bs\sim\bs'$.
Then $\sfc_{e,\bs}(q) = \sfc_{e,\bs'}(q).$
\end{Lem}

\begin{proof}
This is a direct consequence of the characterisation of $(e,\bs)$-cores given in \Cref{charac_e_cores_l}.
\end{proof}

\Cref{equivcharges} allows us to reduce the study of $\sfc_{e,\bs}(q)$
to the case of charges $\bs=(s_1,\ldots,s_\ell)$ such that $0=s_1\leq s_2\leq\ldots\leq s_\ell<\ell$. 
Moreover, we have the following easy special case.

\begin{Lem}\label{dilated} 
Let $\bs\sim(0,\ldots, 0)$. Then $\sfc_{e,\bs}(q) = \sfc_e(q^\ell)$.
\end{Lem}

\begin{proof}
Let $\bs=(k,\ldots,k)$.
By \Cref{charac_e_cores_l}, $\bla$ is an $(e,\bs)$-core
if and only if $\bla=(\la,\ldots,\la)$ where $\la$ is an $e$-core partition.
Hence
$$ \sfc_{e,\bs}(q)=\sum_{\bla\in\cC_{e,\bs}} q ^{|\bla|}=\sum_{\la\in\cC_{e}} q ^{\ell|\la|} = \sfc_e(q^\ell).$$
\end{proof}

Recall that $\brho=(0,1,\ldots,\ell-1)\in\Z^\ell$.
For all $\bsig\in\Z^\ell$, if we set $\bs=\bsig+d\brho$, then \Cref{compare_cores} tells us there
is a bijection between 
$\cC_{[d,\bsig]}^\ell$ and $\cC_{e,\bs}^\ell$ fixing $\ell$-partitions $\bla$ and shifting the charge by $d\brho$.
Set
\[
\widetilde{\rk} (\bs) =\left\lceil \frac{ \widetilde{Y} -\ell+1 }{2\ell} \right\rceil
\text{\quad where\quad }
\widetilde{Y} = \sum_{\substack{1\leq i,j\leq \ell \\ i\neq j }} (s_i- (i-1)d)(s_i-(i-1)d-(s_j-(j-1)d)),
\]
so that $\widetilde{\rk}(\bs)= \rk(\bs-d\brho)=\rk(\bsig)$. Fix $\si\in\Z$ and let $s=\si + d\frac{\ell(\ell-1)}{2}$. The following formula relates the generating function of the spetsial $d$-cores in $\cU^\ell_\si$ 
with the generating functions of $(e,\bs)$-cores in $\cU^\ell_s$.

\begin{Prop}\label{gf_cores_decomp}
For any $\si\in\Z$, $d\geq 1$ and $e=d\ell$, we have
\[
\ds\sfz_{d,\si}^\ell(q)
= \sum_{ \bs\in \sD_e^\ell[\si + d\frac{\ell(\ell-1)}{2}]} q^{\widetilde\rk(\bs)} \sfc_{e,\bs}(q).
\]
\end{Prop}

\begin{proof}
By the previous discussion, we have
\begin{equation*}
\sfz_{d,\si}^\ell(q) 
= \sum_{ \bsig\in \sDe_d^\ell[\si] } \sum_{\bla\in\cC_{[d,\bsig]}^\ell} q^{\rk(|\bla,\bsig\rangle)}
= \sum_{ \bs\in \sD_e^\ell[s] } q^{\widetilde\rk(\bs)} \sum_{\bla\in\cC_{e,\bs}^\ell} q^{|\bla|}.
\end{equation*}
\end{proof}

In fact, only a few values of $\si$ will be necessary in order to make all distinct generating functions of $(e,\bs)$-cores appear as ($q$-shifted) summands of some spetsial $d$-core generating function. By \Cref{red_sig}, every $\sfz^\ell_{d,\si}(q)$ coincides with one of $\sfz^\ell_{d,0}(q),\sfz^\ell_{d,1}(q),\ldots,\sfz^\ell_{d,\left\lfloor\frac{\ell}{2}\right\rfloor}(q)$. Thus, each distinct $(e,\bs)$-core generating function appears, multiplied by an appropriate power of $q$ and with some multiplicity, in one of these.

\medskip

Now, let $\mathcal{S}=\{\bs=(s_1,s_2,\ldots, s_\ell)\in\Z^\ell\;\mid\;0\leq s_1\leq s_2\leq  \ldots \leq s_\ell< \ell\}$ and take $e=\ell$. By \Cref{equivcharges}, there are finitely many distinct generating functions of $(\ell,\bs)$-cores for $\bs\in\Z^\ell$, and all of them are represented by at least one charge $\bs\in\mathcal{S}$. An interesting application of \Cref{gf_cores_decomp} is that the sum of $q$-shifted generating functions of $(\ell,\bs)$-cores for all $\bs\in\mathcal{S}$ is equal to a natural sum of generating functions of cuspidal $\ell$-symbols. 

\begin{Prop}\label{radprop}
\[
\sum_{\substack{\bs=(s_1,\ldots,s_\ell)\in\Z^\ell \\ 0\leq s_1\leq s_2\leq\ldots \leq s_\ell<\ell}}q^{\rk(\bs-\brho)}\sfc_{\ell,\bs}(q) =
\sum_{\si\in\Z/\ell\Z}\sfr^\ell_\si(q).
\]
\end{Prop}
\begin{proof}
Let $\mathcal{S}=\{\bs=(s_1,\ldots,s_\ell)\in\Z^\ell \;\mid \; 0\leq s_1\leq s_2\leq\ldots \leq s_\ell<\ell\}$, and for $\si\in\Z/\ell\Z$, let \[\mathcal{S}_\si=\{\bs\in\mathcal{S}\;\mid\; s_1+s_2+\ldots+s_\ell\equiv \si \bmod \ell\}.\]
On the other hand, for $n\in\Z$ recall the set
\[
\mathcal{D}^\ell_\ell[n]=\{\bt=(t_1,t_2,\ldots,t_\ell)\in\Z^\ell \; \mid \; t_1\leq t_2\leq \ldots \leq t_\ell\leq t_1+\ell \hbox{ and } t_1+t_2+\ldots+ t_\ell=n\}.
\]
Suppose $n\equiv\si \bmod \ell$. Then we have a bijection between $\mathcal{D}^\ell_\ell[n]$ and $\mathcal{S}_\si$ as follows. For $\bt\in\mathcal{D}^\ell_\ell[n]$, set $\mu_1=t_\ell-t_1$, $\mu_2=t_{\ell-1}-t_1$,..., $\mu_{\ell-1}=t_2-t_1$. Then $\sum_{i=1}^\ell t_i=\ell t_1+\sum_{i=1}^{\ell-1}\mu_i$, and thus $\sum_{i=1}^{\ell-1}\mu_i\equiv |\bt|\equiv \si\bmod \ell$. Since $t_\ell\geq t_{\ell-1}\geq \ldots \geq t_2\geq t_1$, we have $\mu_1\geq \mu_2\geq \ldots\geq \mu_{\ell-1}$, so $\mu=(\mu_1,\mu_2,\ldots,\mu_{\ell-1})$ is a partition with $\ell-1$ parts. Moreover, $\mu_1\leq \ell$ since $t_\ell\leq t_1+\ell$. Let $\nu=(\nu_1,\nu_2,\ldots,\nu_\ell)$ be the transpose partition of $\mu$ with $\ell$ parts, where we take the last parts to be $0$ if necessary. Set $s_i=\nu_{\ell+1-i}$. Then $\bs=(s_1,s_2,\ldots,s_\ell)\in\mathcal{S}_\si$ and we have constructed the desired bijection (the inverse map is obvious).\\

We then have
\begin{align*}
\sum_{\si\in\Z/\ell\Z}\sfr^\ell_\si(q)
	&= \sum_{\si\in\Z/\ell\Z}\sfz^\ell_{1,\si}(q)
	= \sum_{\si\in\Z/\ell\Z}\;\sum_{ \bs\in \sD_\ell^\ell\left[\si+\frac{\ell(\ell-1)}{2}\right]} q^{\rk(\bs-\brho)} \sfc_{\ell,\bs}(q)\\
	& = \sum_{\si\in\Z/\ell\Z} \; \sum_{\bs\in\mathcal{S}_\si}q^{\rk(\bs-\brho)} \sfc_{\ell,\bs}(q)
	= \sum_{\bs\in\mathcal{S}}q^{\rk(\bs-\brho)} \sfc_{\ell,\bs}(q),
\end{align*}
where we applied \Cref{d=1_cor} in the first equality and Proposition \ref{gf_cores_decomp} in the second equality, and used that $\si+\frac{\ell(\ell-1)}{2}$ runs over a complete set of residues for $\Z/\ell\Z$ in the third equality.
\end{proof}

The next lemma adds comprehensibility to the charges appearing in the tables in the proof of \Cref{thm_gf_escores}.
\begin{Lem}
Let $\si\in\Z$, let $d\geq 1$, let $\bsig,\btau\in \Z^\ell[\si]$ and let $\bs,\bt\in\Z^\ell[\si+d\frac{\ell(\ell-1)}{2}]$. Then $\bt\sim\bs$ if and only if $\btau$ is a cyclic permutation of $\bsig$.
\end{Lem}
\begin{proof}
Let $\zeta=e^{\frac{2\pi i}{\ell}}$ be a primitive $\ell$'th root of unity, and let $C_\ell=\langle \zeta\rangle$ be the subgroup of $\mathbb{C}^\times$ generated by $\zeta$, a cyclic group of order $\ell$. Write $\bsig=(\si_1,\si_2,\ldots,\si_\ell)$ and $\btau=(\tau_1,\tau_2,\ldots,\tau_\ell)$. Set $s=\si+d\frac{\ell(\ell-1)}{2}$ and set $e=\ell d$. Define an action of $C_\ell$ on $\Z^\ell[\si]$ by $\zeta\cdot \bsig=(\si_\ell,\si_1, \si_2,\ldots,\si_{\ell-1})$, and define an action of $C_\ell$ on $\Z^\ell[s]$ by $\zeta\cdot \bs=(s_\ell-(\ell-1)d,s_1+d,s_2+d,\ldots,s_{\ell-1}+d).$ Then $\zeta\cdot\bs\sim\bs$, and any $\bt\sim\bs$ such that $|\bs|=|\bt|$ satisfies $\zeta^j\cdot\bs$ for some $j\in\{0,\ldots,\ell-1\}$. Let $\bt_{d\brho}:\Z^\ell[\si]\lra\Z^\ell[s]$ be the $d\brho$-shift on charges given by $\bt_{d\brho}(\bsig)=\bsig+d\brho$. It is easy to check that $\zeta\cdot\bt_{d\brho}(\bsig)=\bt_{d\brho}(\zeta\cdot\bsig)$, and as $\bt_{d\brho}$ is invertible, this implies the statement.
\end{proof}

\subsection{Counting $(e,\bs)$-cores from spetsial cores}

For small values of the parameters $d$ and $\ell$, we are able to establish some remarkable formulas for
$\sfc_{e,\bs}(q)$ when $e=\ell d$ using the relationship of \Cref{gf_cores_decomp}.
We also derive positivity results in most cases.

\begin{Thm}\label{thm_gf_escores}
The formulas for $\sfc_{e,\bs}(q)$ in \Cref{tab_escores} hold, from which we derive the information recorded in the rightmost column about positivity of their coefficients $c_{e,\bs}(n)$.
\end{Thm}

\begin{table}[h]
\[
\begin{array}{llllll}
\hline
    \ell & d & e & \bs & \sfc_{e,\bs}(q)&\text{$c_{e,\bs}(n)>0$ for all $n\in\N$?}
\\
\hline
2&1&2&(0,0)& \psi(q^2)& \text{no (positive iff $n$ pronic)}
\\
&&&(0,1)& \phi(q)& \text{no (positive iff $n$ square)}
\\
&2&4&(0,0)&\psi(q^2)\psi(q^4)^2& \text{no (positive iff $n$ even)}
\\
&&&(0,1)&\psi(q^2)\psi(q)^2& \text{yes}
\\
&&&(0,2)& \psi(q^2)\left(\phi(q^2)^2 + 2q\psi(q^4)^2\right)& \text{yes}
\\
&3&6&(0,0)& \psi(q^2)\sfc_3(q^4)^2&\text{no (positive iff $n$ even)}
\\
&&&(0,1)& \text{see \Cref{proplevel3}} &\text{\em conjecturally yes}
\\
&&&(0,2)& \psi(q^2)\left(\sfc_3(q)^2 - q^2\sfc_3(q^4)^2\right) &\text{yes}
\\
&&&(0,3)& \text{see \Cref{proplevel3}} & \text{\em conjecturally yes}
\\
3&1&3&(0,0,0)& \sfc_3(q^3)&\text{no ($0$ for all $n\equiv 1,2\bmod 3$)}
\\
&&&(0,0,1)& \sfc_3(q)&\text{no (characterised by \cite{GO1996, Robbins})}
\\
&&&(0,1,2)& \sfr^3_0(q^3)+3q\sfc_3(q^3)
&\text{no ($0$ for all $n\equiv 2\bmod 3$)}
\\
4 &1&4&(0,0,0,0)&\psi(q^4)\psi(q^8)^2 & \text{no (positive iff $n$ divisible by $4$)}
 \\
   &&&(0,0,0,1)&\text{see \Cref{Eureka_dissection}} & \text{no ($0$ if $n=4,19,112$)}
 \\
 &&&(0,0,0,2)&\text{see \Cref{proplevel4}} & \text{no ($0$ for infinitely many odd $n$)}
\\
    &&&(0,0,1,1)&\psi(q^4)\psi(q^2)^2 & \text{no (positive iff $n$ even)}
 \\
   &&&(0,0,1,2)&\text{see \Cref{Eureka_dissection}} & \text{\em conjecturally yes}
 \\
     &&&(0,0,2,2)& \psi(q^4)\left( \phi(q^4)^2+2q^2\psi(q^8)^2 \right) & \text{no (positive iff $n$ even)}
 \\
    &&&(0,1,1,2)& \psi(q^4)\left( \phi(q^2)^2+2q\psi(q^4)^2\right)

   & \text{no (positive iff $n\not\equiv 3\mod 4$)}
 \\
 &&&(0,1,2,3)&\text{see \Cref{proplevel4}} & \text{no ($0$ for infinitely many even $n$)}
 \\
\hline
\end{array}
\]
\caption{Formulas for the first $(e,\bs)$-core generating functions}
\label{tab_escores}
\end{table}

\begin{Rem}
This is a complete list of the distinct generating functions of $(e,\bs)$-cores for all possible $\bs\in\Z^\ell$ for each of the given values of $e$ when $\ell=2,3,4$.
\end{Rem}

\subsection{Proof of \Cref{thm_gf_escores}}

\subsubsection{Case $\ell=2$.}
Let $\ell=2$ and let $\si\in\{0,1\}.$ 
By Lemma \ref{tinylemma2} and \Cref{gf_spetsialdcores}, it holds that $\sfz_{d,\si}^\ell(q) = \sfz_{d,n}^\ell(q)$ for all $n \equiv \si\bmod \ell$.  

\medskip

\emph{Subcase $d=1$.}

We have $e=d\ell=2$.
Assume first $\si=0$, so that $s=\si+d\frac{\ell(\ell-1)}{2}=1$.
Let us compute the elements $\bsig\in\sDe_1^2[0], \bs\in\sD_2^2[1]$ and their corresponding rank.
$$
\begin{array}{lll}
\hline
\bs\in\sD_2^2[0] & \bsig\in\sDe_1^2[0] & \rk(\bsig)  
\\
\hline
(0,1) & (0,0) & 0
\\
\hline
\end{array}
$$
By \Cref{gf_cores_decomp}, we have
$$\sfz_{1,0}^2(q)= \sfc_{2,(0,1)}(q)$$
and we obtain
$$\sfc_{2,(0,1)}(q) = \sfz_{1,0}^2(q) =\phi(q)$$
where the last identity holds by \Cref{thm_prod_cocores}.

\medskip

Assume now that $\si=1$, so that $s=\si+d\frac{\ell(\ell-1)}{2}=2$.
We compute again $\sDe_1^2[1], \sD_2^2[2]$ and their corresponding rank.
$$
\begin{array}{lll}
\hline
\bs\in\sD_2^2[2] & \bsig\in\sDe_1^2[1] & \rk(\bsig)  
\\
\hline
(0,2) & (0,1) & 0
\\
(1,1) & (1,0) & 0
\\
\hline
\end{array}
$$
We have $(0,2)\sim(1,1)\sim(0,0)$ and by \Cref{dilated} we get
$\sfc_{2,(0,0)}(q)=\psi(q^2).$

\medskip

\emph{Subcase $d=2$.}

We have $e=d\ell=4$.
Assume first $\si=0$, so that $s=\si+d\frac{\ell(\ell-1)}{2}=2$.
Let us compute the elements $\bsig\in\sDe_4^2[0], \bs\in\sD_2^2[2]$ and their corresponding rank.
$$
\begin{array}{lll}
\hline
\bs\in\sD_4^2[2] & \bsig\in\sDe_2^2[0] & \rk(\bsig)  
\\
\hline
(0,2) & (0,0) & 0
\\
(-1,3) & (-1,1) & 1
\\
(1,1) & (1,-1) & 1
\\
\hline
\end{array}
$$
By \Cref{gf_cores_decomp}, we have
$$\sfz_{2,0}^2(q)= q\sfc_{4,(-1,3)}(q)+ \sfc_{4,(0,2)}(q)+ q\sfc_{4,(1,1)}(q) =  \sfc_{4,(0,2)}(q) + 2q \sfc_{4,(0,0)}(q).$$
Since $ \sfc_{4,(0,0)}(q) =  \sfc_4(q^2)$ by \Cref{dilated},
we obtain
$$\sfc_{4,(0,2)}(q) = \sfz_{2,0}^2(q) -  2q\sfc_4(q) = \phi(q)\psi(q)^2 - 2q\sfc_4(q^2)$$
where the last identity holds by \Cref{thm_prod_cocores}.

\begin{Lem}\label{lem_phipsitrick}
We have $\phi(q)\psi(q)^2 =\phi(q^2)^2\psi(q^2) + 4q\sfc_4(q^2)$.
\end{Lem}

\begin{proof}
We have
\begin{align*}
\phi(q)\psi(q)^2 
& =  \phi(q)^2\psi(q^2) \text{\;\quad\qquad\qquad\qquad by \Cref{Sylvie4alt}}
\\
& = (\phi(q^2)^2+4q\psi(q^4)^2)\psi(q^2) \text{\quad by \Cref{phipsi}}
\\
& = \phi(q^2)^2 \psi(q^2) + 4q\psi(q^4)^2 \psi(q^2)
\\
& = \phi(q^2)^2 \psi(q^2) + 4q\sfc_4(q^2)
\end{align*}
since $\sfc_4(q) = \sfc_2(q)\sfc_2(q^2)^2 = \psi(q)\psi(q^2)^2$ by \Cref{factorization gf e-cores}.
\end{proof}

From \Cref{lem_phipsitrick} we obtain as a direct consequence the following formula:
\begin{equation}\label{l=2d=2s=(02)}
\sfc_{4,(0,2)}(q) = \psi(q^2)\left(\phi(q^2)^2+2q\psi(q^4)^2\right) = \phi(q^2)^2\psi(q^2) + 2q\sfc_4(q^2).
\end{equation}
This is the bisection of the series and we have gotten rid of the minus signs. 
Moreover, Formula \Cref{l=2d=2s=(02)} allows us to deduce positivity of the coefficients of $\sfc_{4,(0,2)}(q)$. Let $n\in\N$. The coefficient of $q^{2n+1}$ is given by $c_4(n)$, and $c_4(n)>0$ by \cite[Theorem 3]{Ono1994}. The coefficient of $q^{2n}$ is given by the coefficient of $q^n$ in $\phi(q)^2\psi(q)$. We thus need to show that any $n\in\N$ may be expressed as $n=x^2+y^2+\frac{z^2+z}{2}$ for some $x,y,z\in\N$. As the squares mod $8$ are $0$, $1$ and $4$, Legendre's Three-Square Theorem gives us $8n+1=u^2+v^2+w^2$ for some even non-negative integers $u$ and $v$ and some odd positive integer $w$. Write $u=2u'$, $v=2v'$ and $w=2z+1$ for some $u',v',z\in\N$. Noting that $u'$ and $v'$ must have the same parity (by congruence considerations mod $8$ again), it follows that $u'+v'$ and $u'-v'$ are both even. We then set $x=\frac{u'+v'}{2},\;y=\frac{u'-v'}{2} \in\N$. We have:
\begin{align*}
8n+1 & = u^2+v^2+w^2\\
8n+1 &= 4(u'^2+v'^2)+4z^2+4z+1\\
4n & = 2(u'^2+v'^2)+2(z^2+z)\\
4n & = (u'+v')^2+(u'-v')^2+2(z^2+z)\\
4n & = (2x)^2+(2y)^2+2(z^2+z) \\
n & = x^2 + y^2 + \frac{z^2+z}{2}.
\end{align*}

\medskip

Assume now that $\si=1$, so that $s=\si+d\frac{\ell(\ell-1)}{2}=3$.
We have
$$
\begin{array}{lll}
\hline
\bs\in\sD_4^2[3] & \bsig\in\sDe_2^2[1] & \rk(\bsig)  
\\
\hline
(0,3) & (0,1) & 0
\\
(1,2) & (1,0) & 0
\\
\hline
\end{array}
$$
By \Cref{gf_cores_decomp}, we have
$$\sfz_{2,1}^2(q)= \sfc_{4,(0,3)}(q) + \sfc_{4,(1,2)}(q) = 2\sfc_{4,(0,1)}(q) $$
and we obtain
\begin{equation}\label{Sylvie4}
\sfc_{4,(0,1)}(q) = \frac{1}{2}\sfz_{2,1}^2(q) = \frac{1}{2}\left(2\psi(q^2)\sfc_2(q)^2\right)=\psi(q^2)\psi(q)^2.
\end{equation}
Its coefficients are positive by \Cref{thm_positivity_cocores}.

\medskip

\emph{Subcase $d=3$.}

We have $e=d\ell=6$.
Assume first $\si=0$, so that $s=\si+d\frac{\ell(\ell-1)}{2}=3$.
$$
\begin{array}{lll}
\hline
\bs\in\sD_6^2[3] & \bsig\in\sDe_3^2[0] & \rk(\bsig)  
\\
\hline
(0,3) & (0,0) & 0
\\
(-1,4) & (-1,1) & 1
\\
(1,2) & (1,-1) & 1
\\
\hline
\end{array}
$$
By \Cref{gf_cores_decomp}, we have
$$\sfz_{3,0}^2(q)= q\sfc_{6,(-1,4)}(q)+ \sfc_{6,(0,3)}(q)+ q\sfc_{6,(1,2)}(q) =  \sfc_{6,(0,3)}(q) + 2q \sfc_{6,(0,1)}(q),$$
hence by \Cref{thm_prod_cocores} we have proved the following relationship:
\begin{equation}\label{proplevel3} 
\phi(q)c_3(q)^2 =\sfc_{6,(0,3)}(q) + 2q \sfc_{6,(0,1)}(q).
\end{equation}
Positivity of the terms $\sfc_{6,(0,3)}(q)$ and $\sfc_{6,(0,1)}(q)$ separately is equivalent to the truth of \Cref{conj_typesd} in the case $e=6$.

\medskip

Assume now $\si=1$, so that $s=\si+d\frac{\ell(\ell-1)}{2}=4$.
$$
\begin{array}{lll}
\hline
\bs\in\sD_6^2[4] & \bsig\in\sDe_3^2[1] & \rk(\bsig)  
\\
\hline
(0,4) & (0,1) & 0
\\
(1,3) & (1,0) & 0
\\
(2,2) & (2,-1) & 2
\\
(-1,5) & (-1,2) & 2
\\
\hline
\end{array}
$$
By \Cref{gf_cores_decomp}, we have
$$\sfz_{3,1}^2(q) 
= q^2\sfc_{6,(-1,5)}(q)+ \sfc_{6,(0,4)}(q)+ \sfc_{6,(1,3)}(q)+ q^2\sfc_{6,(2,2)}(q) 
= 2\sfc_{6,(0,2)}(q) + 2q^2 \sfc_{6}(q^2).$$
As $\sfz_{3,1}^2(q)=2\psi(q^2)\sfc_3(q)^2$ by \Cref{thm_prod_cocores}, we obtain 
\begin{align*}
 \sfc_{6,(0,2)}(q) &= \psi(q^2)\sfc_3(q)^2 - q^2\sfc_{6}(q^2) \\
                    &= \psi(q^2)\left(\sfc_3(q)^2 - q^2\sfc_3(q^4)^2\right)\\
                   &=  \psi(q^2)\left(\sfc_3(q)- q\sfc_3(q^4)\right)\left(\sfc_3(q)+ q\sfc_3(q^4)\right),
 \end{align*}
where we applied \Cref{factorization gf e-cores} in the second equality. 

\medskip

Let us now show that $ c_{6,(0,2)}(n)>0 $ for all $n\in\N$. By \cite[Theorem 4.1]{BaruahBerndt},
\[ c_3(4n+1)=c_3(n) \]
for all $n\in\N$. Therefore,
\begin{itemize}
\item the coefficient of $q^n$ in $\sfc_3(q)- q\sfc_3(q^4)$ equals $0$ if $n\equiv 1\bmod 4$, and $c_3(n)$ else,
\item the coefficient of $q^n$ in $\sfc_3(q)+ q\sfc_3(q^4)$ equals $2c_3(n)$ if $n\equiv 1\bmod 4$, and $c_3(n)$ else.
\end{itemize}
It follows that $ c_{6,(0,2)}(n)>0$ if and only $n=x^2+x+a+b$ for some $x,a,b\in\N$ such that $c_3(a),c_3(b)>0$ and $a$ is even. By \Cref{lem_size_3cores}, to show the latter it suffices to show that $n=x^2+x+y^2+y+z^2$ for some $x,y,z\in\N$. This is equivalent to positivity of the coefficient of $q^n$ in $\psi(q^2)^2\phi(q)=\psi(q^2)\psi(q)^2=\sfz^2_{2,1}(q)$, which holds for all $n\in\N$ by \Cref{thm_positivity_cocores}. We conclude that the coefficients of $\sfc_{6,(0,2)}(q)$ are all positive.

\subsubsection{Case $\ell=3$.}

\emph{Subcase $d=1$.}

We have $e=d\ell=3$.
Assume first $\si=0$, so that $s=\si+d\frac{\ell(\ell-1)}{2}=3$.
$$
\begin{array}{lll}
\hline
\bs\in\sD_3^3[3] & \bsig\in\sDe_1^3[0] & \rk(\bsig)  
\\
\hline
(0,1,2) & (0,0,0) & 0
\\
(-1,2,2) & (-1,1,0) & 1
\\
(0,0,3) & (0,-1,1) & 1
\\
(1,1,1) & (1,0,-1) & 1
\\
\hline
\end{array}
$$
Similarly to the previous computations, we have
$$\sfr^3_0(q)=\sfz_{1,0}^3(q)= \sfc_{3,(0,1,2)}(q) + 3q\sfc_{3,(0,0,0)}(q)=\sfc_{3,(0,1,2)}(q) + 3q\sfc_{3}(q^3)$$
and we obtain $\sfc_{3,(0,1,2)}(q) = \sfr^3_0(q) - 3q\sfc_3(q^3).$
We obtain an explicit formula by remembering Ramanujan's formula 
$$\sfr^3_0(q) = 1 + 6 \sum_{n\geq 1} \left( \frac{q^{3n-2}}{1-q^{3n-2}} - \frac{q^{3n-1}}{1-q^{3n-1}} \right).$$

We can also apply \cite[Lemmas 2 and 3]{Hirschhorn2008} to write 
\[ \sfr^3_0(q)=\sfr^3_0(q^3)+6q\sfc_3(q^3).\] 
We then obtain 
\[
\sfc_{3,(0,1,2)}(q)=\sfr^3_0(q^3)+3q\sfc_3(q^3)
\]
from which it is transparent that the coefficient of $q^{3n-1}$ in $\sfc_{3,(0,1,2)}(q)$ equals $0$ for all $n\in\N$. The identity $r^3_0(3n)=r^3_0(n)$ then tells us that $r^3_0(3n)=c_{3,(0,1,2)}(3n)$ for all $n\in\N$. As for the coefficients of $q^{3n+1}$ for $n\in\N$, \cite[Lemmas 2 and 3]{Hirschhorn2008} imply that $r^3_0(3n+1)=6c_3(n)$, and thus $r^3_0(3n+1)=2c_{3,(0,1,2)}(3n+1)$. Thus we completely understand the relationship between the coefficients of $\sfr^3_{0}(q)$ and $\sfc_{3,(0,1,2)}(q)$:
\begin{itemize}
\item $c_{3,(0,1,2)}(n)=r^3_0(n)$ if $n\equiv 0 \bmod 3$,
\item $c_{3,(0,1,2)}(n)=\frac{1}{2}r^3_0(n)$ if $n\equiv 1\bmod 3$,
\item $c_{3,(0,1,2)}(n)=0=r^3_0(n)$ if $n\equiv 2\bmod 3$.\\
\end{itemize}
The sequence of coefficients $c_{3,(0,1,2)}(n)$ is OEIS sequence \href{https://oeis.org/A113062}{A113062}, while the sequence of coefficients $r^3_0(n)$ is OEIS sequence \href{https://oeis.org/A004016}{A004016} \cite{oeis}.\\

Assume now $\si=1$, so that $s=\si+d\frac{\ell(\ell-1)}{2}=4$.
\[
\begin{array}{lll}
\hline
\bs\in\sD_3^3[4] & \bsig\in\sDe_1^3[1] & \rk(\bsig)  
\\
\hline
(0,1,3) & (0,0,1) & 0
\\
(0,2,2) & (0,1,0) & 0
\\
(1,1,2) & (1,0,0) & 0
\\
\hline
\end{array}
\]
We have this time
\[
\sfz_{1,1}^3(q)= 3\sfc_{3,(0,0,1)}(q)
\]
and as $\sfz^3_{1,1}(q)=\sfr^3_1(q)=3\sfc_3(q)$, we obtain
\[
\sfc_{3,(0,0,1)}(q) = \sfc_3(q).
\]

Assume finally $\si=2$, so that $s=\si+d\frac{\ell(\ell-1)}{2}=5$.
\[
\begin{array}{lll}
\hline
\bs\in\sD_3^3[5] & \bsig\in\sDe_1^3[2] & \rk(\bsig)  
\\
\hline
(0,2,3) & (0,1,1) & 0
\\
(1,1,3) & (1,0,1) & 0
\\
(1,2,2) & (1,1,0) & 0
\\
\hline
\end{array}
\]
We obtain
\[
\sfz_{1,2}^3(q)= 3\sfc_{3,(0,1,1)}(q).
\]
We have $\sfz_{1,2}^3(q)=\sfr^3_2(q)=\sfr^3_{-1}(q)=\sfr^3_1(q)$ by \Cref{tinylemma1,tinylemma2}. Thus 
\[
\sfc_{3,(0,1,1)}(q)=\sfc_{3,(0,0,1)}(q)=\sfc_3(q).
\]

\subsubsection{Case $\ell=4$.}

\emph{Subcase $d=1$.}

We have $e=d\ell=4$.
Assume first $\si=0$, so that $s=\si+d\frac{\ell(\ell-1)}{2}=6$.
$$
\begin{array}{lll}
\hline
\bs\in\sD_4^4[6] & \bsig\in\sDe_1^4[0] & \rk(\bsig)  
\\
\hline
(0,1,2,3) & (0,0,0,0) & 0
\\
(-1,1,3,3) & ( -1,0,1,0) & 1
\\
(0,0,2,4) & ( 0,-1,0,1) & 1
\\
(1,1,1,3)&(1,0,-1,0)&1
\\
(0,2,2,2) & ( 0,1,0,-1) & 1
\\
(-1,2,2,3) & ( -1,1,0,0) & 1
\\
(0,0,3,3) & (0,-1,1,0) & 1
\\
(0,1,1,4) & (0,0,-1,1) & 1
\\
(1,1,2,2)&(1,0,0,-1)&1
\\
\hline
\end{array}
$$
We thus have
\begin{align*}
\sfz_{1,0}^4(q)  &= \sfc_{4,(0,1,2,3)}(q) + 4q \sfc_{4,(0,0,0,2)}(q) + 4q \sfc_{4,(0,0,1,1)}(q)\\
			&= \sfc_{4,(0,1,2,3)}(q) + 4q \sfc_{4,(0,0,0,2)}(q) +4q\sfc_{4,(0,1)}(q^2).
\end{align*}
Applying \Cref{charge gf level 4}, \Cref{d=1_cor} and \Cref{Sylvie4}, 
we obtain
\begin{align*}
\phi(q^2)^3+12q\phi(q^2)\psi(q^4)^2 &= \sfr^4_0(q)	\\
 & = \sfz_{1,0}^4(q)  \\
 & = \sfc_{4,(0,1,2,3)}(q) + 4q \sfc_{4,(0,0,0,2)}(q) + 4q\sfc_{4,(0,1)}(q^2)	\\
 &= \sfc_{4,(0,1,2,3)}(q) + 4q \sfc_{4,(0,0,0,2)}(q) + 4q\psi(q^4)\psi(q^2)^2.
\end{align*}

Applying \Cref{Sylvie4alt} with $q^2$ in place of $q$ yields
\[
\sfc_{4,(0,1)}(q^2)=\psi(q^4)\psi(q^2)^2=\phi(q^2)\psi(q^4)^2,
\]
allowing us to deduce that
\begin{equation}\label{proplevel4} 
\sfc_{4,(0,1,2,3)}(q) + 4q \sfc_{4,(0,0,0,2)}(q) = \phi(q^2)^3+8q\phi(q^2)\psi(q^4)^2 .
\end{equation}
The right-hand-side is the bisection of the series. From this, we can deduce that both $\sfc_{4,(0,1,2,3)}(q)$ and $\sfc_{4,(0,0,0,2)}(q)$ have infinitely many coefficients that are $0$. In particular,  $\sfc_{4,(0,1,2,3)}(q)$ has infinitely many even-power coefficients that are $0$ and $\sfc_{4,(0,0,0,2)}(q)$ has infinitely many odd-power coefficients that are $0$, since infinitely many positive integers cannot be expressed as a sum of three squares. On the other hand, $\phi(q)\psi(q^2)^2$ has positive coefficients, so it is possible that $\sfc_{4,(0,1,2,3)}(q)$ has positive odd-power coefficients and and $\sfc_{4,(0,0,0,2)}(q)$ has positive even-power coefficients. 

\medskip

Next, assume $\si=1$, so that $s=1+\frac{4(4-1)}{2}=7$.
\[
\begin{array}{lll}
\hline
\bs\in\sD_4^4[7] & \bsig\in\sDe_1^4[1] & \rk(\bsig)  
\\
\hline
(0,1,2,4) & ( 0,0,0,1) & 0
\\
(0,1,3,3) & ( 0,0,1,0) & 0
\\
(0,2,2,3) & ( 0,1,0,0) & 0
\\
(1,1,2,3) & (1,0,0,0) & 0
\\
(0,0,3,4) & (0,-1,1,1) & 1
\\
(-1,2,3,3) & (-1,1,1,0) & 1
\\
(1,2,2,2) & ( 1,1,0,-1) & 1
\\
(1,1,1,4)&(1,0,-1,1)&1
\\
\hline
\end{array}
\]
We apply \Cref{charge gf level 4} and find
\[
4\psi(q)^3=\sfr^4_1(q)=\sfz^4_{1,1}(q)=4\sfc_{4,(0,0,1,2)}(q)+4\sfc_{4,(0,1,1,1)}(q),
\]
and thus the following relationship holds:
\begin{equation}\label{Eureka_dissection}
\sfc_{4,(0,0,1,2)}(q)+\sfc_{4,(0,1,1,1)}(q)=\psi(q)^3.
\end{equation}

Similarly, if we take $\si=-1$, using that $\sfz^4_{-1}(q)=\sfz^4_1(q)$ we find that 
\[
\sfc_{4,(0,1,2,2)}(q)+\sfc_{4,(0,0,0,1)}(q)=\psi(q)^3.
\]
We have $\sfc_{4,(0,0,0,1)}(q)=\sfc_{4,(0,1,1,1)}(q)$ and $\sfc_{4,(0,1,2,2)}(q)=\sfc_{4,(0,0,1,2)}(q)$.\

\medskip

Finally, assume $\si=2$, so that $s=2+\frac{4(4-1)}{2}=8$.
\[
\begin{array}{lll}
\hline
\bs\in\sD_4^4[8] & \bsig\in\sDe_1^4[2] & \rk(\bsig)  
\\
\hline
(0,1,3,4) & (0,0,1,1) & 1
\\
(0,2,3,3) & (0,1,1,0) & 1
\\
(1,2,2,3) & ( 1,1,0,0) & 1
\\
(1,1,2,4)&(1,0,0,1)&1
\\
(0,2,2,4)&(0,1,0,1)&1
\\
(1,1,3,3)&(1,0,1,0)&1
\\
(0,0,4,4) & (0,-1,2,1) & 3
\\
(-1,3,3,3) & (-1,2,1,0) & 3
\\
(2,2,2,2) & ( 2,1,0,-1) & 3
\\
(1,1,1,5) & (1,0,-1,2) & 3
\\
\hline
\end{array}
\]

We obtain
\begin{align*}
\sfr^4_2(q)&= 4q \sfc_{4,(0,1,1,2)}(q) + 2q \sfc_{4,(0,0,2,2)}(q) +4q^3\sfc_{4,(0,0,0,0)}(q)\\
		&= 4q \sfc_{4,(0,1,1,2)}(q) + 2q \sfc_{4,(0,2)}(q^2)+4q^3\sfc_4(q^4)\\
		&= 4q \sfc_{4,(0,1,1,2)}(q) + 2q\left( \psi(q^4)\phi(q^4)^2+2q^2\psi(q^4)\psi(q^8)^2   \right) +4q^3\psi(q^4)\psi(q^8)^2\\
		&=  4q \sfc_{4,(0,1,1,2)}(q) + 2q \psi(q^4)\phi(q^4)^2+ 8q^3\psi(q^4)\psi(q^8)^2\\
		&= 4q \sfc_{4,(0,1,1,2)}(q) +2q\psi(q^4)\left(\phi(q^4)^2+4q^2\psi(q^4)^2 \right) \\
		&= 4q \sfc_{4,(0,1,1,2)}(q) +2q \psi(q^4)\phi(q^2)^2
\end{align*}
where we applied \Cref{phipsi} in the last step.
We can now apply \Cref{sigma=2 gf cor},
and we obtain
$$4q\sfc_{4,(0,1,1,2)}(q) = 4q\psi(q^4)\phi(q^2)^2+8q^2\psi(q^4)^3.$$
We obtain the bisection of the series, which we may write in the following way by applying \Cref{Sylvie4alt}:
\begin{equation}\label{c4(0112)}
\sfc_{4,(0,1,1,2)}(q) =\psi(q^2)^2\phi(q^2)+ 2q \psi(q^4)^3 .
\end{equation}
Therefore, if $n$ is even, $c_{4,(0,1,1,2)}(n)>0$ by \Cref{pos phipsipsi}.
On the other hand, if $n$ is odd, then $c_{4,(0,1,1,2)}(n)\neq 0$ only if $n\equiv1\mod 4$ and
in fact by Gauss' Eureka theorem, $c_{4,(0,1,1,2)}(n)>0$ in this case.
In summary, $c_{4,(0,1,1,2)}(n)>0$ if and only if $n\not\equiv 3\mod 4$.

\subsection{Defect $0$ blocks of cyclotomic Hecke algebras}\label{sec_def0_Hecke}

In another direction, one can consider representation theory of cyclotomic Hecke algebras,
which are obtained as multi-parameter deformations of the group algebra of $G(\ell,1,n)$  \cite{BroueMalle1993,Ariki2002}.
In the semisimple case, its irreducible representations are labeled by $\ell$-partitions of size $n$,
and much of the partition combinatorics that applies to $S_n$ generalises.
The non-semisimple cases can be reduced to the case where the parameters are
$\xi,\xi^{s_1},\ldots, \xi^{s_\ell}$ where $\xi$ is a primitive $e$-th root of unity.
Writing $\bs=(s_1,\ldots,s_\ell)$, we can therefore consider the non-semisimple cyclotomic Hecke algebra associated with $(e,\bs)$, which we denote by $H_{(e,\bs)}(n)$.

\medskip

Now, this algebra being symmetric, it comes equipped with a defect statistic on the blocks.
Note that this is simpler than the defect statistic for blocks of finite groups as it is a block invariant \cite{ChlouverakiJacon2023}, where the authors also show that it coincides with 
an appropriate weight statistic \cite{Fayers2007}. 
We refer to \cite{LyleMathas2007,GeckJacon2011,JaconLecouvey2021,ChlouverakiJacon2023, DeclercqJacon2024} for more details on block theory of cyclotomic Hecke algebras.
Most interesting to us is the following analogue of the Nakayama conjecture in this context.

\begin{Thm}\cite{JaconLecouvey2021}
\label{naka_hecke}
Two irreducible representations labeled by $\ell$-partitions $\bla$ and $\bmu$ belong to the same block of $H_{(e,\bs)}(n)$ if and only if $|\bla,\bs\rangle_{(e)}=|\bmu, \bs\rangle_{(e)}$.
\end{Thm}

As a direct consequence of \Cref{naka_hecke} and \Cref{thm_gf_escores},
there is a defect $0$ block for $H_{(e,\bs)}(n)$ for all $n\in\N$ if 
$(\ell, e,\bs)\in \{ (2, 4, (0,1)), (2, 4, (0,2)) , (3, 6, (0,2))\}$.
Since \Cref{tab_escores} only treats small values of the parameters, we pose the question of extending
these results. In particular, it might be possible to study positivity (or equivalently, existence of defect $0$ blocks) without using product formulas for the generating functions (which seem quite complicated to establish in general).

\section*{Acknowledgements}

We thank Sylvie Corteel, Olivier Dudas, Abel Lacabanne, Gunter Malle and  Byeong-Kweon Oh for helpful remarks and communications. 
E. Norton thanks RepNet for travel support to work on this project.


\begin{thebibliography}{99}

\bibitem{Ariki2002}
\emph{Representations of quantum algebras and combinatorics of Young tableaux}, 
University Lecture Series, {\bf 26}, 
American Mathematical Society, Providence, RI, 2002.

\bibitem{BaruahBerndt}
N.~D. Baruah and B.~C. Berndt, 
\emph{Partition identities and Ramanujan's modular equations}, 
J. Combin. Theory Ser. A {\bf 114} (2007), no.~6, 1024--1045.

\bibitem{BDR17}
C. Bonnaf\'{e}, F. Dat and R. Rouquier, 
\emph{Derived categories and Deligne-Lusztig varieties II}, 
Ann. of Math. (2), {\bf 185}(2) (2017), 609--670.

\bibitem{BonnafeRouquier02}
C. Bonnaf\'e and R. Rouquier, 
\emph{Cat\'egories d\'eriv\'ees et vari\'et\'es de Deligne--Lusztig}, 
Publ. Math. Inst. Hautes \'Etudes Sci.  {\bf 97} (2003), 1--59. 

\bibitem{Brauer1944}
R. Brauer,
\emph{On the arithmetic in a group ring}, 
Proc. Nat. Acad. Sci. U.S.A. {\bf 30} (1944), 109--114.

\bibitem{Brauer1947}
R. Brauer,
\textit{On a conjecture by Nakayama},
Transactions of the Royal Society of Canada III (1947), no. 41,
11--19.

\bibitem{BroueMalle1993}
M. Brou\'{e} and G. Malle,
\emph{Zyklotomische Heckealgebren}, 
Ast\'{e}risque No. {\bf 212} (1993), 119–189.

\bibitem{BMM1999}
M. Brou\'{e}, G. Malle and J. Michel,
\emph{Towards Spetses I}
Transform. Groups, Vol. 4, No. 2-3 (1999), pp. 157--218.

\bibitem{CGJL}
N. Chapelier-Laget, T. Gerber, N. Jacon and C. Lecouvey,
\textit{Entropy of affine permutations and universality of affine atomic lengths},
arXiv:2603.22256, 2026.

\bibitem{CGJ2025}
M. Chlouveraki, J.-B. Gramain and N. Jacon,
\textit{Generalised hook lengths and Schur elements for Hecke algebras},
Algebr. Comb. {\bf 8} (2025), no. 4, 1069--1084.


\bibitem{ChlouverakiJacon2023}
M. Chlouveraki and N. Jacon,
\emph{Defect in cyclotomic Hecke algebras},
Math. Z. {\bf 305} (2023), no. 4, article 57.

\bibitem{ChlouverakiMalle2026}
M. Chlouveraki and G. Malle,
\emph{Intersections of blocks of cyclotomic Hecke algebras},
Math. Z. {\bf 312} (2026), Paper No. 99.

\bibitem{ConwaySloane1999}
J. H. Conway and N. J. A. Sloane,
\textit{Sphere packings, lattices and groups},
Grundlehren der mathematischen Wissenschaften {\bf 290},
Springer-Verlag, New York, 1999.

\bibitem{Craven2014}
D. Craven,
\emph{The Brauer trees of non-crystallographic groups of Lie type}, 
J. Alg. {\bf 398} (2014), 481--495.

\bibitem{Craven}
D. Craven,
\emph{Representation Theory of Finite Groups: a Guidebook},
Universitext, Springer Nature, Switzerland (2019).

\bibitem{DeclercqJacon2024}
D. Declercq and N. Jacon,
\emph{Blocks of Ariki-Koike algebras and level-rank duality},
arXiv:2409.19355.

\bibitem{DeligneLusztig}
P. Deligne and G. Lusztig, 
\emph{Representations of reductive groups over finite fields},
Ann. of Math. {\bf 103} (1976), 103--161.

\bibitem{DVV2}
O. Dudas, M. Varagnolo and \'E. Vasserot, 
\emph{Categorical actions on unipotent representations of finite classical groups}, 
in {\it Categorification and higher representation theory}, 41--104,  
Contemp. Math., {\bf 683}, Amer. Math. Soc., Providence, RI.


\bibitem{ErdmannMichler}
K. Erdmann and G.~O. Michler, 
\emph{Blocks for symmetric groups and their covering groups and quadratic forms},
 Beitr\"age Algebra Geom. {\bf 37} (1996), no.~1, 103--118.

\bibitem{Fayers2007}
M. Fayers, 
\emph{Core blocks of Ariki-Koike algebras}, 
J. Algebraic Combin. {\bf 26} (2007), no.~1, 47--81.


\bibitem{Fayers2019}
M. Fayers,
\emph{Simultaneous core multipartitions},
European J. Combin. {\bf 76} (2019), 138--158.

\bibitem{FongSrinivasan1982}
P. Fong and B. Srinivasan,
The blocks of finite general linear and unitary groups,
Invent. Math. {\bf 69} (1982), 109--154.

\bibitem{FongSrinivasan1989}
P. Fong and B. Srinivasan,
\textit{The blocks of finite classical groups}, 
J. reine angew. Math. {\bf 396} (1989), 122--191.




\bibitem{GKS1990}
F. G. Garvan, D. Kim and D. Stanton,
\textit{Cranks and $t$-cores},
Invent. Math. {\bf 101} (1990), 1--17.


\bibitem{Chevie}
M. Geck, G. Hiss, F. Lübeck, G. Malle and G. Pfeiffer, 
\emph{CHEVIE -- A system for computing and processing generic character tables for finite groups of Lie type, Weyl groups and Hecke algebras}, 
Appl. Algebra Engrg. Comm. Comput., 7 (1996), pages 175--210.

\bibitem{GeckJacon2011}
M. Geck and N. Jacon, 
{\em Representations of Hecke algebras at roots of unity}, 
Algebra and Applications, {\bf 15}, Springer-Verlag
London, Ltd., London, 2011.

\bibitem{GO1996}
A. Granville and K. Ono,
\emph{Defect zero $p$-blocks for finite simple groups},
Trans. Amer. Math. Soc. {\bf 348} (1996), 331--347.


\bibitem{GruberHiss}
J. Gruber and G. Hiss, 
\emph{Decomposition numbers of finite classical groups for linear primes},
 J. Reine Angew. Math. {\bf 485} (1997), 55--91.

\bibitem{Hirschhorn2008}
M. D. Hirschhorn,
\textit{A letter from Fitzroy House},
Amer. Math. Monthly {\bf 115} (2008), no. 6, 563--566.


\bibitem{Jacon2024}
N. Jacon,
\textit{Size of multipartitions through level-rank duality},
Preprint, 2024.

\bibitem{JaconLecouvey2020}
N. Jacon and C. Lecouvey, 
\emph{Keys and Demazure crystals for Kac-Moody algebras},
 J. Comb. Algebra {\bf 4} (2020), no.~4, 325--358.


\bibitem{JaconLecouvey2021}
N. Jacon and C. Lecouvey,
\emph{Cores of Ariki--Koike algebras},
Doc. Math. {\bf 26} (2021), 103--124.


\bibitem{KassReut2018}
C. Kassel and C. Reutenauer, 
\emph{Complete determination of the zeta function of the Hilbert scheme of $n$ points on a two-dimensional torus},
 Ramanujan J. {\bf 46} (2018), no.~3, 633--655.

\bibitem{KMS2024}
R. Kessar, G. Malle and J. Semeraro,
\emph{Weight conjectures for l-compact groups and spetses}, 
Ann. Sci. \'{E}cole Norm. Sup. (4) {\bf 57} (2024), 841--896.

\bibitem{Kiming}
I. Kiming, 
\emph{A note on a theorem of A. Granville and K. Ono},
 J. Number Theory {\bf 60} (1996), no.~1, 97--102

\bibitem{Lit1951}
D. E. Littlewood,
\emph{Modular representations of symmetric groups},
Proceedings of the Royal Society of London A {\bf 209} (1951), no. 1098,
333--353.


\bibitem{Lusztig1977}
G. Lusztig, 
\emph{Irreducible representations of finite classical groups}, 
Invent. Math. {\bf 43} (1977), no.~2, 125--175.


\bibitem{Lyle2024}
S. Lyle, 
\emph{Core blocks for Hecke algebras of type $B$ and sign sequences}, 
Comm. Algebra {\bf 52} (2024), no.~5, 1965--1981.

\bibitem{LyleMathas2007}
S. Lyle and A. Mathas,
\emph{Blocks of affine and cyclotomic Hecke algebras}, 
Adv. Math. {\bf 216}, no. 2, 854--878.

\bibitem{Malle1995}
G. Malle,
\emph{Unipotente Grade imprimitiver komplexer Spiegelungsgruppen},
J. Alg. {\bf 177} (1995), 768--826.

\bibitem{Malle1998}
G. Malle,
\emph{Spetses},
Doc. Math. Extra Vol. ICM II (1998), 87--96.

\bibitem{Malle2026}
G. Malle,
\emph{Harish-Chandra theories, Ennola $d$-ality and Rouquier blocks for spetses},
arXiv:2607.00515.
 
\bibitem{Michler1986}
G. O. Michler,
\emph{A finite simple group of Lie type has $p$-blocks with different
defects if $p\neq2$},
J. Alg. {\bf 104} (1986), 220--230.


\bibitem{oeis}
OEIS Foundation Inc.,
\textit{The On-Line Encyclopedia of Integer Sequences},
\url{https://oeis.org}, 2026.

\bibitem{OhSun2009}
B.-K. Oh and Z.-W. Sun,
\emph{Mixed sums of squares and triangular numbers (III)},
J. Number Theory  {\bf 129} (2009), no. 4, 964--969.


\bibitem{Olsson1986}
J. B. Olsson,
\emph{Remarks on symbols, hooks and degrees of unipotent characters},
J. Combin. Theory Ser. A {\bf 42} (1986), 223--238.


\bibitem{Ono1994}
K. Ono,
\emph{On the positivity of the number of $t$-core partitions},
Acta Arith. {\bf 66} (1994), no. 3, 221--228.

\bibitem{PW2018}
L. Pehlivan and K. Williams,
\emph{$(k,l)$-universality of ternary quadratic forms
$ax^2+by^2+cz^2$},
Integers {\bf 18} (2018), A20.



\bibitem{Robbins}
N. Robbins, 
\emph{On $t$-core partitions}, 
Fibonacci Quart. {\bf 38} (2000), no.~1, 39--48.


\bibitem{Robinson1947}
G. de B. Robinson,
\emph{On a conjecture by Nakayama},
Transactions of the Royal Society of Canada III (1947), no. 41,
20--25.

\bibitem{Ruhstorfer20}
L. Ruhstorfer,
\emph{On the Bonnaf\'{e}-Dat-Rouquier Morita equivalence}, 
J. Alg. {\bf 558} (2020), 660--676.

\bibitem{ShephardTodd1954}
G. C. Shephard and J. A. Todd, 
\emph{Finite unitary reflection groups}, 
Canadian J. Math., {\bf 6} (1954), 274--304.

\bibitem{TrinhXue2023}
M.-T. Q. Trinh and T. Xue,
\emph{Level-rank dualities from $\Phi$-Harish-Chandra series and affine Springer fibers},
arXiv:2311.17106 (2023).



\bibitem{Uglov1999}
D. Uglov,
\emph{Canonical bases of higher level $q$-deformed Fock spaces and Kazhdan--Lusztig polynomials},
in \emph{Physical Combinatorics} (M.~Kashiwara and T.~Miwa, eds.),
Progress in Mathematics, vol.~191,
Birkh{\"a}user, Basel, 2000.





  \bibitem{Wald04}
 J.-L. Waldspurger,  
 \emph{Une conjecture de Lusztig pour les groupes classiques}, 
 M\'em. SMF {\bf 96} (2004).




\bibitem{Willems1988}
W. Willems,
\emph{Blocks of defect zero in finite simple groups of Lie type},
J. Alg. {\bf 113} (1988), no. 2, 511--522.


\end{thebibliography}
\end{document}